\documentclass{article}

\usepackage{amsmath,amssymb,amsthm,amsfonts,mathrsfs}
\usepackage[pdftex]{graphicx,epsfig,color}
\usepackage{verbatim}
\def\R{\mathbb R}

\def\N{\mathbb N}

\numberwithin{equation}{section}

\newtheorem{theorem}{Theorem}

\newtheorem{lemma}[theorem]{Lemma}
\newtheorem{proposition}[theorem]{Proposition}
\newtheorem{corollary}[theorem]{Corollary}
\newtheorem{definition}[theorem]{Definition\rm}
\newtheorem{remark}[theorem]{Remark}
\newtheorem{claim}[theorem]{Claim}

\numberwithin{theorem}{section}

\title{The Sard conjecture on Martinet surfaces}

\author{A.~Belotto\thanks{University of Toronto, Department of Mathematics, 40 St. George Street,
Toronto, ON, Canada M5S 2E4 (andrebelotto@gmail.com)}
\and
L.~Rifford\thanks{Universit\'e Nice Sophia
    Antipolis, Labo.\ J.-A.\ Dieudonn\'e, UMR CNRS 6621, Parc
    Valrose, 06108 Nice Cedex 02, France \& Institut Universitaire de France ({\tt
      ludovic.rifford@math.cnrs.fr})}}

\date{}

\makeindex

\begin{document}

\maketitle

\begin{abstract}
Given a totally nonholonomic distribution of rank two on a three-dimensional manifold we investigate the size of the set of points that can be reached by singular horizontal paths starting from a same point. In this setting, the Sard conjecture states that that set should be a subset of the so-called Martinet surface of 2-dimensional Hausdorff measure zero. We prove that the conjecture holds in the case where  the Martinet surface is smooth. Moreover, we address the case of singular real-analytic Martinet surfaces and show that the result holds true under an assumption of non-transversality of the distribution   on the singular set of the Martinet surface. Our methods rely on the control of the divergence of vector fields generating the trace of the distribution on the Martinet surface  and some techniques of resolution of singularities.

\end{abstract}


\section{Introduction}\label{Intro}

Let $M$ be a smooth connected manifold of dimension $n\geq 3$ and let $\Delta$ be a totally nonholonomic distribution of rank $m <n$ on $M$, that is a smooth subbundle of $TM$ of dimension $m$ such that for every $x\in M$ there is an open neighborhood $\mathcal{V}$ of $x$ where $\Delta$ is locally parametrized by $m$ linearly independent smooth vector fields $X_x^1, \ldots, X_x^m$ satisfying the so-called H\"ormander or bracket generating condition
\begin{eqnarray*}
\mbox{Lie} \Bigl\{ X_x^1, \cdots ,X_x^m \Bigr\} (y) = T_yM \qquad \forall y\in \mathcal{V}.
\end{eqnarray*}
An absolutely continuous curve $\gamma : [0,1] \rightarrow M$ is called horizontal with respect to $\Delta$  if it satisfies
\begin{eqnarray*}
\dot{\gamma}(t) \in \Delta (\gamma(t)) \qquad  \mbox{for a.e. } t \in [0,1].
\end{eqnarray*}
By Chow-Rashevsky's Theorem, total nonholonomicity plus connectedness implies horizontal path-connectedness. In other words,  for any two points $x,y \in M$ there is an horizontal path $\gamma : [0,1] \rightarrow M$ such that $\gamma(0)=x$ and $\gamma(1)=y$. Given $x\in M$, the set $\Omega_{\Delta}^{x}$ of horizontal paths starting from $x$ whose derivative is square integrable (with respect to a given metric on $M$) can be shown to enjoy the structure of a Hilbert manifold. However, in general given $x, y \in M$ the set $\Omega_{\Delta}^{x,y}$ of paths in  $\Omega_{\Delta}^{x}$ which join $x$ to $y$ fails to be a submanifold of  $\Omega_{\Delta}^{x}$ globally, it may have singularities. It happens to be a submanifold only in neighborhoods of horizontal paths which are not singular. The Sard conjecture for totally nonholonomic distributions is concerned with the size of the set of points that can be reached by those singular paths in $\Omega_{\Delta}^{x}$. It is related to some of the major open problems in sub-Riemannian geometry, see \cite{agrachev14,montgomery02,riffordbook,riffordbourbaki,rt05}.\\

The aim of the present paper is to solve partially this conjecture in the case of rank-two distributions in dimension three. Before stating precisely our result, we wish to define rigorously the notion of singular horizontal path.  For further details on  the material presented here, we refer the reader to Bella\"iche's monograph \cite{bellaiche96}, or to the books by  Montgomery \cite{montgomery02}, by Agrachev, Baralilari and Boscain \cite{abb}, or by the second author \cite{riffordbook}. For specific discussions about the Sard conjecture and  sub-Riemannian geometry, we suggest \cite{agrachev14}, \cite[Chapter 10]{montgomery02} and \cite{riffordbook,riffordbourbaki}.\\

To introduce the notion of singular horizontal path, it is convenient to identify the horizontal paths with the trajectories of a control system. 
It can be shown that there is a finite family $\mathcal{F} = \{X^1, \ldots, X^k\}$ (with $m\leq k \leq m(n+1)$) of smooth vector fields on $M$ such that 
\begin{eqnarray*}
\Delta (x) = \mbox{Span} \Bigl\{ X^1(x), \ldots, X^k(x) \Bigr\}  \qquad \forall x \in M. 
\end{eqnarray*}
For every $x\in M$, there is a non-empty maximal open set $\mathcal{U}^x \subset L^2([0,1], \R^k)$ such that for every control $u=(u_1, \cdots, u_k) \in \mathcal{U}^x$, the solution $x(\cdot; x,u) : [0,1] \rightarrow M$ 
to the Cauchy problem
\begin{eqnarray}\label{Cauchy}
\label{system}
\dot{x}(t) = \sum_{i=1}^k u_i(t) X^i(x(t)) \quad \mbox{for a.e. } t \in [0,1] \quad \mbox{and} \quad  x(0)=x
\end{eqnarray}
is well-defined. By construction,  for every $x\in M$ and every control $u\in  \mathcal{U}^x$ the trajectory $x(\cdot; x,u) $ is an horizontal path in  $\Omega_{\Delta}^{x}$. Moreover the converse is true, any $\gamma \in \Omega_{\Delta}^{x}$ can be written as the solution of (\ref{Cauchy}) for some $u\in \mathcal{U}^x$. Of course, since in general the vector fields $X^1, \ldots, X^k$ are not linearly independent globally on $M$, the control $u$ such that $\gamma=x (\cdot; x,u)$ is not necessarily unique. For every point $x\in M$, the End-Point Mapping from $x$ (associated with $\mathcal{F}$ in time $1$) is defined as 
$$ 
 \begin{array}{rcl}
\mbox{E}^{x} :   \mathcal{U}^{x} & \longrightarrow & M \\
u  & \longmapsto &  x(1;x,u).
\end{array}
$$
It shares the same regularity as the vector fields $X^1, \ldots, X^k$, it is of class $C^{\infty}$.   Given $x\in M$, a control $u \in \mathcal{U}^{x}\subset  L^2([0,1], \R^k)$ is said to be singular (with respect to $x$)  if the linear mapping 
$$
D_uE^{x} \, : \, L^2 \left([0,1],\R^k\right) \, \longrightarrow T_{E^{x}(u)} M
$$
 is not onto, that is if $E^{x}$ is not a submersion at $u$. Then, we call an horizontal path $\gamma \in \Omega_{\Delta}^{x}$ singular if and only if $\gamma = x (\cdot; x,u)$  for some $u \in \mathcal{U}^x$ which is singular (with respect to $x$).  It is worth noting that actually the property of singularity of an horizontal path does depend only on $\Delta$, it is independent of the choice of $X^1, \ldots, X^k$ and of the control $u$ which is chosen to parametrize the path. For every $x\in M$, we denote by $\mathcal{S}^{x}$ the set of controls $u \in \mathcal{U}^{x} \subset  L^2([0,1], \R^k)$ which are singular with respect to $x$. As we said, the Sard conjecture is concerned with the size of the set  of end-points of singular horizontal paths in $ \Omega_{\Delta}^{x}$ given by
$$
\mathcal{X}^x := E^{x} \left(  \mathcal{S}^{x}\right) \subset M.
$$
The set $\mathcal{X}^x$ is defined as the set of critical values of the smooth mapping $E^{x}$, so according to Sard's theorem we may expect it to have Lebesgue measure zero in $M$ which is exactly the statement of the Sard conjecture. Unfortunately, Sard's theorem is known to fail in infinite dimension (see \cite{bm01}), so we cannot prove by "abstract nonsense" that $\mathcal{X}^x$ has measure zero. In fact, singular horizontal paths can be characterized as the projections of the so-called abnormal extremals, which allows, in some cases, to describe the set of singular horizontal paths as the set of orbits of some vector field in $M$. The first case of interest is the case of rank two totally nonholonomic distributions in dimension three whose study is the purpose of the present paper.\\

If $M$ has dimension three and $\Delta$ rank two, it can be shown that the singular horizontal paths are those horizontal paths which are contained in the so-called Martinet surface (see Proposition \ref{PROPsingmartinet})
$$
\Sigma := \Bigl\{ x\in M \, \vert \,  \Delta(x) + [\Delta,\Delta](x) \neq T_xM \Bigr\},
$$
where $[\Delta,\Delta]$ is the (possibly singular) distribution defined by
$$
[\Delta, \Delta] (x) := \Bigl\{ [X,Y](x) \, \vert \, X,Y \mbox{ smooth sections of } \Delta \Bigr\} \qquad \forall x \in M.
$$
Moreover, by total nonholonomicity of the distribution, the set $\Sigma$ can be covered by a countable union of smooth submanifolds of codimension at least one. Consequently for every $x\in M$, the set $\mathcal{X}^x$ is always contained in $\Sigma$ which has zero Lebesgue measure zero, so that the Sard conjecture as stated above holds true for rank two (totally nonholonomic) distributions in dimension three. In fact, since for this specific case singular horizontal paths are valued in a two-dimensional subset of $M$ the Sard conjecture for rank-two distributions in dimension three is stronger and asserts that all the sets $\mathcal{X}^x$ have vanishing $2$-dimensional Hausdorff measure. The validity of this conjecture is supported by a major contribution in the nineties made by Zelenko and Zhitomirskii who proved that for generic rank-two distributions (with respect to the Whitney topology) all the sets $\mathcal{X}^x$ have indeed Hausdorff dimension at most one, see \cite{zz95}. Since then, no notable progress has been made. The purpose of the present paper is to attack the non-generic case. Our first result is concerned with distributions for which the Martinet surface is smooth. 

\begin{theorem}\label{THMdim3}
Let $M$ be a smooth manifold of dimension $3$ and $\Delta$ a rank-two totally nonholonomic distribution on $M$ whose Martinet surface $\Sigma$ is smooth. Then for every $x\in M$ the set $\mathcal{X}^x$ has $2$-dimensional Hausdorff measure zero.
\end{theorem}

The key idea of the proof of Theorem \ref{THMdim3} is to observe that the divergence of the vector field which generates the trace of the distribution on $\Sigma$ is controlled by its norm (see Lemma \ref{LEMlocal1}). To our knowledge, such an observation has never been made nor used before. This idea plays also a major role for our second result which is concerned with the real-analytic case.\\

Let us now assume that both $M$ and $\Delta$ are real-analytic with $\Delta$ of rank two and $M$ of dimension three. By total nonholonomicity, the set $\Sigma$ is a closed analytic set in $M$ of dimension $\leq 2$ and for every $x\in \Sigma$ there is an open neighborbood $\mathcal{V}$ of $x$ and a non-zero analytic function $h:\mathcal{V} \rightarrow \R$ such that $\Sigma \cap \mathcal{V} = \{h=0\}$. In general, this analytic set admits singularities, that is some points in a neighborhood of which $\Sigma$  is not diffeomorphic to a smooth surface.  To prove our second result, we will show that techniques from resolution of singularities allow to recover what is needed to apply the ideas that govern the proof of Theorem \ref{THMdim3}.  Since resolution of singularity is an algebraic process which applies to spaces that include a functional structure and not to sets, we need, before stating Theorem \ref{THMdim3sing}, to introduce a few notions from analytic geometry, we refer the reader to \cite{h73,n66,t71} for more details.  We will consider analytic spaces  $(X,\mathcal{O}_X:=\mathcal{O}_M / \mathcal{I})$ where $X $ is an analytic set in $M$ and $\mathcal{I}$ is a principal reduced and coherent ideal sheaf with support $X$, which means that $X $ admits  a locally finite covering by open sets $(U_{\alpha})_{\alpha \in \mathcal{A}}$ and that there is a family of analytic functions $(h_{\alpha})_{\alpha \in \mathcal{A}}: U_{\alpha} \to \mathbb{R}$  such that the following properties are satisfied:
\begin{itemize}
\item[(i)] For any $\alpha, \beta \in \mathcal{A}$, there exists an analytic function $\xi: U_{\alpha} \cap U_{\beta} \to \mathbb{R}$ such that $\xi(x)\neq 0 $ and $h_{\alpha}(x) = \xi(y) h_{\beta}(x)$ for all $x\in U_{\alpha} \cap U_{\beta}$.
\item[(ii)] For every $\alpha \in \mathcal{A}$, $X\cap U_{\alpha} = \{x \in U_{\alpha} \, \vert \, h_{\alpha}(x)=0\}$ and the set $\{x\in U_{\alpha}\, \vert \, h_{\alpha}(x) = 0 \text{ and } d_x h_{\alpha}=0\}$ is a set  of codimension at least two  in $U_{\alpha}$.
\end{itemize}
The set of singularities or singular set of an analytic space $(X,\mathcal{O}_X)$, denoted by $\mbox{Sing} ((X,\mathcal{O}_X)$,   is defined as  the union of the sets $\{x\in U_{\alpha} \, \vert \,  h_{\alpha}(x) = 0 \text{ and } d_x h_{\alpha}=0\}$ for $\alpha \in \mathcal{A}$. By the properties (i)-(ii) above, $\mbox{Sing} (X,\mathcal{O}_X)$ is an analytic subset of $M$ of codimension at least two, so it can be stratified by strata  $\Gamma_0, \Gamma_1$ respectively of dimension zero and one, where $\Gamma_0$ is a locally finite union of points and $\Gamma_1$ is a locally finite union of analytic submanifolds of $M$ of dimension one. Then, for every $x \in \mbox{Sing} (X,\mathcal{O}_X)$ we define the tangent space $T_x\mbox{Sing} (X,\mathcal{O}_X)$    to $\mbox{Sing} (X,\mathcal{O}_X)$ at $x$ as $\{0\}$ if $x$ belongs to $\Gamma_0$ and $T_x\Gamma_1$ if $x$ belongs to $\Gamma_1$. The proof of our second result is based on the resolution of singularities in the setting of coherent ideals and analytic spaces described above which was obtained by Hironaka \cite{h64,h73} (in fact, we will follow the modern proof of Hironaka's result given by Bierstone and Milman which includes a functorial property \cite{bm97,bm08} - see also \cite{kollarbook,w05} and references therein). In particular, it requires the Martinet surface $\Sigma$ to have the structure of a coherent analytic space. This fact is proven in Appendix \ref{secMartinetspace}, according to it we use from now the notation $\Sigma_{\Delta}=(\Sigma,\mathcal{O}_{\Sigma})$ to refer to the analytic space defined by the Martinet surface. We are now ready to state our second result.

\begin{theorem}\label{THMdim3sing}
Let $M$ be an analytic manifold of dimension $3$ and $\Delta$ a rank-two totally nonholonomic analytic distribution on $M$, assume that 
\begin{eqnarray}\label{ASSTHM2}
\Delta(x) \cap T_x\mbox{Sing}(\Sigma_{\Delta}) = T_x\mbox{Sing}(\Sigma_{\Delta}) \qquad \forall x \in  \mbox{Sing}(\Sigma_{\Delta}).
\end{eqnarray}
Then for every $x\in M$ the set $\mathcal{X}^x$ has $2$-dimensional Hausdorff measure zero.
\end{theorem}

The assumption (\ref{ASSTHM2})  means that at each singularity of $\Sigma_{\Delta}$ the distribution $\Delta$ generates the tangent space to $\mbox{Sing}(\Sigma_{\Delta})$. It is trivially satisfied  in the case where $\Sigma_{\Delta}$ has only  isolated singularities. 

\begin{corollary}
Let $M$ be an analytic manifold of dimension $3$ and $\Delta$ a rank-two totally nonholonomic analytic distribution on $M$, assume that $\Sigma_{\Delta}$ has only isolated singularities, that is $\Gamma^1=\emptyset$. Then for every $x\in M$ the set $\mathcal{X}^x$ has $2$-dimensional Hausdorff measure zero.
\end{corollary}

If the assumption (\ref{ASSTHM2}) in Theorem \ref{THMdim3sing} is not satisfied, that is if there are points $x$ in $\mbox{Sing}(\Sigma_{\Delta})$ where $\Delta(x)$ is transverse to  $T_x\mbox{Sing}(\Sigma_{\Delta})$, then our approach leads to the study of possible concatenations of homoclinic orbits of a smooth vector field defined on $\Sigma$ which vanishes on $\mbox{Sing}(\Sigma_{\Delta})$.  Let us illustrate what may happen by treating an example. In  $\R^3$ let us consider the totally nonholonomic analytic distribution $\Delta$ spanned by the two vector fields
$$
X=\partial_y  \quad \mbox{and} \quad Y = \partial_x + \left[ \frac{y^3}{3}-x^2 y (x+z)\right] \, \partial_z.
$$
We check easily that $[X,Y] = \left[ y^2 -x^2 (x+z)\right] \, \partial_z$, so that the Martinet surface is given by
$$
\Sigma = \Bigl\{ y^2-x^2(x+z)=0\Bigr\}.
$$

\begin{figure}\label{fig1}
\begin{center}
\includegraphics[width=7cm]{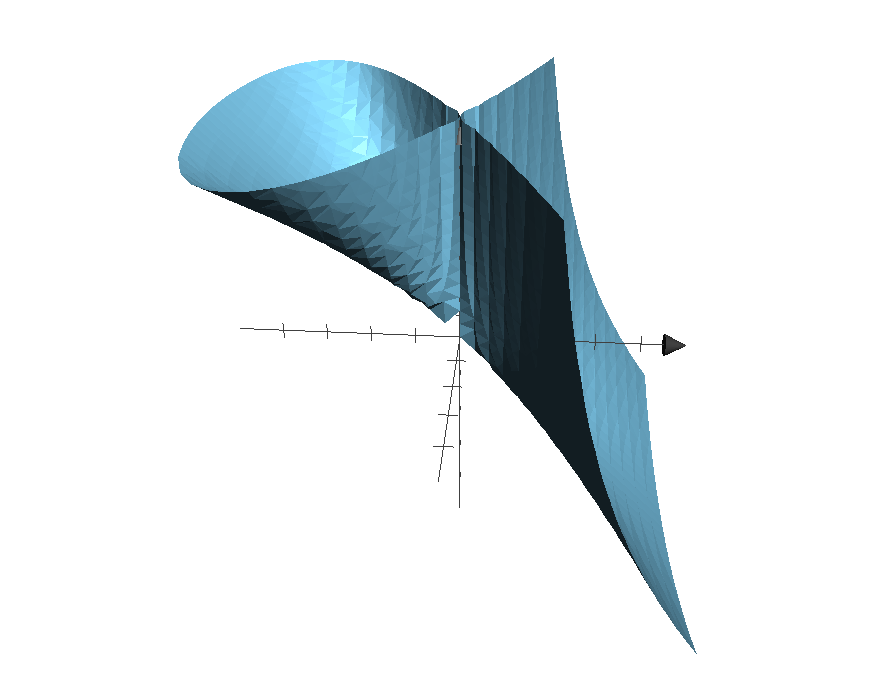}
\caption{The Martinet surface $\Sigma$}
\end{center}
\end{figure}

The Martinet surface $\Sigma$ is an analytic set whose singular set $\mbox{Sing}(\Sigma_{\Delta})$ is the vertical axis $x=y=0$, see Figure \ref{fig1}. We observe that the distribution $\Delta$ is transverse to the $1$-dimensional subset $\mathcal{S}:=\{x=y=0, z>0\}$ of $\mbox{Sing}(\Sigma_{\Delta})$, so the assumption of Theorem \ref{THMdim3sing} is not satisfied. Let $\mathcal{L}$ be the loop part  of $\Sigma$, that is the set of points in $\Sigma$ of the form $(x,y,z)$ with $x\leq 0$, for each $P=(x_P,y_P,z_P) \in \mathcal{S}$ we can construct an horizontal path $\gamma_P : [0,1] \rightarrow \mathcal{L}$ such that $\gamma_P(0)=P$ and $\gamma_P(1)=Q=(x_Q,y_Q,z_Q)$ belongs to $\mathcal{S}$ with $z_Q<z_P$. Let us show how to proceed. First, we notice that  the trace of $\Delta$ on $\Sigma$ outside its singular set is generated by the smooth vector field (we set $h(x,y,z)=y^2-x^2(x+z)$)
$$
\mathcal{Z} = \left( X^1\cdot h\right) \, X^2 - \left( X^2\cdot h\right) \, X^1= 2y \,\partial_x + \left[ 3x^2 + 2x(x+z) \right] \, \partial_y - \frac{4y^4}{3} \, \partial_z,
$$
which  is collinear (on $\Sigma$) to the vector field
$$
\mathcal{X} =  -2xy \,\partial_x - \left[ 3x^3 + 2x^2(x+z) \right] \, \partial_y + \frac{4xy^4}{3} \, \partial_z = -2xy \,\partial_x - \left[ 3x^3 + 2y^2 \right] \, \partial_y + \frac{4xy^4}{3} \, \partial_z.
$$
 If we forget about the $z$-coordinate, the projection $\hat{\mathcal{X}}$ of $\mathcal{X}$ onto the plane $(x,y)$ (whose phase portrait is drawn in Figure 2) has a unique equilibrium at the origin, has the vertical axis as invariant set where it coincides with $-2y^2 \, \partial_y$, and it is equal to $-3x^3 \, \partial_y$ on the horizontal axis. Moreover, $\hat{\mathcal{X}}$  enjoyes a property of symmetry with respect to the horizontal axis, there holds $\hat{X}(x,-y)=-X^1(x,y) \, \partial_x + X^2(x,y) \, \partial_y$ where $\hat{X}=X^1 \, \partial_x + X^2\, \partial_y$. Consequently, the quadrant $x\leq 0, y>0$ is invariant with respect to $\hat{\mathcal{X}}$ and  any trajectory $P(t)=(x(t),y(t),z(t))$ of $\mathcal{X}$ starting from a point $\bar{P}=P(0)$ of the form $(\bar{x},0,\bar{z})$ with $\bar{x}=-\bar{z}<0$  satisfies

\begin{figure}\label{fig2}
\begin{center}
\includegraphics[width=7cm]{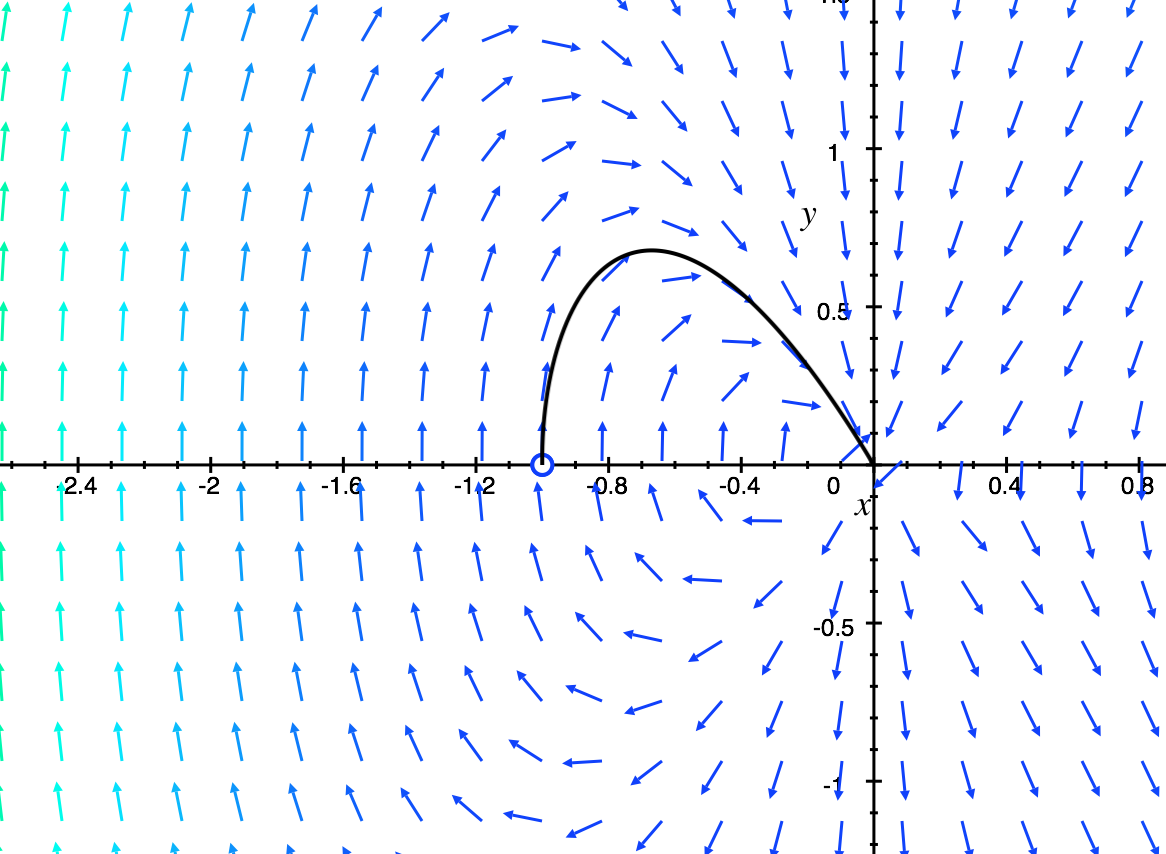}
\caption{The phase portrait of $\hat{\mathcal{X}}$}
\end{center}
\end{figure}

$$
x(t) < 0, \quad y(t) > 0, \quad \dot{x}(t) > 0, \quad \dot{z}(t) < 0 \quad \mbox{and} \quad P(t) \in \Sigma \qquad \forall t >0.
$$
Moreover, we can check that the curvature of its projection $\hat{P}([0,+\infty)]$ has the sign of $\dot{x}\ddot{y}-\dot{y}\ddot{x}$ which is equal to  $18x^7-24x^4y^2$ so it has a constant sign. In conclusion, the curve $\hat{P}(t)$ converges to the origin as $t$ tends to $\pm\infty$, it is symmetric with respect to the horizontal axis,  and  it surrounds a convex surface which is contained in the rectangle $[\bar{x},0] \times [-|\bar{x}|^{3/2},|\bar{x}|^{3/2}]$ (because $|y|=|x| \sqrt{x+z} \leq |x| \sqrt{z}$ which is $\leq |\bar{x}|^{3/2}$ on $[0,+\infty$). Then $\hat{P}: \R \rightarrow \R^2$ is homoclinic and its length $\ell(\hat{P})$ satisfies 
$$
2 \left| \bar{x} \right| \leq \ell (\hat{P}) := \int_{-\infty}^{+\infty} \left| \dot{\hat{P}}(t) \right| \, dt \leq 2 \left| \bar{x} \right| + 4  \left| \bar{x} \right|^{3/2}. 
$$
Moreover, we check easily that (see Appendix \ref{appendixcomputations})
\begin{eqnarray}\label{18july1}
 - \frac{2}{3} z(-\infty)^{9/2} \, \ell ( \hat{P} ) \leq z(+\infty)-z(-\infty) \leq -\frac{z(+\infty)^{11/2}}{35}, 
 \end{eqnarray}
\begin{eqnarray}\label{18july2}
z(-\infty) \leq \ell (\hat{P}) \left[ \frac{1}{2} + \frac{|z(-\infty)|^{9/2}}{3} \right],
\end{eqnarray}
\begin{eqnarray}\label{18july3}
\ell (\hat{P}) \leq \ell (P) :=\int_{-\infty}^{+\infty} \left| \dot{P}(t) \right| \, dt
\end{eqnarray}
and
\begin{eqnarray}\label{18july4}
\ell (P)  \leq 2 \, \left[ \left| z(-\infty) \right| + 2  \left| z(-\infty) \right|^{3/2}\right]  \, \left[ \sqrt{2} +   \frac{2}{3} |z(-\infty)|^{9/2} \right].
\end{eqnarray}
In conclusion, if we fix $z_0>0$, then there is $\bar{x}_0<0$ such that the orbit $P_0(t)=(x_0(t),y_0(t),z_0(t))$ starting from $(\bar{x}_0,0,-\bar{x}_0)$ at time $t=0$ satisfies $z_0 = z_0(-\infty)  :=\lim_{t \rightarrow -\infty} z_0(t)$, $z_0(+\infty) := \lim_{t \rightarrow -\infty} z_0(t)= z_1$ for some $z_1<z_0$, has length $\ell_0:=\ell (P)$ and the inequalities (\ref{18july1})-(\ref{18july4}) are satisfied with $z_0(+\infty), z_0(-\infty), \ell_0(\hat{P}_0), \ell_0(P_0)$. If we repeat this construction from $z_1$, then we get a decreasing sequence of positive real numbers $\{z_k\}_{k\in \N}$ together with  a sequence of lengths $\{\ell_k\}_{l\in \N}$ such that  for every $k\in \N$,
\begin{eqnarray}\label{18july5}
-\frac{2}{3} z_k^{9/2} \ell_k\leq z_{k+1}-z_k \leq -\frac{z_{k+1}^{11/2}}{35}, \quad z_k \leq K\, \ell_k, \quad \ell_k \leq K \, \left[ z_k + 2 z_k^{3/2}\right],
\end{eqnarray}
where $K>0$ stands for the maximum of $ 1/2+ |z_0|^{9/2}/3$ and $2\sqrt{2}+(4/3)z_0^{9/2}$. Moreover, the sequences $\{z_k\}_{k\in \N}, \{\ell_k\}_{l\in \N}$ are associated with a sequence of singular horizontal paths $\{\gamma_k\}_{k\in \N}$ of length $\ell_k$ which joins $z_k$ to $z_{k+1}$ for every $k$. Therefore, concatenating the paths $\gamma_1, \gamma_2, \ldots,$ we get a singular horizontal path $\gamma^{z_0}:[0,+\infty) \rightarrow \mathcal{L}$ which depends upon the starting point $z_0$ and which tends to $0$ as $t$ tends to $+\infty$, see Figure 3. The union of the all singular paths $\gamma^{z_0}$ obtained in this way with $z_0\in \mathcal{S}$ will fill the loop-part $\mathcal{L}$ of $\Sigma$, so we may think that the set of horizontal paths starting from the origin reach a set of positive $2$-dimensional Hausdorff measure. 

\begin{figure}\label{fig3}
\includegraphics[width=6cm]{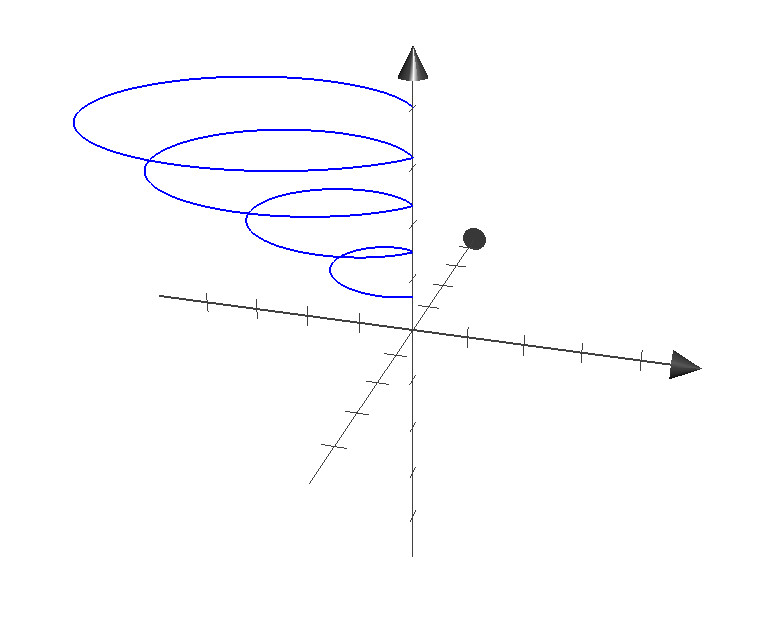} \hfill
\includegraphics[width=6cm]{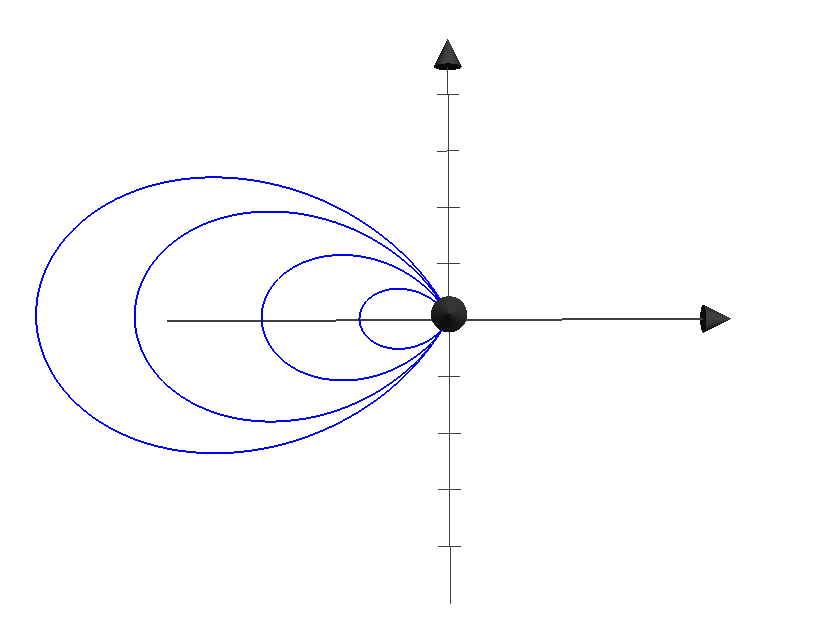}
\caption{A concatenation of homoclinic singular orbits converging to $0$ in the $3D$-space and from above}
\end{figure}

Hopefully, we can prove that all the paths constructed above have infinite length and so are not admissible because they do not allow to construct horizontal paths starting from the origin. To see this, assume that there are sequences $\{z_k\}_{k\in \N}, \{\ell_k\}_{l\in \N}$ satisfying (\ref{18july5}) such that $\sum_{k\in \N}\ell_k$ is finite.  By the second inequality in (\ref{18july5}) the sum $\sum_{k\in \N}z_k$ should be finite as well. But since
$$
z_{k+1}-z_k \geq -\frac{2}{3} z_k^{9/2} \ell_k \geq -\frac{2K}{3} z_k^{9/2} \, \left[ z_k + 2 z_k^{3/2}\right] \geq  -\frac{2K}{3} z_k^{11/2},
$$
for $k$ large enough, we have for for every integers $k, p$
$$
z_{k+p}-z_k \geq -\frac{2K}{3} z_k^{11/2}p.
$$
Which means that $z_{k+p} \leq z_k/2$ holds when $p\geq \frac{3}{4K} \frac{1}{z_k^{9/2}}=:p_k$ so that
$$
\sum_{j=k}^{k+p_k} z_j \geq p_k \, z_k \geq  \frac{3}{4K} \frac{1}{z_k^{7/2}} \longrightarrow_k \infty
$$
and contradicts the finiteness of $\sum_{k\in \N} z_k$.  As we can see, the estimates (\ref{18july5}) prevent the existence of horizontal paths obtained as concatenations of infinitely many homoclinic orbits. The availability of such estimates and more generally possible extensions of Theorem \ref{THMdim3sing} in absence of assumption  (\ref{ASSTHM2}) will be the subject of a subsequent paper. \\

The paper is organized as follows: The proof of Theorem \ref{THMdim3} is given in Section \ref{proofTMHdim3}.  The Section \ref{proofTHMdim3sing} is devoted to the proof of Theorem \ref{THMdim3sing} which is divided in two parts. The first part (Section \ref{SECfirstpart}) consists in extending  to the singular analytic case an intermediate result already used in the proof of Theorem \ref{THMdim3} and the second part  (Section \ref{SECsecondpart}) presents the result from resolution of singularities (Proposition \ref{PROPMAIN}) which allows to apply the arguments of divergence which are used in the proof of Theorem \ref{THMdim3}. The Section \ref{SECconical} in between those two sections, which  should be seen as a warm-up to Section \ref{SECsecondpart}, is dedicated to the study of the specific case of a conical singularity. There, we explain how a simple blowup yields Proposition \ref{PROPMAIN} for this case. The proof of Proposition \ref{PROPMAIN} is given in Section \ref{SECPROPMAIN} and all the material and results from resolution of singularities in manifolds with corner which are necessary for its proof are provided in Section \ref{secRS}. Finally, the appendix gathers the proofs of several results which are referred to in the body of the paper.\\

\noindent {\bf Acknowledgments.} The first author would like to thank Edward Bierstone for several useful discussions. The second author is grateful to Adam Parusinski for enlightening discussions. We would also like to thank the hospitality of the Fields Institute and the Universidad de Chile.

\section{Proof of Theorem \ref{THMdim3}}\label{proofTMHdim3}

Throughout this section we assume that $\Delta$ is a rank-two totally nonholonomic distribution on $M$ (of dimension three) whose Martinet surface $\Sigma$ is a smooth submanifold. First of all, we notice that, since for every $x\in M$ the set $\mathcal{X}^x$ is contained in $\Sigma$, the result of Theorem \ref{THMdim3} holds if $\Sigma$ has dimension one. So we  may assume from now on that $\Sigma$ has dimension two. Let us define two subsets $\Sigma_{tr}, \Sigma_{tan}$ of $\Sigma$ by
$$
\Sigma_{tr} := \Bigl\{ x\in \Sigma \, \vert \,  \Delta(x) \cap T_x\Sigma \mbox{ has dimension one}\Bigr\} \quad \mbox{and} \quad \Sigma_{tan} := \Bigl\{ x\in \Sigma \, \vert \,  \Delta(x) = T_x\Sigma \Bigr\}.
$$
By construction, $\Sigma_{tr}$ is an open subset of $\Sigma$  and $\Sigma_{tan} = \Sigma \setminus \Sigma_{tr}$ is closed. The following holds.

\begin{lemma}\label{LEM1rect}
The set  $\Sigma_{tan}$ is countably smoothly $1$-rectifiable, that is it can be covered by countably many submanifolds of dimension one.
\end{lemma} 

\begin{proof}[Proof of Lemma \ref{LEM1rect}]
We need to show that for every $x\in \Sigma_{tan}$ there is an open neighborhood $\mathcal{V}_x$ of $x$ in $\Sigma$ such that $ \Sigma_{tan} \cap \mathcal{V}_x$ is contained in a smooth submanifold of dimension one. Let $\bar{x}\in \Sigma_{tan}$ be fixed. Taking a sufficiently small open neighborhood $\mathcal{V}$ of $\bar{x}$ in $M$ and doing a change of coordinates if necessary we may assume that there is a set of coordinates $(x_1,x_2,x_3)$ such that $\bar{x}=0$ and $\Sigma \cap \mathcal{V} = \left\{ x_3=0\right\}$ and there are  two smooth vector fields $X^1, X^2$ on $\mathcal{V}$ of the form 
$$
X^1 = \partial_{x_1} + A^1 \,  \partial_{x_3} \quad \mbox{and} \quad X^2 = \partial_{x_2} + A^2 \, \partial_{x_3}
$$
where $A^1, A^2$ are smooth functions such that  $A^1(0)=A^2(0)=0$ and 
$$
\Delta(x) = \mbox{Span} \Bigl\{X^1(x), X^2(x)\Bigr\} \qquad \forall x\in \mathcal{V}.
$$
 Then we have $[X^1,X^2]= \left( A^2_{x_1} - A^1_{x_2} + A^1 A^2_{x_3}- A^2 A^1_{x_3} \right) \, \partial_{x_3}$ on $\mathcal{V}$ which yields ($f_{x_i}$ denotes the partial derivative of a smooth function $f$ with respect to the $x_i$ variable)
$$
\Sigma \cap \mathcal{V} = \Bigl\{ x_3=0\Bigr\} =  \Bigl\{ A^1_{x_2}-A^2_{x_1}+ A^1 A^2_{x_3}- A^2 A^1_{x_3}=0\Bigr\},
$$
and moreover a point $x \in \Sigma \cap \mathcal{V}$ belongs to $\Sigma_{tan}$ if and only if $A^1(x)=A^2(x)=0$. Let $\mathcal{O}$ be the set of $x \in \Sigma \cap \mathcal{V}$ for which $A^i_{x_j}(x) \neq 0$ for some pair $i,j $ in $\{1,2\}$.  The set  $\mathcal{O}$ is open in $\Sigma$ and by the Implicit Function Theorem the set $\Sigma_{tan} \cap \mathcal{O}$ can be covered by a finite union of smooth submanifolds of $\Sigma$ of dimension one. Therefore, it remains to show that the set 
$$
 \mathcal{F} := \left( \Sigma_{tan} \cap \mathcal{V} \right) \setminus \mathcal{O} = \Bigl\{ x \in \Sigma_{tan} \cap \mathcal{V} \, \vert \, A^1_{x_1}(x)= A^1_{x_2}(x)= A^2_{x_1}(x)= A^2_{x_2}(x)=0\Bigr\}
 $$
can be covered by countably many smooth submanifolds of dimension one. By total nonholonomicity, for every $x\in \Sigma \cap \mathcal{V}$ there are a least integer $r=r_x\geq 3$ and a $r$-tuple $(i_1, \ldots, i_r) \in \{1,2\}^r$ such that the vector field $Z^r=Z_1^r \partial_{x_1} + Z_2^r \partial_{x_2} + Z_3^r \partial_{x_3} $ defined by 
$$
Z^r :=\left[ \cdots \left[ [X^1,X^2], X^{i_3}], X^{i_4}], \cdots    , X^{i_r}     \right]\right.
$$
satisfies $Z_3^r(x)\neq 0$ and all the vector fields $Z^j= \sum_{i=1}^3 Z^j_i \partial_{x_i}$ for $j=2, \ldots, r-1$ defined by  
$$
Z^j:= \left[ \cdots \left[ [X^1,X^2], X^{i_3}], X^{i_4}], \cdots    , X^{i_{j}}     \right]\right.
$$
satisfy $Z^j_3(x)=0$. Then for every $x\in \mathcal{F}$ we have $Z_3^{r-1}(x)=0$ and
\begin{multline*}
0 \neq Z_3^r(x)  =  \left[  Z^{r-1}, X^{i_r}  \right] (x)\\
 =  A^{i_{r}}_{x_1} (x)\, Z_1^{r-1} (x) +  A^{i_{r}}_{x_2}(x) \, Z_2^{r-1} (x) + A^{i_{r}}_{x_3}(x) \, Z_3^{r-1} (x) \\
 - \left(Z_3^{r-1}\right)_{x_{i_r}} (x)- \left(Z_3^{r-1}\right)_{x_{3}} (x) \, A^{i_r}(x) =  - \left(Z_3^{r-1}\right)_{x_{i_r}} (x).
\end{multline*}
Consequently, taking $\mathcal{V}$ smaller if necessary (in order to use that $X^1, X^2$ are well-defined in a neighborhood of $\bar{\mathcal{V}}$) and using compactness, there is a finite number of smooth functions  $z_1, \ldots, z_N : \mathcal{V} \rightarrow \R$ such that for every $x \in \mathcal{F}$ there are $l(x) \in \{1, \ldots, N\}$ and $i(x) \in \{1,2\}$ such that $z_{l(x)}=0$ and $( z_{l(x)})_{x_i(x)} \neq 0$. We conclude easily by the Implicit Function Theorem.
\end{proof}

Define the singular distribution, that is a distribution with non-constant rank,  $L_{\Delta}$  by
$$
L_{\Delta}(x) = \left\{ \begin{array}{ccl} \{0\} & \mbox{ if }  & x \in \Sigma_{tan} \\
\Delta(x) \cap T_x\Sigma            & \mbox{ if } & x \in \Sigma_{tr} 
\end{array}
\right.
\qquad \forall x \in \Sigma.
$$
We observe that by Proposition \ref{PROPsingmartinet}, for every non-constant singular horizontal path $\gamma:[0,1] \rightarrow M$ there is a set $I$ which is open in $[0,1]$ such that $\gamma(t)\in \Sigma_{tr}$ for any $t\in I$ and $\gamma(t) \in \Sigma_{tan}$ for any $t\in [0,1] \setminus I$, and moreover, for almost every $t\in I$, $\dot{\gamma}(t) \in L_{\Delta} (\gamma(t))$. By smoothness of $L_{\Delta}$ on $\Sigma_{tr}$, the open set  $\Sigma_{tr}$ can be foliated by the orbits of  $L_{\Delta}$. Let us fix a Riemannian metric $g$ on  $M$ whose geodesic distance is denoted $d^g$. Through each point $z$ of $\Sigma_{tr}$ there is a maximal orbit $\mathcal{O}_z$, that is a one-dimensional smooth submanifold of $\Sigma_{tr}$ which is either compact and diffeomorphic to a circle, or open and parametrized by a smooth curve parametrized by arc-length (with respect to the metric $g$) of the form
$$
\begin{matrix}
\gamma : (-\alpha,\beta) \rightarrow \Sigma_{tr} \quad \mbox{with} \quad \gamma(0)=z, \, 
\end{matrix}
$$
where $\alpha, \beta$ belong to $[0,\infty]$.  Note that every open orbit admits two parametrizations as above. For every $z\in \Sigma_{tr}$ whose orbit $\mathcal{O}_z$ is open, we call half-orbit of $z$, denoted by $\omega_z$,  any of the two subsets of $\mathcal{O}_z$ given by $\gamma((-\alpha,0])$ or $\gamma ([0,\beta))$. This pair of sets does not depend on the parametrization $\gamma$ of $\mathcal{O}_z$. We observe that if some half-orbit $\omega_z$ of the form $\gamma((-\alpha,0])$ is part of a singular horizontal path between two points, then it must have finite length, or equivalently finite $\mathcal{H}^1$ measure, and in consequence $\alpha$ is finite and  $\gamma$ necessarily has a limit  as $t$ tends to $\alpha$. In this case, we call end of $\omega_z$, denoted by $\partial \omega_z$, the limit $\lim_{t\downarrow \alpha} \gamma(t)$ which by construction belongs to $\Sigma \setminus \Sigma_{tr} = \Sigma_{tan}$. The following result will allow to work in a neighborhood of a point of $  \Sigma_{tan}$. In the statement, $\mathcal{H}^1$ and $\mathcal{H}^2$ stand respectively for the $1$  and $2$-dimensional Hausdorff measures associated with $d^g$.

\begin{lemma}\label{LEMxbar1}
Assume that there is $x\in M$ such that  $\mathcal{X}^x$ has positive  $2$-dimensional Hausdorff measure. Then there is $\bar{x} \in \Sigma_{tan}$  such that for every neighborhood $\mathcal{V}$ of $\bar{x}$ in $M$, there are two closed sets $S_0, S_{\infty}$ in $\Sigma$ satisfying the following properties:
\begin{itemize}
\item[(i)] $S_0 \subset \Sigma_{tr} \cap \mathcal{V}$ and $\mathcal{H}^2(S_0)>0$,
\item[(ii)] $S_{\infty} \subset   \Sigma_{tan} \cap \mathcal{V}$,
\item[(iii)] for every $z\in S_0$, there is a half-orbit $\omega_z$ which is contained in $\mathcal{V}$ such that $\mathcal{H}^1(\omega_z)\leq 1$ and $\partial \omega_z \in S_{\infty}$.
\end{itemize}
\end{lemma}

\begin{proof}[Proof of Lemma \ref{LEMxbar1}]
Let $x\in M$ be fixed such that $\mathcal{H}^2(\mathcal{X}^x)>0$. Since by Lemma \ref{LEM1rect} the set $\Sigma_{tan}$ has Hausdorff dimension at most one, the set  $\tilde{\mathcal{X}}^x:= \mathcal{X}^x\cap \Sigma_{tr}$ is a subset of the smooth surface $\Sigma$ with positive area (for the volume associated with the restriction of $g$ to $\Sigma_{tr}$ or equivalently for any Lebesgue measure on $\Sigma$). Let us denote by $\mathcal{L}(\tilde{\mathcal{X}}^x)$ the set of Lebesgue density points of $\tilde{\mathcal{X}}^x$ on $\Sigma$ and define the singular distance from $x$, $d_x : \Sigma \rightarrow [0,\infty]$ by
$$
d_x(z) :=   \inf  \Bigl\{ \mbox{length}^g(\gamma) \, \vert \,   \gamma:[0,1] \rightarrow \Sigma, \, \mbox{hor. }, \, \gamma(0)=x, \, \gamma(1)=z  \Bigr\}, 
$$
for every $z\in \Sigma$, where $d_x(z)=\infty$ if there is no singular horizontal path joining $x$ to $z$. By construction, $d_x$ is lower semicontinuous on $\Sigma$ and finite on $\tilde{\mathcal{X}}^x$. We claim that the set $\mathcal{I}$ of $z\in \tilde{\mathcal{X}}^x$ such that any half-orbit $\omega_z$ of $L_{\Delta}$ satisfies $\mathcal{H}^{1} (\omega_z)=\infty$ has measure zero in $\Sigma$. As a matter of fact, consider some  $z\in \tilde{\mathcal{X}}^x$ such that $\mathcal{H}^{1} (\omega_z)=\infty$ for any half-orbit $\omega_z$. Since $z$ belongs to $\tilde{\mathcal{X}}^x$, there is a  (singular) horizontal path $ \gamma:[0,1] \rightarrow \Sigma$ with $\gamma(0)=z$, $\gamma(1)=x$ and $d_x(z)= \mbox{length}^g(\gamma) $ (the existence of such a minimizing arc follows easily by compactness arguments, see \cite{riffordbook}). If $\gamma([0,1]) $ is not contained in $ \Sigma_{tr}$ then it must contains an half-orbit of $z$ which is assumed to have infinite length, which contradicts $ \mbox{length}^g(\gamma) =d_x(z)<\infty$. Therefore, $\gamma([0,1]) \subset  \Sigma_{tr}$, $x$ belongs to $ \Sigma_{tr}$, and $z$ thas to be on the orbit of $x$, which proves that  $\mathcal{I}$  has measure zero in $\Sigma$. In conclusion, the set $\mathcal{S}:= \mathcal{L}(\tilde{\mathcal{X}}^x) \setminus \mathcal{I}$  has positive measure  in $\Sigma$ and for every $z\in \mathcal{S}$ there is an half-orbit $\omega_z$ with finite length. So there is $K>0$ and a closed subset $S$ of $\mathcal{S}$ with positive measure such that for every $z\in S$ there is an half-orbit $\omega_z$ such that $\mathcal{H}^{1}(\omega_z) < K$. In other words, any $z\in S$ can be joined to some point of  $\Sigma \setminus \Sigma_{tr} = \Sigma_{tan}$ by a singular horizontal path of length $< K$. Define the projection onto  the set $ \Sigma_{tan}$ (here $\mathcal{P}\left( \Sigma_{tan} \right)$ denotes the set of subsets of $ \Sigma_{tan}$)
$$
\mathcal{P}_{tan} \, : \, \Sigma_{tr}  \rightarrow \mathcal{P}\left(  \Sigma_{tan} \right)
$$
 by
\begin{multline*}
\mathcal{P}_{tan} (z) := \Bigl\{x'\in   \Sigma_{tan} \, \vert \, \exists \mbox{ an horizontal path } \gamma:[0,1] \rightarrow \Sigma \mbox{ with } \mbox{length}^g(\gamma)\leq K\\
\qquad \mbox{ s.t. } \gamma([0,1)) \subset  \Sigma_{tr}, \, \gamma(0)=z \mbox{ and }  \gamma(1)=x'\Bigr\}\\
=  \Bigl\{x'\in    \Sigma_{tan} \, \vert \, \exists \mbox{ an half-orbit } \omega_z \subset \mathcal{V} \mbox{ with } \mathcal{H}^1(\omega_z) \leq K  \mbox{ and } \partial \omega_z =\{x'\}\Bigr\},
\end{multline*}
for every $z\in \Sigma_{tr}$. By classical compactness results in sub-Riemannian geometry, the domain of $\mathcal{P}_{tan} $ is closed in $\Sigma_{tr}$ and its graph is closed in $\Sigma_{tr}  \times \Sigma_{tan}$. We have just proved that for every $z\in S$, the set $\mathcal{P}_{tan} (z) $ is nonempty. Now if we consider a countable and locally finite covering $\{\mathcal{B}^1_i\}_i$ of the closed set  $\Sigma_{tan}$ by geodesic open balls of radius $1$, there is a closed subset $S_1$ of $S$ of positive measure such that $\mathcal{P}_{tan} (z) \cap \mathcal{B}^1_{i_1}\neq \emptyset$ for every $z\in S_1$. Considering a finite covering of $\mathcal{B}^1_{i_1}$ by open sets $\{\mathcal{B}^2_i\}_i$ of diameters less than $1/2$, we get a closed set $S_2$ such that for some $i_2$, $\mathcal{P}_{tan} (z) \cap \mathcal{B}^2_{i_2} \neq \emptyset$ for all $z \in S_2$. By a recursive argument, we construct a decreasing sequence of open sets $\{\mathcal{V}_l\}_{l\in \N}$ of (geodesic) diameter $\leq 1/2^l$  along with a sequence of closed sets $\{S_l\}_{l\in \N}$ of positive measure such that for every $l\in \N$, $\mathcal{P}_{tan} (z) \cap \mathcal{V}_l \neq \emptyset$ for any $z\in S_l$. Let us prove that $\bar{x}$ such that $\cap_{l\in \N} \overline{\mathcal{V}_l} =\{\bar{x}\}$ yields the result. 

Let $\mathcal{V}$ be a neighborhood of $\bar{x}$ in $M$. By construction, there is $l\in \N$ such that $\overline{\mathcal{V}_l} \subset \mathcal{V}$ and $\mathcal{P}_{tan} (z) \cap \mathcal{V}_l \neq \emptyset$ for any $z\in S_l$. Set 
$$
S_{\infty} := \Bigl\{  x'\in  \Sigma_{tan} \cap \overline{\mathcal{V}_l}  \, \vert \, \exists z \in S_l \mbox{ s.t. }  x' \in   \mathcal{P}_{tan} (z)     \Bigr\}.
$$
By construction, $S_{\infty}$ is closed, nonempty, and satisfies (ii). Moreover, for every $z\in S_l$ there is an half-orbit $\omega_z$ with $\mathcal{H}^1(\omega_z)\leq K$ and $\partial \omega_z \in S_{\infty}$. Furthermore, we note that if  $x'\in S_{\infty}$, $z\in S_l$ and $\omega_z$ is an half-orbit such that $\partial \omega_z=x'$, then for every $z'$ in $\omega_z$, $x'$ is an end of the half-orbit $\omega_{z'}$ of $z'$ which is contained in $\omega_z$. This means that we can replace $S_l$ by some set $S_0$ such that (i) and (iii) are satisfied. 
 \end{proof}

From now on, we assume that there is some $x\in M$ such that  $\mathcal{X}^x$ has positive  $2$-dimensional Hausdorff measure and we fix some $\bar{x}\in \Sigma_{tan}$ given by Lemma \ref{LEMxbar1}. Since $\Sigma$ is a smooth surface, there are a relatively compact open neighborhood $\mathcal{V}$ of $\bar{x}$ in $M$ and a set of coordinates $(x_1,x_2,x_3)$ in $\mathcal{V}$ such that $\bar{x}=0$ and $\Sigma \cap \mathcal{V} = \left\{ x_3=0\right\}$, and there are  two smooth vector fields $X^1, X^2$ on $\mathcal{V}$ of the form 
\begin{eqnarray}\label{X1X2}
X^1 = \partial_{x_1} + A^1 \,  \partial_{x_3} \quad \mbox{and} \quad X^2 = \partial_{x_2} + A^2 \, \partial_{x_3}
\end{eqnarray}
with $A^1(0)=A^2(0)=0$ such that $\Delta(x) = \mbox{Span} \left\{X^1(x), X^2(x)\right\}$ for every $x\in \mathcal{V}$. In addition, without loss of generality we may assume that the above property holds on a neighborhood of $\overline{\mathcal{V}}$. Define the smooth vector field $\mathcal{Z} = \sum_{i=1}^3 \mathcal{Z}_i \, \partial_{x_i}$ on $\mathcal{V}$ by
\begin{eqnarray}\label{EQZ}
\mathcal{Z} :=   \left( X^1 \cdot x_3  \right) \, X^2  - \left( X^2 \cdot x_3  \right) \, X^1        = A^1 \, X^2 - A^2 \, X^1.
\end{eqnarray}
By construction, $\mathcal{Z}$ is a section of $\Delta$ and since $\mathcal{Z} \cdot x_3 = \mathcal{Z}_3= 0$ it is tangent to $\Sigma$. Moreover since $X^1, X^2$ are linearly independent in $\mathcal{V}$, the vector $\mathcal{Z}(x)$ vanishes for some $x\in \Sigma \cap \mathcal{V}$ if and only if $X^1\cdot x_3 = X^2 \cdot x_3=0$, that is if $\Delta$ is tangent to $\Sigma$ at $x$ or in other words if $x$ belongs to $\Sigma_{tan}$. In conclusion, the restriction $\mathcal{Z}$ of $\mathcal{Z}$ to $\Sigma$ is tangent to $\Sigma$ and it generates $L_{\Delta}$ which means that $L_{\Delta}(x) = \mbox{Span} (\mathcal{Z}(x))$ for any $x\in \Sigma \cap \mathcal{V}$. The following result is the key result in the proof of Theorem \ref{THMdim3}. The notation $\mbox{div}^{\Sigma}$ stands for the divergence on $\Sigma \cap \mathcal{V} =\{x_3=0\}$ with respect to the euclidean metric, that is, $\mbox{div}^{\Sigma} \mathcal{Z} = \left(\mathcal{Z}_1\right)_{x_1} + \left( \mathcal{Z}_2\right)_{x_2} $. We observe that since $\mathcal{Z}_3=0$ we have $\mbox{div}^{\Sigma} \mathcal{Z} = \mbox{div} \mathcal{Z}$ if $\mbox{div}$ denotes the divergence with respect to the Euclidean metric in $\mathcal{V}$ with coordinates $(x_1,x_2,x_3)$. 

\begin{lemma}
\label{LEMlocal1}
There is $K>0$ such that
\begin{eqnarray}\label{divK}
\left| \mbox{div}_x^{\Sigma} \mathcal{Z} \right| \leq K \, |\mathcal{Z}(x)| \qquad \forall x\in \Sigma \cap \mathcal{V}.
\end{eqnarray}
\end{lemma}

\begin{proof}[Proof of Lemma \ref{LEMlocal1}]
Firstly, we can notice that by construction, we have 
$$
[X^1,X^2] = \left( A^2_{x_1} + A^2_{x_3} \, A_1 -   A^1_{x_2} - A^1_{x_3} \, A_2 \right) \, \partial_{x_3} = 0 \quad \mbox{on} \quad \Sigma \cap \mathcal{V}.
$$
Secondly, by (\ref{X1X2})-(\ref{EQZ}) we can compute $\mathcal{Z}$ and its divergence. On $ \Sigma \cap \mathcal{V}$ we have 
$$
\mathcal{Z} = A^1 \, \partial_{x_2} - A^2\, \partial_{x_1} \quad \mbox{and} \quad \mbox{div}_x^{\Sigma} \mathcal{Z} = \mbox{div}_x \mathcal{Z} = A^1_{x_2}- A^2_{x_1} =  A^2_{x_3} \, A_1 - A^1_{x_3} \, A_2,
$$
so that
$$
\left| \mbox{div}_x^{\Sigma} \mathcal{Z} \right| \leq \sqrt{2} \, \max \left\{ \left| A^1_{x_3} \right|,  \left| A^2_{x_3} \right|  \right\} \, |\mathcal{Z}(x)| \qquad \forall x \in \Sigma \cap \mathcal{V}.
$$ 
We conclude easily by compactness of $\overline{\mathcal{V}}$.
\end{proof}

We are now ready to conclude the proof of Theorem \ref{THMdim3}. By Lemma \ref{LEMxbar1}, there are two closed sets $S_0, S_{\infty} \subset \Sigma \cap \mathcal{V}$ satisfiying the properties (i)-(iii). Denote by $\varphi_t$ the flow of $\mathcal{Z}$. For every $z\in S_0$, there is $\epsilon \in \{-1,1\}$ such that $\omega_z = \{ \varphi_{\epsilon t} (z) \, \vert \, t \geq 0\}$. Then, there there is $\epsilon \in \{-1,1\}$ and $S_0^{\epsilon} \subset S_0$ of positive measure such that for every $z\in S_0^{\epsilon}$ there holds
\begin{eqnarray}\label{22july1}
\omega_z =\Bigl\{ \varphi_{\epsilon t} (z) \, \vert \, t \geq 0\Bigr\} \subset \Sigma \cap \mathcal{V},  \quad \mathcal{H}^1(\omega_z) \leq 1, \quad \mbox{and} \quad \lim_{t\rightarrow +\infty} d\left( \varphi_{\epsilon t} (z), S_{\infty}\right)=0. 
\end{eqnarray}
Set for every $t\geq 0$,
$$
S_t := \varphi_{\epsilon t} \left( S_0^{\epsilon} \right).
$$
Denote by $\mbox{vol}^{\Sigma}$ the volume associated with the Euclidean metric on $\Sigma$. Since $S_{\infty}$ has volume zero (by (iii) and Lemma \ref{LEM1rect}), by the dominated convergence  Theorem, the last property in (\ref{22july1}) yields
\begin{eqnarray}\label{volSt}
\lim_{t\rightarrow +\infty} \mbox{vol}^{\Sigma} \left( S_t\right)=0.
\end{eqnarray}
Moreover, there is $C>0$ such that   for every $z\in S_0^{\epsilon}$ and every $t\geq 0$, we have ($\mathcal{H}^1$ is the Hausdorff measure with respect to $d^g$ while $|\cdot|$ denotes the Euclidean norm)
$$
\int_{0}^t \left| \mathcal{Z}  \left(\varphi_{\epsilon s}(z) \right) \right|   \, ds \leq C \mathcal{H}^1 \left( \omega_z\right)  \leq C.
$$
Therefore by Proposition \ref{PROPintdiv} and (\ref{divK}), we have for every $t\geq 0$
\begin{eqnarray*}
\mbox{vol}^{\Sigma} (S_t) = \mbox{vol}^{\Sigma} \left( \varphi_{\epsilon t} (S_0^{\epsilon})\right) & = & \int_{S_0^{\epsilon}} \exp\left( \int_{0}^t  \mbox{div}^{\Sigma}_{\varphi_{\epsilon s}(z)}(\epsilon \, \mathcal{Z})  \, ds\right) \,  d\mbox{vol}^{\Sigma}(z)\\
&\geq & \int_{S_0^{\epsilon}} \exp \left( -K \int_{0}^t  \left|\mathcal{Z}\left( \varphi_{\epsilon s}(z) \right)\right| \, ds\right) \,  d\mbox{vol}^{\Sigma}(z)\\
& \geq & e^{-K C} \, \mbox{vol}^{\Sigma}(S_0),
\end{eqnarray*}
which contradicts (\ref{volSt}). The proof of Theorem \ref{THMdim3} is complete.

\section{Proof of Theorem \ref{THMdim3sing}}\label{proofTHMdim3sing}

Throughout this section we assume that $\Delta$ is an analytic rank-two totally nonholonomic distribution on the analytic manifold $M$ of dimension three. The proof of Theorem \ref{THMdim3sing} is divided in two parts. The first part in Section \ref{SECfirstpart} is devoted to an extension of Lemma \ref{LEMxbar1} to the singular analytic case and the second part in Section \ref{SECsecondpart} contains the result from resolution of singularities, Proposition \ref{PROPMAIN}, which is required to end up the proof as in Section \ref{proofTMHdim3}. The proof of Proposition \ref{PROPMAIN}, which is rather involved,  is postponed to the next sections. To make the result more comprehensible, we study first in Section \ref{SECconical} the case of of a conical singularity.  As before we equip $M$ with a Riemannian metric $g$ whose geodesic distance is denoted by $d^g$ and we denote by $\mathcal{H}^1$ and $\mathcal{H}^2$ respectively the $1$  and $2$-dimensional Hausdorff measures associated with $d^g$. Moreover, we denote by $\mbox{vol}^g$ and  $\mbox{div}^g$ the volume form and divergence operator associated with $g$.

\subsection{First part of the proof}\label{SECfirstpart}

Recall that  by analyticity of $\Delta$ and $M$, the Martinet surface  defined by
$$
\Sigma := \Bigl\{ x\in M \, \vert \,  \Delta(x) + [\Delta,\Delta](x) \neq T_xM \Bigr\}
$$
is an analytic set in $M$ which enjoyes the structure of a coherent analytic space denoted by $\Sigma_{\Delta}=(\Sigma,\mathcal{O}_{\Sigma})$ (see Appendix \ref{secMartinetspace}). Following  \cite{bm88,lojasiewicz65}, we say that a point $x \in \Sigma$ is regular, if there exists an open neighborhood $U$ of $x$ such that $\Sigma \cap U$ is a smooth submanifold of $U$. A point which is not regular is called a singularity of the set $\Sigma$, the set of singularities of $\Sigma$ is denoted by $\mbox{Sing}(\Sigma)$. It is worth noticing that the singular set $\mbox{Sing}(\Sigma_{\Delta}) $, that is the singular set  of the analytic space $\Sigma_{\Delta}$, which appears in the statement of Theorem \ref{THMdim3sing} contains (but might be bigger than) the set $\mbox{Sing}(\Sigma)$. The set $\Sigma$ admits a stratification by strata  $\Sigma^0, \Sigma^1, \Sigma^2$ respectively of dimension zero, one, and two, with $\mbox{Sing}(\Sigma) = \Sigma^0\cup \Sigma^1$ and $\Sigma^2= \Sigma\setminus \mbox{Sing} (\Sigma)$, where $\Sigma^0$ is a locally finite union of points, $\Sigma^1$ is a locally finite union of analytic submanifolds of $M$ of dimension, and  $\Sigma^2$ is an analytic submanifold of $M$ of dimension two. Let us define two subsets  $\Sigma^{2}_{tr}, \Sigma^{2}_{tan}$ of $\Sigma^2$ by
$$
\Sigma_{tr}^2 := \Bigl\{ x\in \Sigma^2 \, \vert \,  \Delta(x) \cap T_x\Sigma^2 \mbox{ has dimension one}\Bigr\} \quad \mbox{and} \quad \Sigma_{tan}^2 := \Bigl\{ x\in \Sigma^2 \, \vert \,  \Delta(x) = T_x\Sigma^2 \Bigr\}.
$$
By construction, $\Sigma^{2}_{tr}$ is an open subset of $\Sigma^2$ and $\Sigma^{2}_{tan}$ is an analytic subset of $\Sigma^2$ of dimension at most one. Then, define two subsets  $\Sigma^{1}_{tr}, \Sigma^{1}_{tan}$ of $\Sigma^1$ by
$$
\Sigma^{1}_{tr} := \Bigl\{ x\in \Sigma^1 \, \vert \,  \Delta(x) \cap T_x\Sigma^1 =\{0\}\Bigr\} \quad \mbox{and} \quad \Sigma^{1}_{tan} := \Bigl\{ x\in \Sigma^1 \, \vert \,  T_x\Sigma^1 \subset \Delta(x) \Bigr\}.
$$
By the characterization of singular horizontal paths given in Proposition \ref{PROPsingmartinet}, we check easily that a non-constant horizontal path $\gamma : [0,1] \rightarrow M$ is singular if and only if it is contained in $\Sigma$ and horizontal with respect to the singular distribution $L_{\Delta}$ on $\Sigma$ defined by
$$
 L_{\Delta} (x) = \left\{ 
 \begin{array}{cl}
 \{0\} & \mbox{ if } x \in \Sigma^0 \cup \Sigma^{1}_{tr}\\
 T_x\Sigma^1 & \mbox{ if } x \in \Sigma^{1}_{tan}\\
 \Delta(x) \cap T_x\Sigma^{2} & \mbox{ if } x \in \Sigma^{2}_{tr}\\
 T_x\Sigma^{2} & \mbox{ if } x \in \Sigma^{2}_{tan},
 \end{array}
 \right.
 $$
 that is if it  satisfies
 \begin{eqnarray*}
\dot{\gamma}(t) \in L_{\Delta} (\gamma(t)) \qquad  \mbox{for a.e. } t \in [0,1].
\end{eqnarray*}
Since $L_{\Delta}$ is an analytic one-dimensional distribution on $\Sigma^2_{tr}$, the analytic surface $\Sigma^2_{tr}$ can be foliated by the orbits of  $L_{\Delta}$. By repeating the construction we did before Lemma  \ref{LEMxbar1}, we can associate to each point $z$ of $\Sigma^2_{tr}$, whose orbit is open, two half-orbits in $\Sigma^2_{tr}$ whose ends  in $\Sigma \setminus \Sigma^2_{tr}=\mbox{Sing} (\Sigma) \cup  \Sigma^{2}_{tan}$, denoted by $\partial \omega_z$ (where $\omega_z$ denotes the half-orbit), are well-defined if they have finite length. The proof of the following result follows the same lines as that of Lemma \ref{LEMxbar1} with $ \mbox{Sing}(\Sigma) \cup   \Sigma^{2}_{tan}$ instead of $\Sigma_{tan}$.

\begin{lemma}\label{LEMxbar2}
Assume that there is $x\in M$ such that  $\mathcal{X}^x$ has positive  $2$-dimensional Hausdorff measure. Then there is $\bar{x} \in \mbox{Sing}(\Sigma)\cup  \Sigma^{2}_{tan}$  such that for every neighborhood $\mathcal{V}$ of $\bar{x}$ in $M$, there are two closed sets $S_0, S_{\infty}$ in $\Sigma$ satisfying the following properties:
\begin{itemize}
\item[(i)] $S_0 \subset \Sigma_{tr}^2 \cap \mathcal{V}$ and $\mathcal{H}^2(S_0)>0$,
\item[(ii)] $S_{\infty} \subset \left( \mbox{Sing}(\Sigma) \cup   \Sigma^{2}_{tan}\right)  \cap \mathcal{V}$,
\item[(iii)] for every $z\in S_0$, there is an half-orbit $\omega_z$ which is contained in $\mathcal{V}$ such that $\mathcal{H}^1(\omega_z)\leq 1$ and $\partial \omega_z \in S_{\infty}$.
\end{itemize}
\end{lemma}

From now on, we assume that there is some $x\in M$ such that  $\mathcal{X}^x$ has positive  $2$-dimensional Hausdorff measure and we fix some $\bar{x}\in  \mbox{Sing}(\Sigma)\cup  \Sigma^{2}_{tan}$ given by Lemma \ref{LEMxbar2}. If $\bar{x}$ belongs to $\Sigma^2_{tan}$ we can conclude the proof exactly as for Theorem  \ref{THMdim3}, so we can assume  from now that $\bar{x}$ belongs to $\mbox{Sing}(\Sigma)$. Since $\Sigma$ has the structure of an analytic space, there is an open neighborbood $\mathcal{V}$ of $\bar{x}$ and an analytic function $h:\mathcal{V} \rightarrow \R$ such that
$$
\Sigma \cap \mathcal{V} =\Bigl\{ x\in \mathcal{V} \, \vert \, h(x)=0\Bigr\}
\quad
\mbox{and} \quad \mbox{Sing}(\Sigma_{\Delta})\cap\mathcal{V} =  \Bigl\{x\in \mathcal{V} \, \vert \, h(x)=0, \, d_xh=0\Bigr\}.
$$
Without loss of generality, we may also assume that there are two analytic vector fields $X,Y$ on $\mathcal{V}$ such that 
$$
\Delta (x) = \mbox{Span} \Bigl\{ X(x), Y(x) \Bigr\} \qquad \forall x \in \mathcal{V}.
$$
Define the analytic vector field $\mathcal{Z}$ on $\mathcal{V}$ by
\begin{eqnarray}\label{EQZ2}
\mathcal{Z} :=  \left(X\cdot h\right) \, Y -   \left( Y\cdot h\right) \, X,
\end{eqnarray}
where $X \cdot h, Y\cdot h$ denote respectively the Lie derivatives along $X$ and $Y$. By construction, the restriction $\mathcal{Z}$ of $\mathcal{Z}$ to $\Sigma$ is tangent to $\Sigma$, in particular  it is vanishing on $\mbox{Sing}(\Sigma_{\Delta})\cup\Sigma^2_{tan}$ and it generates $L_{\Delta}$ on $\Sigma^2_{tr} \setminus \mbox{Sing}(\Sigma_{\Delta})$ which means that $L_{\Delta}(x) = \mbox{Span} (\mathcal{Z}(x))$ for any $x\in ( \Sigma^2_{tr} \setminus \mbox{Sing}(\Sigma_{\Delta})) \cap \mathcal{V}$. Moreover the following property holds.

\begin{lemma}
\label{LEMlocalsing}
There is $K>0$ such that
\begin{eqnarray}\label{divKsing}
\left| \mbox{div}_x^{g} \mathcal{Z} \right| \leq K \, |\mathcal{Z}(x)| \qquad \forall x\in \Sigma \cap \mathcal{V}.
\end{eqnarray}
\end{lemma}

\begin{proof}[Proof of Lemma \ref{LEMlocalsing}]
 Taking $\mathcal{V}$ smaller and doing a change of coordinates, we may assume that in coordinates $(x_1,x_2,x_3)$ in $\mathcal{V}$ we have
$$
X = \partial_{x_1} \quad \mbox{and} \quad Y = \partial_{x_2} + A \, \partial_{x_3},
$$
where $A$ is smooth in $\mathcal{V}$. So we have $[X,Y] = A_{1} \, \partial_{x_3}$ with $A_{1}:=\frac{\partial A}{\partial x_1}$, so that $A_1(x)=0$ for every $x\in \Sigma$. Setting $h_i=\frac{\partial h}{\partial x_i}$ for $i=1, 2, 3$, we have
$$
\mathcal{Z} = h_1(\partial_{x_2} + A \, \partial_{x_3}) - (h_2+A \, h_3)\partial_{x_1} =: \mathcal{Z}_1 \partial_{x_1} + \mathcal{Z}_2 \partial_{x_2}  + \mathcal{Z}_3 \partial_{x_3}, 
$$
which yields (here we set $A_{3}:=\frac{\partial A}{\partial x_3}$ and $h_{ij} = \frac{\partial^2 h}{\partial x_i\partial x_j}$ for $i,j=1, 2, 3$)
$$
\mbox{div}^0 \, \mathcal{Z} =  h_{12} + h_{13}\, A + h_1 \, A_3 - h_{21} - A_1\, h_3 -A \, h_{31}=    h_1 \, A_3  - A_1\, h_3=  \mathcal{Z}_2 \, A_3 -  A_1 \, h_3,
$$
where $\mbox{div}^0$ denotes the divergence with respect to the Euclidean metric in coordinates $(x_1,x_2,\allowbreak x_3)$. Since $A_1=0$ on $\Sigma \cap \mathcal{V}$ and $g$ and the Euclidean metric are equivalent, we conclude easily. 
\end{proof}

By Lemma \ref{LEMxbar2}, there are two closed sets $S_0, S_{\infty} \subset \Sigma \cap \mathcal{V}$ satisfying the properties (i)-(iii). Denote by $\varphi_t$ the flow of $\mathcal{Z}$ and fix $z\in S_0$. By (iii), there is an half-orbit $\omega_z$ which is contained in $\mathcal{V}$ such that $\mathcal{H}^1(\omega_z)\leq 1$ and $\partial \omega_z \in S_{\infty}$. Two cases may appear, either  $\mathcal{Z}(z')\neq 0$ for all  $z'\in \omega_z$ or there is $z'\in \omega_z$ such that  $\mathcal{Z}(z')= 0$. In the first case, we infer the existence of $\epsilon \in \{-1,1\}$ such that $\omega_z = \{ \varphi_{\epsilon t} (z) \, \vert \, t \geq 0\}$. However, in the second case , since $\omega_z$ is contained in $\Sigma^2_{tr}$ and the restriction of $\mathcal{Z}$ to $\Sigma^2_{tr}$ vanishes on $\mbox{Sing}(\Sigma_{\Delta}) \cap \Sigma^2_{tr}$, we only deduce that there is $\epsilon \in \{-1,1\}$ such that $\varphi_{\epsilon t} (z)$ tends to some point of $\mbox{Sing}(\Sigma_{\Delta}) \cap \Sigma^2_{tr}$ as $t$ tends to $+\infty$. In conclusion, if we set $S_{\infty}':=S_{\infty}\cup (\mbox{Sing}(\Sigma_{\Delta}) \cap\mathcal{V})$, then $S_{\infty}'$ has $2$-dimensional Hausdorff measure zero (by (ii) and because $\mbox{Sing} (\Sigma_{\Delta})$ and $\Sigma^2_{\tan}$ are analytic subset of $M$ of codimension at least two ) and there is $\epsilon \in \{-1,1\}$ and $S_0^{\epsilon} \subset S_0$ with $\mathcal{H}^2(S_0)>0$  such that for every $z\in S_0^{\epsilon}$ there holds
\begin{eqnarray}\label{22july2}
\int_{0}^{+\infty} \left| \mathcal{Z}  \left(\varphi_{\epsilon s}(z) \right) \right|   \, ds \leq 1, \quad \varphi_{\epsilon t}(z) \notin S_{\infty}' \, \forall t\geq 0, \quad \mbox{and} \quad  \lim_{t\rightarrow +\infty} d^g\left( \varphi_{\epsilon t} (z), S_{\infty}'\right)=0,
\end{eqnarray}
where $ \left| \mathcal{Z}  \left(\varphi_{\epsilon s}(z) \right) \right| $ stands form the norm of $ \mathcal{Z}  \left(\varphi_{\epsilon s}(z) \right)$ with respect to $g$ and $d^g(\cdot,S_{\infty}')$ denotes the distance to $S_{\infty}'$ with respect to $g$. Set for every $t\geq 0$,
$$
S_t := \varphi_{\epsilon t} \left( S_0^{\epsilon} \right)
$$
and denote by $\mbox{vol}^{\Sigma}$ (resp. $\mbox{div}^{\Sigma}$) the volume (resp. the divergence) associated with the Riemannian metric induced by $g$ on $\Sigma$ which is well defined only on the smooth part $\Sigma^2= \Sigma\setminus \mbox{Sing} (\Sigma)$ of $\Sigma$. If we plan to conclude as in the proof of Theorem \ref{THMdim3}, we need two things. On the one hand, we need to show that $\lim_{t\rightarrow +\infty} \mbox{vol}^{\Sigma} \left( S_t\right)=0$ and on the other hand we need to get a control of $\mbox{div}^{\Sigma}(\mathcal{Z})$ by $|\mathcal{Z}|$ on $\Sigma^2= \Sigma\setminus \mbox{Sing} (\Sigma)$. By (\ref{22july2}) and the fact that  $\mathcal{H}^2(S_{\infty}')=0$,  the asymptotic  property on the volume of $S_t$ could be obtained. However, a control on the divergence seems difficult to obtain due to the last term in the formula (see Proposition \ref{PROPdivcodim1})  
\begin{eqnarray*}
\mbox{div}_x^{\Sigma} \mathcal{Z} =   \mbox{div}_x^g \mathcal{Z} + \frac{ \left(\mathcal{Z} \cdot \left( |\nabla h|^2\right)\right) (x)}{2|\nabla_x h|^2},
\end{eqnarray*}
which holds for any $x\in \left( \Sigma\setminus \mbox{Sing} (\Sigma_{\Delta}) \right) \cap \mathcal{V}$, because $|\nabla_x h|^2$ tends to zero as $x$ in $ \left( \Sigma\setminus \mbox{Sing} (\Sigma_{\Delta}) \right) \cap \mathcal{V}$ tends to $\mbox{Sing} (\Sigma_{\Delta})$. Let us study what happens in the case of a conical singularity.

\subsection{The case of a conical singularity}\label{SECconical}

Assume in the present section that $\Delta$ is an analytic totally nonholonomic distribution in $\R^3$ spanned globally by the two vector fields $X$ and $Y$,  that the Martinet surface is given by  (see Figure 4)
\begin{eqnarray}\label{23july3}
\Sigma = \Bigl\{ x =(x_1,x_2,x_3) \in \R^3 \, \vert \, h(x)=0\Bigr\} \quad \mbox{with} \quad h(x) = x_3^2-x_1^2-x_2^2,
\end{eqnarray}
and recall that the vector field $\mathcal{Z}$ on $\R^3$ which allows to recover the trace of $\Delta$ on $\Sigma$ outside the singularities is defined by
\begin{eqnarray*}
\mathcal{Z} :=  \left(X\cdot h\right) \, Y -   \left( Y\cdot h\right) \, X.
\end{eqnarray*}
Moreover, recall also that if $g$ is a given metric on $\R^3$ then, by Lemma \ref{divKsing}, there is an open neighborhood of the origin and  $K>0$ such that 
\begin{eqnarray}\label{23july4}
\left| \mbox{div}_x^{g} \mathcal{Z} \right| \leq K \, |\mathcal{Z}(x)| \qquad \forall x\in \Sigma \cap \mathcal{V}.
\end{eqnarray}
In order to perform a blow-up of $\Sigma$ at the origin, we consider the smooth mapping $\sigma : \R^3 \rightarrow \R^3$ defined by
$$
\sigma (u,v,w) = \left( uw,vw,w\right)
$$
whose restriction to $\R^3 \setminus E$ with $E=\{w=0\}$ is a smooth diffeomorphism onto its image $\R^3\setminus \{x_3=0\}$. Denoting by $h^*, X^*, Y^*$ and $\mathcal{Z}^*$ the pull-back of $h, X, Y, \mathcal{Z}$ by $\sigma$ on $\Sigma\setminus E$  we note that 
$$
h^*(u,v,w) = w^2 \left( 1- u^2-v^2 \right) \qquad \forall (u,v,w) \in \R^3 \setminus E
$$
and by (\ref{EQZ})
$$
\mathcal{Z}^* :=  \left(X^*\cdot h^*\right) \, Y^* -   \left( Y^*\cdot h^*\right) \, X^* \qquad \mbox{on } \R^3 \setminus E.
$$

\begin{figure}\label{fig4}
\includegraphics[width=5cm]{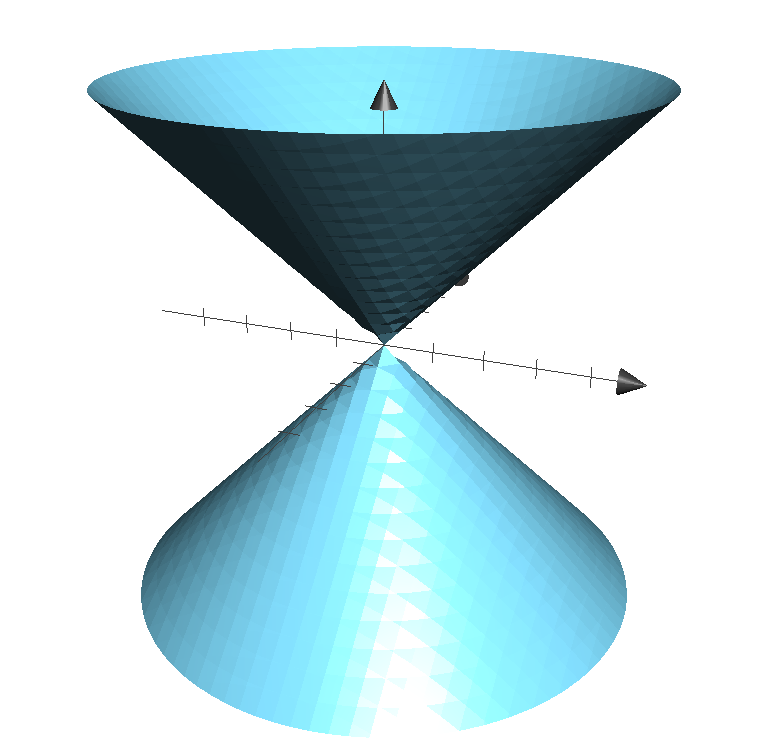} \hfill
\includegraphics[width=5cm]{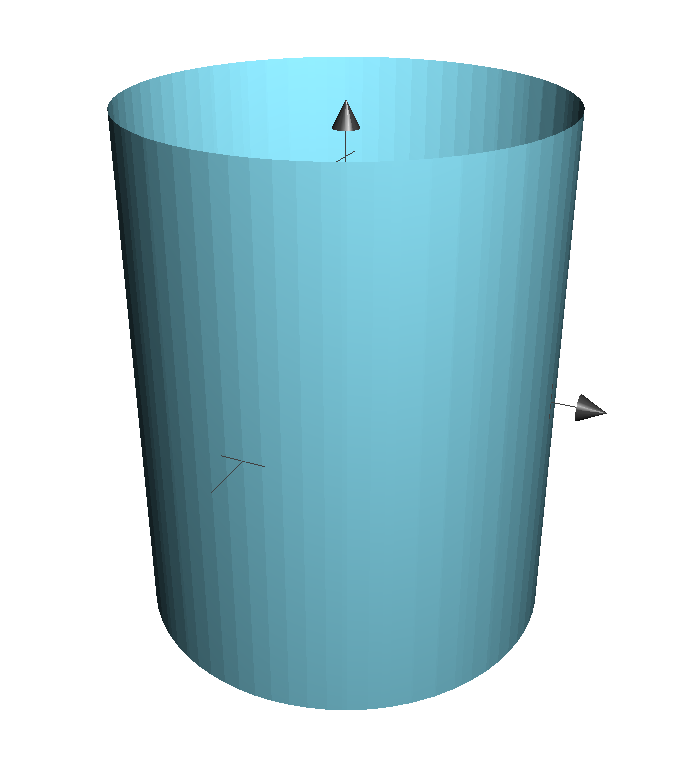}
\caption{The Martinet surfaces $\Sigma$ on the left and $\tilde{\Sigma}$ on the right}
\end{figure}

Then we set 
$$
\tilde{\Sigma} := \Bigl\{ (u,v,w) \in \R^3 \, \vert \, \tilde{h}(u,v,w)=0\Bigr\} \quad \mbox{with} \quad \tilde{h}(u,v,w) = 1 - u^2-v^2
$$
which is a smooth surface (see Figure 4) and
$$
\tilde{X} = w \, X^*, \quad \tilde{Y} = w \, Y^*, \quad \tilde{\mathcal{Z}} :=  \left(\tilde{X} \cdot \tilde{h}\right) \, \tilde{Y} -   \left( \tilde{Y} \cdot \tilde{h}\right) \, \tilde{X} \qquad \mbox{on } \R^3 \setminus E.
$$
We can notice that on $ \R^3 \setminus E$, we have
\begin{eqnarray*}
\mathcal{Z}^* & = &  \left(X^*\cdot \left( w^2 \, \tilde{h} \right)\right) \, Y^* -   \left( Y^*\cdot \left( w^2 \, \tilde{h} \right)\right) \, X^* \\
& = & w^2 \, \left[  \left(X^*\cdot \tilde{h}\right) \, Y^* -   \left( Y^*\cdot  \tilde{h}  \right) \, X^*  \right] + 2w \, \tilde{h} \, \left[  \left(X^*\cdot w \right) \, Y^* -   \left( Y^*\cdot  w  \right) \, X^*  \right] \\
& = &  \left(\tilde{X} \cdot \tilde{h}\right) \, \tilde{Y} -   \left( \tilde{Y} \cdot \tilde{h}\right) \, \tilde{X} + 2 \tilde{h} \, \left[  \left(X^*\cdot w \right) \, \tilde{Y} -   \left( Y^*\cdot  w  \right) \, \tilde{X}  \right] \\
& = & \tilde{\mathcal{Z}} + \tilde{h} \, \tilde{W},
\end{eqnarray*}
where $\tilde{W}$ is a vector field satisfying
\begin{eqnarray}\label{Wh}
\tilde{W} \cdot \tilde{h} & = & 2 \left[  \left(X^*\cdot w \right) \, \tilde{Y} -   \left( Y^*\cdot  w  \right) \, \tilde{X}  \right] \cdot \tilde{h} \nonumber\\
& = &\frac{2}{w}  \left[  \left(\tilde{X}\cdot w \right) \, \tilde{Y} \cdot \tilde{h}-   \left( \tilde{Y}\cdot  w  \right) \, \tilde{X} \cdot \tilde{h} \right] = - \frac{2}{w}\tilde{\mathcal{Z}}\cdot w.
\end{eqnarray}
Consider now the euclidean metric $\tilde{g} = du^2 +dv^2+dw^2$ in $(u,v,w)$ coordinates and denote by $\mbox{div}^{\tilde{g}}$ the associated divergence and by $\mbox{div}^{\tilde{\Sigma}}$ (resp. $\mbox{vol}^{\tilde{\Sigma}}$) the divergence (resp. the volume) associated to the metric induced on $\tilde{\Sigma}$ by $\tilde{g}$. Thanks to Proposition \ref{PROPdivcodim1} we have on $ \tilde{\Sigma} \setminus E$,
$$
\mbox{div}^{\tilde{\Sigma}}(\mathcal{Z}^*) = \mbox{div}^{\tilde{\Sigma}}(\tilde{\mathcal{Z}}) + \mbox{div}^{\tilde{\Sigma}}(\tilde{h}\tilde{W}) = \mbox{div}^{\tilde{\Sigma}}(\tilde{\mathcal{Z}}) =  \mbox{div}^{\tilde{g}}(\tilde{\mathcal{Z}}) + \frac{ \left| \left( \tilde{\mathcal{Z}} \cdot |\tilde{\nabla} \tilde{h}|^2\right) \right|}{2|\tilde{\nabla} \tilde{h}|^2},
$$
(where $\tilde{\nabla}\tilde{h}$ stands for the gradient of $\tilde{h}$ with respect to $\tilde{g}$ and the norm is taken with respect to $\tilde{g}$ as well) which implies, since $|\tilde{\nabla}\tilde{h}|^2=1$ on $\tilde{\Sigma}$, 
\begin{eqnarray*}
\mbox{div}^{\tilde{\Sigma}}(\mathcal{Z}^*) =  \mbox{div}^{\tilde{g}}(\tilde{\mathcal{Z}}).
\end{eqnarray*}
We can compute $\mbox{div}^{\tilde{g}}(\tilde{\mathcal{Z}})$, we have on $\tilde{\Sigma} \setminus E$,
$$
\mbox{div}^{\tilde{g}}(\tilde{\mathcal{Z}})  = \mbox{div}^{\tilde{g}}(\mathcal{Z}^*)  -  \mbox{div}^{\tilde{g}}\left( \tilde{h} \, \tilde{W} \right) =  \mbox{div}^{\tilde{g}}(\mathcal{Z}^*)  - \tilde{W}\cdot \tilde{h}.
$$
Moreover, if $\mathcal{Z} = A \, \partial_1 + B \, \partial_2 + C \, \partial_3$ then we have 
$$
\mathcal{Z}^*  = \left(  \frac{1}{w} \, A^* - \frac{u}{w} \, C^*\right) \, \partial_u +  \left(  \frac{1}{w} \, B^* - \frac{v}{w} \, C^*\right) \, \partial_v + C^* \, \partial_w
$$
which yields for $w\neq0$,
\begin{eqnarray*}
 \mbox{div}^{\tilde{g}}(\mathcal{Z}^*)  &= &  \frac{1}{w} \, \frac{ \partial A^*}{\partial u} - \frac{1}{w} \, C^* -  \frac{u}{w} \, \frac{\partial C^*}{\partial u}  + \frac{1}{w} \, \frac{ \partial B^*}{\partial v} - \frac{1}{w} \, C^* -  \frac{v}{w} \, \frac{\partial C^*}{\partial v} +  \frac{\partial C^*}{\partial w}      \\
 & = &    A_1^*  - \frac{1}{w} \, C^*  - u  \, C_1^*  + B_2^*  - \frac{1}{w} \, C^* - v\, C_2^* +  u\, C_1^* +  v\, C_2^* + C_3^* \\
 & = &  A_1^*  + B_2^* + C_3^* - \frac{2}{w} \, C^*\\
 & = & \left[\mbox{div}^g(\mathcal{Z})\right]^* - \frac{2}{w} \, \mathcal{Z}^*\cdot w =  \left[\mbox{div}^g(\mathcal{Z})\right]^* - \frac{2}{w} \, \tilde{\mathcal{Z}}\cdot w 
 \end{eqnarray*}
 on $\tilde{\Sigma}$ (because $\mathcal{Z}^*=\tilde{\mathcal{Z}}$), where $g$ denotes the euclidean metric (in coordinates $(x_1,x_2,x_3)$). In conclusion, by (\ref{Wh}) we have   
 \begin{eqnarray}\label{divconical}
\mbox{div}^{\tilde{\Sigma}} \left(\mathcal{Z}^*\right)  =  \left[\mbox{div}^g(\mathcal{Z})\right]^* \qquad \mbox{on } \tilde{\Sigma} \setminus E.
\end{eqnarray}
Assume now that the $\bar{x}$ given by Lemma  \ref{LEMxbar2} is equal to $0$ which is the unique singularity of $\Sigma$ and that there is a set $S_0$ of positive measure in $\Sigma \setminus \{\bar{x}\}$  such that 
$$
\varphi_t(z) \in \Sigma \setminus \{\bar{x}\} \quad \forall t \geq 0 \quad \mbox{and} \quad    \lim_{t \rightarrow +\infty} \varphi_t(z)=\bar{x} \qquad \forall z \in S_0,
$$
and
\begin{eqnarray}\label{23july29}
\int_{0}^{+\infty} \left| \mathcal{Z}  \left(\varphi_{s}(z) \right) \right|   \, ds \leq 1\qquad \forall z \in S_0,
\end{eqnarray}
where as before $\varphi_t$ denotes the flow of $\mathcal{Z}$. If we pull-back everything by $\sigma$, by setting 
$$
S_0^*=\sigma^{-1}(S_0) \quad \mbox{and} \quad S_t^* := \varphi_t^* \left( S_0^*\right) \qquad \forall t \geq 0,
$$
where  $\varphi_t^*$ denotes the flow of $\mathcal{Z}^*$, then all positive orbits starting from $S_0^*$ accumulate on the circle $E\cap \tilde{\Sigma}$ at infinity, so that
\begin{eqnarray}\label{23july1}
\lim_{t \rightarrow +\infty} \mbox{vol}^{\tilde{\Sigma}} (S_t^*)=0.
\end{eqnarray}
Then thanks to (\ref{divKsing}) and  (\ref{divconical})-(\ref{23july29}), we can repeat what we did  at the end of the proof of Theorem \ref{THMdim3} to get for every $t\geq 0$, 
\begin{eqnarray*}
\mbox{vol}^{\tilde{\Sigma}} \left(S_t^* \right) & = & \int_{S_0^*} \exp\left( \int_{0}^t  \mbox{div}^{\tilde{\Sigma}}_{\varphi_s^*(z)}\left( \mathcal{Z}^*\right)  \, ds\right) \,  d\mbox{vol}^{\tilde{\Sigma}}(z)\\
 & = & \int_{S_0^*} \exp\left( \int_{0}^t  \mbox{div}^{g}_{\sigma(\varphi_s^*(z))}\left( \mathcal{Z}\right)  \, ds\right) \,  d\mbox{vol}^{\tilde{\Sigma}}(z)\\
&\geq & \int_{S_0} \exp \left( -K \int_{0}^t  \left|\mathcal{Z}\left( \varphi_s(\sigma(z)) \right)\right| \, ds\right) \,  d\mbox{vol}^{\tilde{\Sigma}}(z)\\
& \geq & e^{-K} \,\mbox{vol}^{\tilde{\Sigma}}\left(S_0^* \right),
\end{eqnarray*}
which contradicts (\ref{23july1}) and conclude the proof of Theorem \ref{THMdim3sing} in the case where $\Sigma$ is given by (\ref{23july3}).

\subsection{Second part of the proof}\label{SECsecondpart}

Return to the proof of Theorem \ref{THMdim3sing} and, for sake of simplicity, set $S_0:=S_0^{\epsilon}\setminus \mbox{Sing}(\Sigma_{\Delta})$, $S_{\infty}=S_{\infty}'$ and assume that $\epsilon=1$. We recall that by construction, $S_0 \subset \Sigma \setminus \mbox{Sing}(\Sigma_{\Delta})$, $\mathcal{H}^2(S_0)>0$ and for every $z\in S_0$ there holds
\begin{eqnarray}\label{22july5}
\int_{0}^{+\infty} \left| \mathcal{Z}  \left(\varphi_{ s}(z) \right) \right|   \, ds \leq 1, \quad  \varphi_{ t}(z) \notin S_{\infty} \, \forall t\geq 0, \quad \mbox{and} \quad  \lim_{t\rightarrow +\infty} d^g\left( \varphi_{ t} (z), S_{\infty}\right)=0,
\end{eqnarray}
In the general case, as shown in the previous section, a way to conclude the proof of Theorem  \ref{THMdim3sing} is to find a change of coordinates outside the singularities in which $\Sigma$, denoted by $\tilde{\Sigma}$ in new coordinates, admits a volume form for which the divergence $\mbox{div}^{\tilde{\Sigma}}$ of the pull-back of $\mathcal{Z}$ can be compared with $\mbox{div}^g(\mathcal{Z})$ (see (\ref{divconical})) and for which the volume of $S_t:= \varphi_t ( S_0 )$ tends to zero as $t$ tends to $+\infty$.  This is the purpose of the next result (we recall that a mapping $\sigma: \mathcal{W} \rightarrow \mathcal{V}$ is proper if for every compact set  $K \subset \mathcal{V}$ the set $\sigma^{-1}(K)$ is compact as well).

\begin{proposition}\label{PROPMAIN}
There are a three dimensional manifold with corners $\mathcal{W}$, with boundary set $E \subset \mathcal{W}$, a smooth and proper mapping $\sigma: \mathcal{W} \rightarrow \mathcal{V}$, a smooth surface $\tilde{\Sigma} \subset \mathcal{W}$, and a smooth volume form $\tilde{\omega}$ on $\tilde{\Sigma} \setminus E$ such that the following properties are satisfied:
\begin{itemize}
\item[(i)] $\sigma( \tilde{\Sigma}) = \Sigma$, $\sigma(E)= \mbox{Sing} (\Sigma_{\Delta})$ and the restriction of $\sigma$ to $\mathcal{W}\setminus E$ is a diffeomorphism onto its image $\mathcal{V} \setminus \mbox{Sing}(\Sigma_{\Delta})$,
\item[(ii)] setting $Z^*:=\sigma^*Z$ on $\mathcal{W} \setminus E$ and $S_t^*=\sigma^{-1}(S_t)=\varphi_t^* (\sigma^{-1}(S_0)) \subset \tilde{\Sigma} \setminus E$, and denoting respectively by $\mbox{vol}^{\tilde{w}}$ and $\mbox{div}^{\tilde{\omega}}$ the volume and divergence operator associated to $\tilde{\omega}$ on $\tilde{\Sigma} \setminus E$ we have
\begin{eqnarray}\label{7aout2}      
\mbox{div}^{\tilde{\omega}} \left(\mathcal{Z}^*\right)  =  \left[\mbox{div}^g(\mathcal{Z})\right]^*  \mbox{ on } \tilde{\Sigma} \setminus E
\end{eqnarray}
and
\begin{eqnarray}\label{7aout3}
\lim_{t \rightarrow +\infty}    \mbox{vol}^{\tilde{w}} \left(   S_t^*\right) =0.
\end{eqnarray}
\end{itemize}
\end{proposition}

Proposition \ref{PROPMAIN} will follow from techniques of resolution of singularities. Its proof is given in the next section.

\section{Proof of Proposition \ref{PROPMAIN}}\label{SECPROPMAIN}

Our proof of Proposition \ref{PROPMAIN} relies on an elementary, but extensive, construction. It is based on a result of resolution of singularities (Proposition \ref{PROPresolution})  and two lemmas (Lemmas \ref{LEMstep2} and \ref{LEMvolume}). We present an overview of the proof in the next section, prove the lemmas in Sections \ref{ssecStep2} and \ref{ssecStep4}, and since it requires some material from algebraic geometry, we postpone the proof of Proposition  \ref{PROPresolution} to Section \ref{secRS}.

\subsection{Overview of the proof}\label{ssecOverview}

Recall that we are considering an open neighborhood $\mathcal{V}$ of $\bar{x}$, an analytic function $h:\mathcal{V} \rightarrow \R$ such that
$$
\Sigma \cap \mathcal{V} =\Bigl\{ x\in \mathcal{V} \, \vert \, h(x)=0\Bigr\}
\quad
\mbox{and} \quad \mbox{Sing}(\Sigma_{\Delta})\cap\mathcal{V} =  \Bigl\{x\in \mathcal{V} \, \vert \, h(x)=0, \, d_xh=0\Bigr\},
$$
and an analytic vector-field $\mathcal{Z}$ on $\mathcal{V}$ defined by
\begin{eqnarray}\label{EQZ3}
\mathcal{Z} :=  \left(X\cdot h\right) \, Y -   \left( Y\cdot h\right) \, X,
\end{eqnarray}
which generated $L_{\Delta}$ outside of the singular locus of the analytic space $\Sigma_{\Delta}$, where $X$ and $Y$ are local analytic sections of $\Delta$.  Note that $\mathcal{V}$ can be assumed to be orientable and let us denote by $\omega=\mbox{vol}^g$ the volume form associated with the metric $g$ on $\mathcal{V}$. 

\begin{proposition}\label{PROPresolution}[Resolution of singularities]
There exist a three-dimensional orientable Riemannian manifold with corners $\mathcal{W}$, with boundary set $E \subset \mathcal{W}$, a smooth and proper morphism $\sigma : \mathcal{W} \to \mathcal{V}$, a number $r\in \mathbb{N}$, and smooth functions $\rho_i : \mathcal{W} \to \mathbb{R}$ for $i =1, \ldots, r$ which are strictly positive on $\mathcal{W} \setminus E$ such that the following properties hold:
\begin{itemize}
\item[(i)] The restriction of $\sigma$ to $\mathcal{W}\setminus E$ is a diffeomorphism onto its image $\mathcal{V} \setminus \mbox{Sing}(\Sigma_{\Delta})$ and $\sigma(E)= \mbox{Sing} (\Sigma_{\Delta})$. 
\item[(ii)] There exist a smooth function $\tilde{h} : \mathcal{W} \to \mathbb{R}$ and exponents $\alpha_i \in \mathbb{N}$ for $i=1, \ldots, r$ such that
\[
h \circ \sigma =  \alpha  \cdot \tilde{h} \quad \text{ with } \quad \alpha:= \prod_{i=1}^r \rho_i^{\alpha_i}.
\]
Furthermore, the surface $\tilde{\Sigma} =  \{\tilde{h}=0\}$ is an analytic sub-manifold of $\mathcal{W}$ such that $\sigma( \tilde{\Sigma}) = \Sigma$ and $\tilde{\Sigma} \cap \{d\tilde{h}=0\} = \emptyset$.
\item[(iii)] There exist a smooth metric $\tilde{g}$ over $\mathcal{W}$ and exponents $\beta_i \in \mathbb{N}$ for $i=1, \ldots, r$ such that
\[
d\sigma^{\ast}(\omega)  = \beta \cdot  \lambda \quad \text{ with } \quad \beta:= \prod_{i=1}^r \rho_i^{\beta_i},
\]
where $\lambda$ is the volume form associated with the metric $\tilde{g}$.
\item[(iv)] There exists a smooth vector field $Z$ over $\tilde{\Sigma}$ which is given by the following expression at points in $\tilde{\Sigma}\setminus E$:
\[
Z: = \frac{\beta}{\alpha} \cdot\mathcal{Z}^{\ast} =  \prod_{i=1}^r \rho_i^{\beta_i-\alpha_i} \cdot \mathcal{Z}^{\ast}. 
\]
where $\mathcal{Z}^{\ast}$ denotes the pulled-back vector-field $\mathcal{Z}$, that is, $\mathcal{Z}^{\ast}:= d\sigma^{-1}(\mathcal{Z})$, which is well-defined outside of $E$. Moreover, since $\Delta(y) \cap T_y(\mbox{Sing}(\Sigma_{\Delta})) = T_y(\mbox{Sing}(\Sigma_{\Delta}))$ for all $y\in \mbox{Sing}(\Sigma_{\Delta})$,  $Z$ is tangent to $\Sigma \cap E$.
\end{itemize}
\end{proposition}

Now, we consider the smooth function $\xi: \tilde{\Sigma} \to \mathbb{R}$ and the normal vector $N:\tilde{\Sigma} \to T \mathcal{W}$ given by 
$$
\xi (x) := \bigl| \tilde{\nabla}\tilde{h}  \bigr|^{\tilde{g}}_x \quad \mbox{and} \quad N(x) = \frac{ \tilde{\nabla}_x\tilde{h}}{ | \tilde{\nabla}_x\tilde{h}  |^{\tilde{g}}_x } \qquad \forall x \in \tilde{\Sigma},
$$
where $|\cdot|^{\tilde{g}}, \tilde{g}$ denotes respectively the norm and connection associated with $\tilde{g}$. Note that $\xi$ does not vanish. Then, we define
\[
\tilde{\omega} : =  \frac{1}{\xi} \cdot i_{N}\left( \frac{1}{\alpha} \cdot  d\sigma^{\ast}(\omega)  \right) = \frac{1}{\xi} \cdot i_{N} \left( \frac{\beta}{\alpha} \cdot   \lambda \right)
\]
which is a smooth volume form on  $\tilde{\Sigma}\setminus E$ and we denote by $\tilde{\lambda}$ the volume form $\tilde{\lambda} :=   i_{N} \lambda$ which is defined over $\tilde{\Sigma}$. The following result (which extends (\ref{divconical}) which was obtained in the conical case) yields (\ref{7aout2}).

\begin{lemma}\label{LEMstep2}
For every  $ x \in \tilde{\Sigma}\setminus E$, we have $\mbox{div}^{\tilde{\omega}}_x (\mathcal{Z}^{\ast}) = \left[ \mbox{div}^{g}(\mathcal{Z}) \right]^{\ast}_x:=  \mbox{div}^{g}_{\sigma(x)}(\mathcal{Z})$.
\end{lemma}

Recall that by the construction performed in Section \ref{SECsecondpart}, there is a set $S_0 \subset \Sigma \setminus \mbox{Sing}(\Sigma_{\Delta})$ with $\mathcal{H}^2(S_0)>0$ such that  for every $z\in S_0$, the three properties in (\ref{22july5}) hold. Since the set $S_t:= \varphi_t ( S_0 )$ is contained in $\Sigma \setminus Sing(\Sigma_{\Delta})$ for all $t\geq 0$ and $Z$ differs from $\mathcal{Z}^{\ast}$ by a positive function in $\tilde{\Sigma} \setminus E$, we infer that all orbits of $Z$ starting from $\tilde{S}_0:= \sigma^{-1}(S_0)$ converge to $E \cap \tilde{\Sigma}$. More precisely, since $S_{\infty}$ is compact and $\sigma$ is a proper morphism, there is a relatively compact set $\tilde{S}_{\infty}$ with $\sigma(\tilde{S}_{\infty}) \subset S_{\infty}$ such that all the orbits $\varphi_{ t}^Z (x)$ with $x\in \tilde{S}_0$ tend to $\tilde{S}_{\infty}$ as $t$ tends to $+\infty$. By the dominated convergence  Theorem, we deduce easily that the sets $\tilde{S_t}$ defined by $\tilde{S}_t :=\varphi_t^{Z} (\sigma^{-1}(S_0)) \subset \tilde{\Sigma} \setminus E$ for every $t \in [0,\infty)$ satisfy (because $\lambda$ is not degenerate over $E$)
\begin{eqnarray}\label{7aout1}
\lim_{t \rightarrow +\infty}    \mbox{vol}^{\tilde{\lambda}} \bigl( \tilde{S}_t\bigr) =0.
\end{eqnarray}
Then (\ref{7aout3}) will follow readily from the following lemma.

\begin{lemma}\label{LEMvolume}
Let $\mathcal{W}$ be a three dimensional manifold with corners with boundary set $E \subset \mathcal{W}$ and $\mathcal{S}$ be a smooth submanifold of $\mathcal{W}$ of dimension two, let $v$ a smooth volume form on $\mathcal{S}$, $V$ a smooth vector field on $\mathcal{W}$ whose restriction to $\mathcal{S}$ is tangent to $\mathcal{S}$, and $f:\mathcal{S}\setminus E  \rightarrow (0,+\infty)$ be a smooth function.  Let $\hat{v}$ and $\hat{V}$ be the volume form and vector field defined respectively by 
$$
\hat{v} :=  (f)^{-1} \, v  \quad\mbox{and} \quad \hat{V} := f \, V,
$$
which are well-defined (and smooth)  on $\mathcal{S} \setminus E$. Denoting respectively by $\varphi_t^V, \varphi_t^{\hat{V}}$ the flows of $V, \hat{V}$, assume that there are two relatively compact sets $K_0, K_{\infty} \subset \mathcal{W}$  such that the following properties are satisfied:
\begin{itemize}
\item[(P1)] $K_0 \subset  \mathcal{S} \setminus E$ and $\mbox{vol}^{v}(K_0)>0$.
\item[(P2)]  $ K_{\infty} \subset \mathcal{S} \cap E$ and $\mbox{vol}^{v}(K_{\infty})=0$.
\item[(P3)] The sets $K_t:=\varphi_t^V(K_0) \subset \mathcal{S}\setminus E$ and $\hat{K}_t:=\varphi_t^{\hat{V}}(K_0) \subset \mathcal{S} \setminus E$ are defined for all $t  \in [0,+\infty)$. 
\item[(P4)] The intersection $\overline{K_t} \cap E$ is empty for all $t \in [0,\infty)$.
\item[(P5)] For every $x \in K_0$, the $\omega$-limit set of the orbit of $V$ through $x$ is contained in $K_{\infty}$.
\end{itemize}
Then, apart from shrinking $K_0=\hat{K}_0$, we can guarantee that
\begin{eqnarray}\label{7aout99}
\mbox{vol}^{\hat{v}} \bigl(\hat{K}_0 \bigr)>0  \quad \mbox{and} \quad
\lim_{t \to +\infty}  \mbox{vol}^{\hat{v}} \bigl(\hat{K}_t \bigr)   = 0.
\end{eqnarray}
\end{lemma}

To prove  (\ref{7aout3}), noting  that $\tilde{\lambda} = \frac{\alpha\, \xi}{\beta } \tilde{\omega}$ and $Z = \frac{\beta}{\alpha} \mathcal{Z}^{\ast}$,  we just need to apply Lemma \ref{LEMvolume} with $\mathcal{S}=\tilde{\Sigma}, v=  \frac{1}{\xi}\lambda,  V= Z ,f= \frac{\alpha}{\beta}, K_0=\tilde{S}_0$, and $K_{\infty}=\tilde{S}_{\infty}$.

\subsection{Proof of Lemma \ref{LEMstep2}} \label{ssecStep2}

We follow the notations of the previous subsection. First of all, by equation \eqref{EQZ3}, the pulled back vector field $\mathcal{Z}^{\ast}$ satisfies the the following equation on $\mathcal{W}\setminus E$:
\begin{equation}\label{10aout2}
\begin{aligned}
\mathcal{Z}^{\ast} = \tilde{\mathcal{Z}} + \tilde{h} \, \tilde{W}, \quad \text{ where }\quad & \tilde{\mathcal{Z}} := \alpha \big[ ({X^{\ast}} \cdot \tilde{h} )  {Y^{\ast}} - ({Y^{\ast}} \cdot \tilde{h}  ) {X^{\ast}}  \big],\\
& {\tilde{W}}:=({X^{\ast}}\cdot \alpha )  {Y^{\ast}} - ({Y^{\ast}} \cdot \alpha  ) {X^{\ast}},\\
& {X^{\ast}} := d\sigma^{-1}(X), \text{ and } {Y^{\ast}} := d\sigma^{-1}(Y)
\end{aligned}
\end{equation}
and $\tilde{\mathcal{Z}}$ is the vector-field which controls the dynamic over $\tilde{\Sigma}$ (since $\tilde{\Sigma} = (\tilde{h}=0)$). Moreover, note that  $\tilde{W}\cdot \tilde{h}=-\alpha^{-1} \, \left( \tilde{\mathcal{Z}} \cdot \alpha\right)$.  Recall that $\tilde{\lambda}$ is the volume form on $\tilde{\Sigma}$ given by $i_N\lambda$ and that $\tilde{\Sigma}=\{\tilde{h}=0\}$. Then by Proposition \ref{PROPdivcodim1}, we have
$$
div^{\tilde{\lambda}}(\tilde{\mathcal{Z}} ) =  div^{\lambda}(\tilde{\mathcal{Z}}) + \frac{\tilde{Z} \cdot \xi^2}{2 \,\xi^2} = div^{\lambda}(\tilde{\mathcal{Z}}) + \frac{\tilde{Z} \cdot \xi}{ \xi}.
$$
Moreover,
\begin{eqnarray*}
\mbox{div}^{\lambda}(\tilde{\mathcal{Z}} ) & = & \mbox{div}^{\lambda}(\mathcal{Z}^* ) -  \mbox{div}^{\lambda}  \left(   \tilde{h} \, \tilde{W} \right) \\
& = &\mbox{div}^{\lambda}(\mathcal{Z}^* ) - \tilde{W}\cdot \tilde{h}  - \tilde{h} \, \mbox{div}^{\lambda}  \left( \tilde{W} \right)\\
& = & \mbox{div}^{\lambda}(\mathcal{Z}^* ) + \frac{\tilde{\mathcal{Z}}\cdot \alpha}{\alpha}  - \tilde{h} \, \mbox{div}^{\lambda}  \left( \tilde{W} \right).
\end{eqnarray*}
Next, from the fact that $d\sigma^{\ast}\omega = \beta \lambda$ we conclude that:
\[
div^{\lambda}(\mathcal{Z}^{\ast}) = div^{\omega}(\mathcal{Z})^{\ast} - \frac{\mathcal{Z}^{\ast} \cdot \beta}{\beta} 
\]
In conclusion, since  $\tilde{\omega} = \frac{\beta}{\alpha \xi}\tilde{\lambda}$, we have on $\tilde{\Sigma}\setminus E$,
\begin{eqnarray*}
 div^{\tilde{\omega}}(\mathcal{Z}^{\ast}) = div^{\tilde{\omega}}(\tilde{\mathcal{Z}} ) &=  & div^{\tilde{\lambda}}(\tilde{\mathcal{Z}} ) +  \frac{\alpha\,  \xi}{\beta}\left(\tilde{\mathcal{Z}} \cdot \frac{\beta}{\alpha \, \xi} \right)\\
&  =  &div^{\omega}(\mathcal{Z})^{\ast} -  \frac{\mathcal{Z}^{\ast} \cdot \beta}{\beta} + \frac{\tilde{\mathcal{Z}}\cdot \alpha}{\alpha} + \frac{\tilde{Z} \cdot \xi}{\xi} +  \frac{\alpha\,  \xi}{\beta}\left(\tilde{\mathcal{Z}} \cdot \frac{\beta}{\alpha \, \xi} \right)
\end{eqnarray*}
which is equal to $div^{\omega}(\mathcal{Z})^{\ast}$. 
\subsection{Proof of Lemma \ref{LEMvolume}}\label{ssecStep4}

Let $\mathcal{W}, \mathcal{S}, v, V, f,, \hat{v}, \hat{V}, K_0, K_{\infty}$ be as in the statement of Lemma \ref{LEMvolume}. 
First of all, since by (P1) $K_0$ is a relatively compact set with positive measure in $\mathcal{S} \setminus E$, apart from shrinking $K_0$, we can suppose that $K_0$ is a compact subset of $\mathcal{S} \setminus E$ with positive measure ({\it i.e. $\mbox{vol}^{v}(K_0)>0$}). For every $x\in \mathcal{S} \setminus E$, denote by 
$r(x,\cdot)$ the solution to the Cauchy problem
\begin{equation}
\label{eq:r0}
\frac{\partial r}{\partial t} (x,t) = f\left(\varphi^{\hat{V}}_t(x)\right) \quad \text{ and } \quad r(x,0)= 0
\end{equation}
which is well-defined on a maximal subinterval $[0,T(x,t))$ of the set $[0,\infty)$. By assumption (P3), the function $r$ is well-defined on $K_0 \times [0,+\infty)$, smooth on its domain, and it satisfies $\varphi^{\hat{V}}_t(x) = \varphi_{r(x,t)}^V(x)$ for any $(x,t)$ in its domain. Let us consider the functions $r_0, d_0: [0,+\infty) \to [0,+\infty)$ given by
$$
r_0(t) := \inf \Bigl\{r(x,t) \, \vert \, \,x\in K_0 \Bigr\} \quad \mbox{and} \quad d_0(t) := \sup \Bigl\{ d\left(\varphi_t^V(x),K_{\infty}\right) \, \vert \, \,x\in K_0 \Bigr\} \qquad \forall t \geq 0,
$$
where $d$ is a complete geodesic distance on $\mathcal{W}$ and $d(\cdot,K_{\infty})$ denotes the distance function to the set $K_{\infty}$.

\begin{claim}
Apart from shrinking $K_0$, we may assume that 
$$
\lim_{t \to \infty} r_0(t) = +\infty \quad \mbox{and} \quad \lim_{t\rightarrow +\infty} d_0(t) = 0.
$$
Moreover, there is $\delta >0$ such that for every $r>0$ and every pair of integers $k,l\geq 0$ with $k\neq l$ the sets $\varphi_{r+k\delta}^V(K_0)$ and $  \varphi_{r+l\delta}^V(K_0)$ are disjoint.
\label{cl:1}
\end{claim}
\begin{proof}[Proof of Claim \ref{cl:1}]
Let us show that for every $x \in K_0$ we have $\lim_{t \to \infty} r(x,t) =\infty$. Let $x_0 \in K_0$ be fixed. Since $f$ is positive, the function $t \geq 0 \mapsto r_0(x_0,t)$ is increasing so its limit at infinity is either  $+\infty$ or a positive real number $\alpha$. If $\lim_{t \to +\infty} r(x_0,t) =\alpha$, then $\frac{\partial r}{\partial t} (x_0,t)$ tends to zero at infinity and consequently by (\ref{eq:r0}), $\lim_{t \to +\infty} f(\varphi^{\hat{V}}_t(x_0))  = \lim_{s \uparrow \alpha} f(\varphi_{s}^V (x_0)  ) =0$ which contradicts (P4) and the fact that $f$ is smooth and positive. In conclusion, $r(x,t) \xrightarrow{t \to \infty} \infty$ point-wise over $K_0$. Therefore, by Egorov's Theorem \cite[Theorem 4.17]{wz77}, there exists a compact subset of $K_0$ of positive measure where $r(x,t) \xrightarrow{t \to \infty} \infty$ uniformly. So, apart from shrinking $K_0$, we conclude that $\lim_{t \to \infty} r_0(t) = +\infty$. To get $\lim_{t\rightarrow +\infty} d_0(t) = 0$, we observe that by (P5) the function $ x \mapsto d(\varphi_t^V(x))$ converges point-wise to zero as $t$ tends to $+\infty$ and we apply Egorov's Theorem. It remains to prove the second part. Since $\lim_{t\rightarrow +\infty} d_0(t) = 0$ there is $\delta >0$ such that  $d_0(t)<d_0(0)$ for any $t\geq \delta$. Let $r>0$ and  $k,l\geq 0$ be a pair of integers with $k< l$. If $\varphi_{r+l \delta}^V(y) = \varphi_{r+ k\delta}^V(x)$ for some pair  $x,y \in K_0$, then $
\varphi_{(l-k) \delta}^V(y) =x$ belongs to $K_0$ which means that $d_0((l-k)\delta)\leq d_0(0)$, a contradiction.
\end{proof}

By (\ref{9aout1}), we have 
$$
\mbox{div}_x^{\hat{v}} \bigl(\hat{V}\bigr) =  f(x) \, \mbox{div}^{v}_x(V) \qquad \forall x \in \mathcal{S} \setminus E,
$$
which, by Proposition \ref{PROPintdiv} yields for every $t\geq 0$,
\begin{multline*}
 \mbox{vol}^{\hat{v}}  \bigl(  \hat{K}_t \bigr) =      \mbox{vol}^{\hat{v}}  \left( \varphi_t^{\hat{V}}(K_0) \right) = \int_{K_0} \exp \left(   \int_0^t \mbox{div}_{\varphi_s^{\hat{V}}(x)}^{\hat{v}} (\hat{V}) \, ds    \right) \hat{v}(x)\\
 = \int_{K_0} \frac{1}{f(x)} \, \exp\left(\int_{0}^t  f\left(\varphi^{\hat{V}}_s(x)\right)   \, \mbox{div}_{\varphi^{\hat{V}}_s(x)}^{v}(V)  \, ds \right)  \,   v(x).
\end{multline*}
Since $f$ is smooth and  positive  on the compact set $K_0$, we infer that  there is $C>0$ such that  for every $t\geq 0$,
$$
 \mbox{vol}^{\hat{v}}  \bigl(  \hat{K}_t \bigr)  \leq  C \, \int_{K_0}  \exp\left(\int_{0}^t  f\left(\varphi^{\hat{V}}_s(x)\right)   \, \mbox{div}_{\varphi^{\hat{V}}_s(x)}^{v}(V)  \, ds \right)  \,   v(x),
$$
which by doing the change of coordinates $r = r(x,s)$ for every $x$ yields 
\begin{eqnarray}\label{10aout1}
 \mbox{vol}^{\hat{v}}  \bigl(  \hat{K}_t \bigr)  \leq  C \, \int_{K_0}  \exp\left(\int_{0}^{r(x,t)}  \mbox{div}_{\varphi^{V}_r(x)}^{v}(V)  \, dr \right)  \,   v(x) \qquad \forall t \geq 0.
\end{eqnarray}
Let $t>0$ be fixed, set $r_i(t):=r_0(t)+i \delta$ for every positive integer $i$ and  define the partition $\{K_0^i\}_i$ of $K_0$ by
$$
K_0^i := \Bigl\{ x \in K_0 \, \vert \, r(x,t) \in  \left[r_i(t),r_{i+1}(t) \right)  \Bigr\} \qquad \forall i \in \N.
$$
Then (\ref{10aout1}) gives
\begin{eqnarray*}
 \mbox{vol}^{\hat{v}}  \bigl(  \hat{K}_t \bigr)  &\leq & C \, \sum_{i=0}^{+\infty} \int_{K_0^i}  \exp\left(\int_{0}^{r(x,t)}  \mbox{div}_{\varphi^{V}_r(x)}^{v}(V)  \, dr \right)  \,   v(x) \\
 & =  & C \, \sum_{i=0}^{+\infty} \int_{K_0^i}  \exp\left(\int_{0}^{r_i(t)}  \mbox{div}_{\varphi^{V}_r(x)}^{v}(V)  \, dr \right)    \exp\left(\int_{r_i(t)}^{r(x,t)}  \mbox{div}_{\varphi^{V}_r(x)}^{v}(V)  \, dr \right)  \,   v(x)\\
 & = & C e^{\delta C} \, \sum_{i=0}^{+\infty} \int_{K_0^i}  \exp\left(\int_{0}^{r_i(t)}  \mbox{div}_{\varphi^{V}_r(x)}^{v}(V)  \, dr \right)  \, v(x),
\end{eqnarray*}
where we used that taking $C$ larger if necessary we may assume that $  \mbox{div}_{\varphi^{V}_r(x)}^{v}(V)$ is bounded from above by $C$ (because $K_{\infty}$ is relatively compact and $\lim_{t\rightarrow +\infty} d_0(t) = 0$). Consequently, using again Proposition \ref{PROPintdiv} we can write
$$
 \mbox{vol}^{\hat{v}}  \bigl(  \hat{K}_t \bigr)  \leq C e^{\delta C} \, \sum_{i=0}^{+\infty}  \mbox{vol}^{v}\left(\varphi_{r_i(t)} (K_0^i)  \right).
$$
By construction, all the sets $\varphi_{r_i(t)} (K_0^i)$ are disjoint and contained in the set $\cup_{s\geq r_0(t)}  \varphi_{s} (K_0)$. In conclusion, we have for every $t>0$,
$$
 \mbox{vol}^{\hat{v}}  \bigl(  \hat{K}_t \bigr)  \leq C e^{\delta C} \,  \mbox{vol}^{v} \left(   \cup_{s\geq r_0(t)}  \varphi_{s} (K_0)  \right)
$$
which by (\ref{cl:1}) tends to zero as $t$ tends to $+\infty$.

\section{Resolution of singularities in manifolds with corners}\label{secRS}

In this section we introduce the main ideas and results about resolution of singularities which are necessary for this manuscript. Our goal is to prove an embedded resolution of singularities in the category of analytic manifold with corners (Theorem \ref{THMhironaka} below) and to use this result to prove Proposition \ref{PROPresolution} in subsection \ref{SSECresolution}. An important technical point is Lemma \ref{lem:Orientabi} below, which motivates the use of manifolds with corners.

Theorem \ref{THMhironaka} seems to be well known, but we have not found a precise reference. In what follows, we present a simple proof (for the special case of a hypersurface), based on the functoriality of the usual resolution of singularities (we follow \cite{bm08} -  see also \cite{kollarbook,w05} and the references therein). Our argument can be adapted to any other notion of embedded resolution of singularities (e.g. coherent ideal sheaves and marked ideals), provided that this notion is also functorial in respect to open immersions (which preserve the corners - see definition below).





\subsection{Real-analytic manifold with corners}\label{SSECmanifoldcorners}

\paragraph{Manifold with corners.} In this section, we combine the presentations given in \cite[Section 1]{km15} and \cite[Sections 2.1 and 2.2]{panazzolo06}. Set
\[
\mathbb{R}_{+}:= [0,\infty) \quad \text{ and }\quad \mathbb{N} = \{0,1,2, \ldots\}
\]
The model ($n$-dimensional) manifold with corners is a product
\[
\mathbb{R}^{n,k}:= \mathbb{R}_{+}^k \times \mathbb{R}^{n-k} 
\]
for $k \in \{0, \ldots, n\}$, on which the smooth functions, forming the ring $C^{\infty}(\mathbb{R}^{n,k}) =C^{\infty}(\mathbb{R}^n) \Big|_{\mathbb{R}^{n,k}}$ are taken to be those obtained by restriction from the smooth functions on $\mathbb{R}^n$. We will also consider the sheaf of rings $\mathcal{O}_{\mathbb{R}^{n,k}}$ of analytic germs over $\mathbb{R}^{n,k}$, which is the restriction of the analytic sheaf $\mathcal{O}_{\mathbb{R}^n}$ over $ \mathbb{R}^{n,k}$. 

An $n$-dimensional manifold with corners $M$ is a (paracompact, Hausdorff) topological manifold with boundary, with a ring of smooth functions $C^{\infty}(M)$ with respect to which it is everywhere locally diffeomorphic to one of these model spaces (that we call a local chart). We say that $M$ is analytic if there is also a coherent sheaf $\mathcal{O}_{M}$ of analytic functions with respect to which $M$ is everywhere locally bi-analytic to one of these models. If there is no risk of ambiguity, the stalk of $\mathcal{O}_M$ at a point $x\in M$ is denoted by $\mathcal{O}_x$ (instead of $\mathcal{O}_{M,x}$). A morphism $\Phi : M \to N$ between two manifolds with corners is said to be smooth (resp. analytic) if $\Phi$ is smooth (resp. analytic) in each local chart.


\paragraph{Boundary set of the manifold with corners.} Each point $x \in M$ necessarily has a well-defined (boundary) codimension given by the number of independent non-negative coordinate functions vanishing at $x$ in a local chart, which we denote by $b(x)$ (it is independent of the choice of local chart). A boundary face of codimension $1$ is the closure of one of the connected components of the set of points of codimension 1 ({\it i.e.} of $\{x \in M :b(x)= 1\}$); the set of such faces is denoted $\mathcal{D}$. In particular $\mathcal{D}$ consists of boundary hypersurfaces. We also require two extra conditions on the set $\mathcal{D}$:
\begin{itemize}
\item[(i)] All boundary hypersurfaces $D \in \mathcal{D}$ are embedded analytic varieties (not necessarily connected). In particular, the ideal of smooth functions vanishing on $D$, denoted by $\mathcal{I}_{D}^{\infty} \subset C^{\infty}(M)$, is principal (see \cite[Section 1]{km15}). Therefore, there exists a non-negative generator $\rho_D \in C^{\infty}(M)$ of this ideal, {\it i.e.} $\mathcal{I}_{D}^{\infty} = \rho_D \cdot C^{\infty}(M)$.

\item[(ii)] The set $\mathcal{D}$ is finite and ordered, that is, we can write $\mathcal{D} = \{D^{(1)}, \ldots, D^{(l)}\}$, where $l \in \mathbb{N}$ and each $D^{(i)}$ is an embedded analytic subvariety of $\mathcal{D}$. 
\end{itemize}

From now on, we will denote by $(M,\mathcal{D})$ a manifold with corner.

\begin{remark}
We impose condition $(ii)$ (which is always trivially verified on any relatively compact subset of $M$) because the border of the manifold with corner will play a similar role to the exceptional divisors in resolution of singularities. We stress that this hypothesis could be relaxed, but it simplifies the presentation and it is trivially satisfied in our context.
\end{remark}

\paragraph{Open immersions compatible with corners.} We say that an analytic morphism $\phi: (M,\mathcal{D}) \to (N,\mathcal{E})$ is an open immersion compatible with corners (c.f \cite[p. 656, definition of b-map]{km15}), if $\phi$ is a local immersion with respect to each local model of $M$ and $N$ (in particular, $\phi$ can be locally extended to a neighborhood of the origin of $\mathbb{R}^{n,k}$ in each local model) and $\phi$ is compatible with the order of the boundaries, that is, $\phi(D^{(i)}) \subset E^{(j(i))}$ for some $j(i)$, and given $i_1<i_2$, we have that $j(i_1) < j(i_2)$.

\paragraph{Ideal sheaves over a manifold with corners.} Let $\mathcal{I}$ be a coherent ideal sheaf of $\mathcal{O}_M$. Given a point $x \in M$ and an ideal $J \subset \mathcal{O}_x$, the order of $\mathcal{I}$ in respect $J$ is:
\[
\mu_J(\mathcal{I}) := \max\{ t\in \mathbb{N}; \,\mathcal{I} \cdot \mathcal{O}_x \subset J^t   \}
\]
In particular, the order of $\mathcal{I}$ at $x$ is given by $\mu_x(\mathcal{I}):=\mu_{m_x}(\mathcal{I})$, where $m_x$ denotes the maximal ideal of the local ring $\mathcal{O}_x$. The support of $\mathcal{I}$, which we denote by $Supp(\mathcal{I})$, is the set $\{ x \in M; \,\mu_x(\mathcal{I}) \geq 1 \}$. Given a connected sub-manifold $\mathcal{C}$, the order $\mu_{\mathcal{C}} (\mathcal{I})$ denotes the generic order of $\mathcal{I}$ in respect to the reduced ideal sheaf $\mathcal{I}_{\mathcal{C}}$ whose support is $\mathcal{C}$, that is, $\mu_{\mathcal{C}} (\mathcal{I}):= min\{ \mu_{\mathcal{I}_{\mathcal{C}}}(\mathcal{I}_x) ; \, \forall x\in \mathcal{C}  \}$.

Under the hypothesis that $\mathcal{I}$ is a principal ideal sheaf, we define the singular set of $\mathcal{I}$, which we denote by $Sing(\mathcal{I})$, as the set of points of order at least two, that is $Sing(\mathcal{I}) = \{ x \in M; \,\mu_x(\mathcal{I}) \geq 2 \}$.


\subsection{Real-blowings-up}

We follow here \cite[Sections 2.3, 2.4 and 2.5]{panazzolo06}.

\paragraph{Real blowing-up of $\mathbb{R}^n$.} We start by defining real-blowings-up over local charts outside of the corners. Consider a coordinate system $(u_1, \ldots, u_n)$ of $\mathbb{R}^n$. A real blowing-up of $\mathbb{R}^n$ with center $\mathcal{C} = (u_1,\dots,u_t)$ ({\it i.e.} $\mathcal{C}$ is the support of the ideal $(u_1,\dots,u_t)$), is the real-analytic surjective map
\[
\begin{matrix}
\sigma :\, \mathbb{S}^{t-1} \times \mathbb{R}_{+} \times \mathbb{R}^{n-t} &\longrightarrow &\mathbb{R}^n\\
 (\bar{x},\tau,x) & \longmapsto & (\tau \cdot \bar{x}, x)
\end{matrix}
\]
where $\mathbb{S}^{t-1}:=\{\bar{x}\in \mathbb{R}^t; \, \|\bar{x}\|=1\}$ and $x= (x_{t+1}, \ldots, x_{n})$. The set $\widetilde{M}:=\mathbb{S}^{t-1} \times \mathbb{R}_{+} \times \mathbb{R}^{n-t} $ will be called the blowed-up space, $\mathcal{C}$ will be called the center of blowing-up, and $E:= \sigma^{-1}(\mathcal{C}) = \mathbb{S}^{t-1} \times \{0\} \times \mathbb{R}^{n-t} $ will be called the exceptional divisor of $\sigma$. Note that $\sigma$ is a bi-analytic morphism from $\widetilde{M}\setminus E$ to $\mathbb{R}^n \setminus \mathcal{C}$.  

\paragraph{Directional charts.} In the notation of the previous paragraph, there are $2t$ local charts covering $\widetilde{M} = \mathbb{S}^{t-1} \times \mathbb{R}_{+} \times \mathbb{R}^{n-t}$. Indeed, for each index $j \in \{1, \ldots, t\}$, we consider the $u_j^{+}$-chart and the $u^{-}_j$-chart, whose expressions are given directly in terms of the application $\sigma$:
$$
\sigma_{j}^{+}: \,\mathbb{R}^{j-1} \times \mathbb{R}_{+} \times \mathbb{R}^{n-j} \to \mathbb{R}^n \cap \{u_j\geq 0\}  \mbox{ and } \sigma_{j}^{-}: \,\mathbb{R}^{j-1} \times \mathbb{R}_{+} \times \mathbb{R}^{n-j} \to \mathbb{R}^n \cap \{u_j\leq 0\},
$$
which are defined as follows:
\[
\sigma^{\epsilon}_j : \left\{
\begin{aligned}
u_i &= x_i \cdot x_j \quad \text{ if } i=1, \ldots, j-1, j+1, \ldots, t\\
u_j &= \epsilon  x_j \phantom{ \cdot x_j} \\
u_i &= x_i \phantom{\cdot x_j}\,\,\, \quad \text{ if } i=t+1, \ldots, n,
\end{aligned}\right.
\]
where $\epsilon \in\{+,-\}$. Note that there exist bi-analytic applications \cite[Proposition 2.11]{panazzolo06}:
\[
\begin{aligned}
\tau_{j}^{+}:& \,\mathbb{R}^{j-1} \times \mathbb{R}_{+} \times \mathbb{R}^{n-j} \to 
\left(\mathbb{S}^{t-1} \cap \{\bar{x}_j >0\}  \right) \times \mathbb{R}_{+} \times \mathbb{R}^{n-t} \\
\text{and} \quad \tau_{j}^{-}:& \,\mathbb{R}^{j-1} \times \mathbb{R}_{+} \times \mathbb{R}^{n-j} \to \left(\mathbb{S}^{t-1} \cap \{\bar{x}_j <0\}  \right) \times \mathbb{R}_{+} \times \mathbb{R}^{n-t}
\end{aligned}
\]
such that $\sigma_{j}^{+} = \sigma \circ \tau_{j}^{+}$ and $\sigma_{j}^{-} = \sigma \circ \tau_{j}^{-}$ (this is why we abuse notation and call $\sigma_{j}^{\pm}$  a chart). In particular, note that $\widetilde{M}= \mathbb{S}^{t-1} \times \mathbb{R}_{+} \times \mathbb{R}^{n-t}$ is a manifold with corner and $E=\mathbb{S}^{t-1} \times \{0\} \times \mathbb{R}^{n-t}$ is its border, which we call exceptional divisor.

\paragraph{Blowing-up in general manifolds with corners.} First of all, the definition of real blowing-up can be easily extended to the local models $\mathbb{R}^{n,k}$, as long as the center of blowing-up is admissible to the border: we say that $\mathcal{C}$ is an admissible center (and that the blowing-up $\sigma$ is admissible) in respect to $\mathbb{R}^{n,k}$ if $\mathcal{C}= (u_{i_1}, \ldots, u_{i_t})$, for some sublist of indexes $[i_1, \ldots, i_t]\subset [1, \ldots, n]$. In other words, $\mathcal{C}$ has normal crossings in respect to the border $\mathcal{D}$. In this case, an admissible real blowing-up is a proper analytic map between manifolds with corners $\sigma: \widetilde{M} \to \mathbb{R}^{n,k}$ and all of the notions given before (including the directional charts) can be adapted, {\it mutatis mutandis}. It is now clear that we can extend these definitions to any general manifold with corners (provided that $\mathcal{C}$ is adapted in respect to $(M,\mathcal{D})$, {\it i.e.} $\mathcal{C}$ is admissible in every local chart). See details in \cite[Proposition 2.12]{panazzolo06}.

Given an admissible real-blowing-up $\sigma: (\widetilde{M},\widetilde{\mathcal{D}})\to (M,\mathcal{D})$, the set of divisors $\widetilde{\mathcal{D}}$ over $\widetilde{M}$ is given by the strict transforms of $D \in \mathcal{D}$ ({\it i.e.} by the sets $\overline{\sigma^{-1}\left(D \setminus \mathcal{C} \right)}$) and the set $E$, where $E$ is the extra border produced by $\sigma$ (which enters in the end of the list, as the $l+1$ term). Finally, given a coherent ideal sheaf $\mathcal{I}$ over $M$, we define the (total) transform of $\mathcal{I}$ as the ideal sheaf $\mathcal{I}^{\ast}:=\mathcal{I} \cdot \mathcal{O}_{\widetilde{M}}$, {\it i.e.} the ideal sheaf which is locally generated by all functions $f\circ \sigma$, where $f\in \mathcal{I} \cdot \mathcal{O}_{M,x}$ for some point $x \in M$.

\begin{lemma}
Consider an admissible real blowing-up $\sigma: (\widetilde{M},\widetilde{\mathcal{D}}) \to (M,\mathcal{D})$ with exceptional divisor $E$. There exists a non-negavite smooth function $\rho_E: \widetilde{M} \to \mathbb{R}$ whose support is contained in $E$ (that is, if $x\notin E$, then $\rho_E(x) > 0$). Furthermore, at each point $x\in E$ there exists a coordinate system $(x_1, \ldots, x_n)$ centered at $x$ where $E = (x_1=0)$, and a locally defined smooth unit $\xi$ (that is $\xi(x) \neq 0$) such that $\rho_E = x_1 \cdot \xi$.
\label{lem:rho}
\end{lemma}
\begin{proof}[Proof of Lemma \ref{lem:rho}]
Note that $E$ is an embedded variety, so there exists a smooth function $\rho_E : \widetilde{M} \to \R$ which generates the smooth ideal $\mathcal{I}^{\infty}_{E}$. This means that $E = (\rho_E=0)$ and that $\rho_E$ divides any smooth function whose support contains $E$.

Since the zero locus of $\rho_E$ is in the border, we can suppose that $\rho_E$ is strictly positive outside of $E$. Next, given a point $x\in E$ and a coordinate system $(x_1, \ldots, x_n)$ centered at $x$ and defined in a neighborhood $\mathcal{V}$ where $E\cap \mathcal{V} = (x_1=0)$, we consider the locally defined smooth function $f=x_1$. Since $\rho_E$ divides $f$, by Malgrange division Theorem, there exists a function $\xi(x)$ such that $\rho_E  = x_1 \cdot\xi $. Finally, since $f=x_1$ is a local generator of the ideal $\mathcal{I}^{\infty}_{E}$, we conclude that $\xi$ is a unit. 
\end{proof}

\begin{lemma}[Preserving orientability]
Let $(M,\mathcal{D})$ be an orientable manifold with corners and $\omega$ be a fixed $C^{\infty}$ volume form over $M$ which defines the orientation of $M$. Consider an admissible real blowing-up $\sigma: (\widetilde{M},\widetilde{\mathcal{D}}) \to (M,\mathcal{D})$ with center $\mathcal{C}$ and denote by $E$ its exceptional divisor. Then $(\widetilde{M},\widetilde{\mathcal{D}})$ is also an orientable manifold with corners and the differential form:
\[
\widetilde{\omega}:= \rho_E^{-\beta} \,d\sigma^{\ast}(\omega)
\]
is a volume form over $\widetilde{M}$, where $\beta = codim(\mathcal{C})-1$ and $\rho_E$ is the function defined in Lemma \ref{lem:rho}.

\label{lem:Orientabi}
\end{lemma}
\begin{proof}[Proof of Lemma \ref{lem:Orientabi}]
It is clear that it is enough to prove the result locally over the pre-image of $\mathcal{C}$ (since the construction of $\widetilde{\omega}$ is global, $\sigma$ is an isomorphism outside of $\mathcal{C}$ and $\rho_E$ is strictly positive outside of $E$). So, let $a\in \mathcal{C}$ and $b \in \sigma^{-1}(a)$ be fixed and consider an analytic local chart $u= (u_1, \ldots, u_n)$ over $a$ where $\mathcal{C} = (u_1, \ldots, u_r)$ and $b$ is the origin of the $u_1^{+}$-chart
\[
u_1 =  x_1, \quad u_i= x_i \cdot x_1 \text{ if }i=2, \ldots, t, \quad u_i=x_i \text{ if } i=t+1, \ldots, n
\]
In this coordinate system, $\omega = \zeta \cdot du_1 \wedge \ldots \wedge du_n$, where $\zeta$ is a $C^{\infty}$ unit ({\it i.e.} $\zeta(a)\neq 0$). Now, the pull-back of $\omega$ under these local expressions is given by:
\[
d\sigma^{\ast}(\omega)  = \zeta \circ \sigma \cdot x_1^{\beta} \, dx_1 \wedge \ldots \wedge dx_n,
\]
where $\beta = t-1 = codim(\mathcal{C})-1$. Since $\rho_E$ is locally equal to $x_1 \xi$, where $\xi$ is a $C^{\infty}$ locally defined positive unit, we conclude that:
\[
\widetilde{\omega} = \zeta \circ \sigma \cdot \xi^{-\beta} dx_1 \wedge \ldots \wedge dx_n
\]
and $\zeta \circ \sigma \cdot \xi^{-\beta} (b) \neq 0$. Therefore, $\widetilde{\omega}$ is a locally defined volume form. We can easily conclude.
\end{proof}

\subsection{Resolution of singularities of a hypersurface in a manifold with corners}

Let $\sigma: (\widetilde{M},\widetilde{D}) \to (M,D)$ be a real-blowing-up with center $\mathcal{C}$. Given a coherent and principal ideal sheaf $\mathcal{I}$, set $\alpha:= \mu_{\mathcal{C}}(\mathcal{I})$. There exists a coherent ideal sheaf $\mathcal{I}^{st}$, called the strict transform of $\mathcal{I}$, and (if $\alpha\geq 1$) there exists a coherent ideal sheaf $\mathcal{I}'$, called the weighted transform (with weight 1) of $\mathcal{I}$, such that
\[
\mathcal{I}^{\ast}=:  \mathcal{I}_E^{\alpha} \, \mathcal{I}^{st} \quad \text{ and } \quad \mathcal{I}^{\ast}=:  \mathcal{I}_E \, \mathcal{I}' (\text{ if }\alpha\geq 1)
\]
where $\mathcal{I}_E$ is the reduced ideal sheaf whose support is the exceptional divisor $E$. A sequence of strict (resp. weighted) admissible real-blowings-up is given by
\[
(M_r,\mathcal{I}_r,D_r) \xrightarrow{\sigma_r} \ldots  \xrightarrow{\sigma_2} (M_1,\mathcal{I}_1,D_1)\xrightarrow{\sigma_1} (M_0,\mathcal{I}_0,D_0)
\]
where each successive real-blowing-up $\sigma_{i+1}: (M_{i+1},D_{i+1} )\to (M_i,D_i)$ is admissible and the ideal sheaf $\mathcal{I}_{i+1}$ denotes the strict (resp. weighted) transform of $\mathcal{I}_i$. We first prove a preliminary result about weighted sequence of blowings-up:

\begin{proposition}
Let $(M,\mathcal{D})$ be an analytic real-manifold with corners and consider a coherent and principal ideal sheaf $\mathcal{I}$ over $M$. For every relatively compact open set $M_0 \subset M$, there exists a sequence of weighted admissible real-blowings-up with smooth centers
\begin{equation}
(M_r,\mathcal{I}_r,D_r) \xrightarrow{\sigma_r} \ldots  \xrightarrow{\sigma_2} (M_1,\mathcal{I}_1,D_1)\xrightarrow{\sigma_1} (M_0,\mathcal{I}_0,D_0)
\label{eq:res}
\end{equation}
(where $D_0 = \mathcal{D} \cap M_0$, $\mathcal{I}_0 = \mathcal{I} \cdot \mathcal{O}_{M_0}$ and $\mathcal{I}_{i+1}$ is the weighted transform of $\mathcal{I}_i$) such that $\mathcal{I}_r$ is a principal and reduced ideal sheaf whose support is a non-singular submanifold (with normal crossings with $D_r$), the morphism $\sigma = \sigma_1 \circ \ldots \circ \sigma_r$ restricted to $\sigma^{-1}(M_0 \setminus Sing(\mathcal{I}_0))$ is an isomorphism and $\sigma(E_r) = Sing(\mathcal{I}_0)$. Furthemore, a resolution sequence \ref{eq:res} can be associated to every ideal $\mathcal{I}$ in a way that is functorial with respect to open immersions compatible with corners.
\label{THMhironakaprel}
\end{proposition}

\begin{proof}[Proof of Proposition \ref{THMhironakaprel}]
Because of the functorial property, it is enough to assume that $M_0$ is a relatively compact open set of one of the local models $\mathbb{R}^{n,k}$ of $M$ (see e.g. \cite[Claim 5.1]{bm08}). By the definition of an analytic manifold with corners, there exists a relatively compact open set $\widetilde{M}_0$ of $\mathbb{R}^n$, which contains $M_0$ and where the ideal $\mathcal{I}_0$ admits an analytic extension $\widetilde{\mathcal{I}}_0$. We consider the triple $(\widetilde{M}_0,\widetilde{\mathcal{I}}_0,\widetilde{D}_0)$, where $\widetilde{D}_0 = D_0 = \{D^{(1)}, \ldots, D^{(l)}\}$ is the locally ordered exceptional divisor. 

The proof of the Proposition follows from (usual) resolution of singularities of marked ideals \cite[Theorem 1.3]{bm08} (applied to the marked ideal $\underline{\mathcal{I}}_0:=(\widetilde{M}_0,\widetilde{M}_0,\widetilde{\mathcal{I}}_0, \widetilde{D}_0,1)$, in the notation of \cite{bm08}). Indeed, by \cite[Theorem 1.3]{bm08}, there exists sequence of weighted admissible (polar) blowings-up:
\[
(\widetilde{M}_{r+1},\widetilde{\mathcal{I}}_{r+1},\widetilde{D}_{r+1}) \xrightarrow{\tau_{r+1}} \ldots  \xrightarrow{\tau_2} (\widetilde{M}_1,\widetilde{\mathcal{I}}_1,\widetilde{D}_1)\xrightarrow{\tau_1} (\widetilde{M}_0,\widetilde{\mathcal{I}}_0,\widetilde{D}_0)
\]
such that $\widetilde{\mathcal{I}}_{r+1}$ is the structural sheaf (that is, it is a everywhere generated by a unit). Now, note that if $x$ is a point contained in $Supp(\mathcal{I}_0) \setminus Sing(\mathcal{I}_0)$, then the invariant $inv$ defined by Bierstone and Milman (\cite[Theorem 7.1]{bm08} - see also \cite[p. 256]{bm97}) is minimal and by \cite[Theorem 7.1(2)]{bm08}, the only blowing-up with a center containing $x$ must be $\tau_{r+1}$. Therefore, by construction $\widetilde{\mathcal{I}_r}$ is a principal and reduced ideal sheaf whose support is a non-singular submanifold (with normal crossings with $\widetilde{D}_r$), the morphism $\tau = \tau_1 \circ \ldots \circ \tau_r$ restricted to $\tau^{-1}(M_0 \setminus Sing(\mathcal{I}_0))$ is an isomorphism and $\tau(D_r) = Sing(\mathcal{I}_0)$, c.f. \cite[Section 12]{bm97}. Furthermore, the resolution sequence is functorial in respect to open immersions 

\begin{claim}
The sequence of polar blowings-up $\vec{\tau}$ gives rise to a sequence of real-analytic blowings-up:
\[
(M_r,\mathcal{I}_r,D_r) \xrightarrow{\sigma_r} \ldots  \xrightarrow{\sigma_2} (M_1,\mathcal{I}_1,D_1)\xrightarrow{\sigma_1} (M_0,\mathcal{I}_0,D_0)
\]
such that $\mathcal{I}_r$ is a principal and reduced ideal sheaf whose support is a non-singular submanifold (with normal crossings with $D_r$), the morphism $\sigma = \sigma_1 \circ \ldots \circ \sigma_r$ restricted to $\sigma^{-1}(M_0 \setminus Sing(\mathcal{I}_0))$ is an isomorphism and $\sigma(D_r) = Sing(\mathcal{I}_0)$. Furthermore the resolution sequence is functorial in respect to open immersions compatible with corners.
\end{claim}

The Proposition follows from this claim, which we prove by induction on the height $r$ of the sequence of blowings-up $\vec{\tau}$. Indeed, if $r=0$, then the result is obvious. So assume that the claim is proved for every sequence of blowings-up $\vec{\tau}'$ with height $r'< r$, and consider a sequence of blowings-up $\vec{\tau}$ with height $r$.

The first blowing-up $\tau_1$ has center $\widetilde{\mathcal{C}}_1$ adapted to the exceptional divisor $\widetilde{D}_0$. In other words, $\widetilde{\mathcal{C}}_1 = (u_{i_1}, \ldots, u_{i_t})$ for some list of indexes $[i_1, \ldots, i_t] \subset [1,\ldots, n]$. We consider $\mathcal{C}_1 = \widetilde{\mathcal{C}}_1$ as the center of blowing-up of the first real blowing-up $\sigma$, in order to get:
\[
(M_1,\mathcal{I}_1,D_1) \xrightarrow{\sigma_1} (M_0,\mathcal{I}_0,D_0)
\]
We argue over the $\sigma_1^{+}$-chart (since the treatment in each chart is analogous):
\[
\left\{
\begin{aligned}
u_1 &= x_1,\\
u_i &= x_1 \cdot x_i, \quad i=2, \ldots, t\\
u_i &= x_i, \phantom{\cdot x_i} \quad \,\,\,i=t+1 \ldots, n,
\end{aligned}
\right.
\]
where $(u_1, \ldots, u_n)$ is a coordinate system of $\mathbb{R}^{n,k}$ where $D_0 = \{\prod_{j=1}^k u_{i_j}=0 \}$ for some collection of indexes $[i_1, \ldots, i_k] \subset [1,\ldots, n]$; and $x_1\geq 0$. We further consider the $\tau_1$-chart of the polar blowing-up $\tau$:
\[
\left\{
\begin{aligned}
u_1 &= y_1,\\
u_i &= y_1 \cdot y_i, \quad i=2, \ldots, t\\
u_i &= y_i, \phantom{\cdot y_i} \quad \,\,\,i=t+1 \ldots, n
\end{aligned}
\right.
\]
and note that the there exists a trivial open immersion from the $\sigma_1^{+}$-chart to the $\tau_1$-chart. Next, consider the restriction of the sequence of (polar) blowings-up $(\tau_r, \ldots, \tau_2)$ to the $\tau_1$-chart, which has height at most $r-1<r$ (some of the blowings-up might be isomorphisms when restricted to this chart). By the induction hypothesis, this sequence gives rise to a sequence of real blowings-up $(\sigma^{1,+}_r, \ldots, \sigma^{1,+}_2)$ over the $\sigma_1^{+}$-chart, which is functorial in respect to open immersions compatible with corners. By the functoriality property, it is clear that the sequences $(\sigma^{j,\pm}_r, \ldots, \sigma^{j,\pm}_2)$ glue, which gives rise to a sequence of weighted admissible real-blowings-up $(\sigma_r, \ldots, \sigma_2)$ defined over $(M_1,\mathcal{I}_1,D_1)$. In order to conclude, we just need to remark that the first blowing-up $\sigma_1$ is also functorial in respect to open immersions compatible with corners, since $\tau_1$ is the first blowing-up of a sequence of blowings-up which is functorial. This proves the Claim and the Proposition.
\end{proof}

\begin{theorem}\label{THMhironaka}
Let $(M,\mathcal{D})$ be an analytic real-manifold with corners and consider a coherent and principal ideal sheaf $\mathcal{I}$ over $M$. For every relatively compact open set $M_0 \subset M$, there exists a sequence of strict admissible real-blowings-up with smooth centers
\[
(M_r,\mathcal{I}_r,D_r) \xrightarrow{\sigma_r} \ldots  \xrightarrow{\sigma_2} (M_1,\mathcal{I}_1,D_1)\xrightarrow{\sigma_1} (M_0,\mathcal{I}_0,D_0)
\]
(where $D_0 = \mathcal{D} \cap M_0$, $\mathcal{I}_0 = \mathcal{I} \cdot \mathcal{O}_{M_0}$ and $\mathcal{I}_{i+1}$ is the strict transform of $\mathcal{I}_i$) such that $\mathcal{I}_r$ is a principal and reduced ideal sheaf whose support is a non-singular submanifold (with normal crossings with $D_r$), the morphism $\sigma = \sigma_1 \circ \ldots \circ \sigma_r$ restricted to $\sigma^{-1}(M_0 \setminus Sing(\mathcal{I}_0))$ is an isomorphism and $\sigma(E_r) = Sing(\mathcal{I}_0)$. Furthermore, denoting by $\alpha_i : = \mu_{\mathcal{C}_{i}}(\mathcal{I}_{i-1})$ the order of $\mathcal{I}_{i-1}$ in respect to the center $\mathcal{C}_{i}$, and by $\mathcal{J}_{E_i}$ the pull-back via $\sigma$ of the reduced ideal sheaf $\mathcal{I}_{E_i}$ in $M_i$ whose support is the exceptional divisor $E_i$:
\begin{equation}\label{eq:hironaka}
\mathcal{I} \cdot \mathcal{O}_{M_r} =  \mathcal{I}_r \prod_{i=1}^r  \mathcal{J}_{E_i}^{\alpha_i}.
\end{equation}
\end{theorem}

\begin{proof}[Proof of Theorem \ref{THMhironaka}]
Consider the sequence of admissible real blowings-up \eqref{eq:res} given by Proposition \ref{THMhironakaprel}, but let $\mathcal{I}_{i+1}$ be the strict transform of $\mathcal{I}_i$. It is clear that $\mathcal{I}_r$ is a principal and reduced ideal sheaf whose support is a non-singular submanifold (with normal crossings with $D_r$), the morphism $\sigma = \sigma_1 \circ \ldots \circ \sigma_r$ restricted to $\sigma^{-1}(M_0 \setminus Sing(\mathcal{I}_0))$ is an isomorphism and $\sigma(E_r) = Sing(\mathcal{I}_0)$. Now, since $M_0$ is a relatively compact set, without loss of generality we may assume that each blowing-up of the sequence has a connected center. Equation \eqref{eq:hironaka} now follows from the definition of strict transform.
\end{proof}

\subsection{Proof of Proposition \ref{PROPresolution}}\label{SSECresolution}

We follow the notation of section \ref{ssecOverview}. Let $\mathcal{I}$ denote the ideal sheaf in $\mathcal{V}$ generated by the function $h$ (in particular, $\Sigma_{\Delta} \cap  \mathcal{V}= (\Sigma \cap \mathcal{V}, \mathcal{O}_{\mathcal{V}}/\mathcal{I})$). Note that $\mathcal{V}$ is a manifold with corners, where the boundary $D_0$ is empty. By Theorem \ref{THMhironaka} there exists a sequence of blowings-up
\[
(\mathcal{V}_r,\mathcal{I}_r,D_r) \xrightarrow{\sigma_r} \ldots  \xrightarrow{\sigma_2} (\mathcal{V}_1,\mathcal{I}_1,D_1)\xrightarrow{\sigma_1} (\mathcal{V},\mathcal{I},D)
\]
such that the support of $\mathcal{I}_r$ is a non-singular analytic surface. We take $\mathcal{W}:=\mathcal{V}_r$ and we denote by $E$ the border set $D_r$ and by $\widetilde{\mathcal{I}} = \mathcal{I}_r$. The morphism $\sigma: \mathcal{W} \to \mathcal{V}$ is the composition of blowings-up $\sigma:= \sigma_1 \circ \ldots \circ \sigma_r$, and the number $r$ of the enunciate of Proposition \ref{PROPresolution} is the number of blowings-up. The morphism $\sigma$ is proper (since it is a composition of proper morphisms), the restriction of $\sigma$ to $\mathcal{W}\setminus E$ is a diffeomorphism onto its image $\mathcal{V} \setminus \mbox{Sing}(\Sigma_{\Delta})$ and $\sigma(E)= \mbox{Sing} (\Sigma_{\Delta})$, since $\mbox{Sing} (\Sigma_{\Delta})\cap \mathcal{V} = Sing(\mathcal{I}_0)$. Note also that $\mathcal{W}$ is a Riemannian manifold and it is orientable by Lemma \ref{lem:Orientabi}. This concludes part (i).

Let $E_i$ denote the exceptional divisor associated to $\sigma_i$. Consider the function $\rho_{E_i}$ defined in Lemma \ref{lem:rho}, and denote by $\rho_i:= \rho_{E_i} \circ \sigma_{i+1} \circ \ldots \circ \sigma_r$. It is clear that $\rho_i$ is positive everywhere outside of the border $E$. Next, we note that the function $h\circ \sigma$ is a generator of the pulled back ideal sheaf $\mathcal{I} \cdot \mathcal{O}_{\mathcal{W}}$, and by the definition of $\rho_i$ and equation \eqref{eq:hironaka}, we have that 
\[
\alpha : = \prod \rho_i^{\alpha}\quad  \text{ divides } \quad h\circ \sigma.
\]
We denote by $\tilde{h}$ the smooth function given by the division (which exists by the Malgrange division Theorem). Since this is a generator of the strict transform $\widetilde{\mathcal{I}}$, we conclude that $\tilde{\Sigma} := (\tilde{h}=0)$ is an analytic sub-manifold, which concludes part (ii). Assertion (iii) of Proposition \ref{PROPresolution} is immediate from the definition of $\rho_i$ and Lemma \ref{lem:Orientabi} (where $\beta_i:= codim(\mathcal{C}_i)-1$). 

We prove part (iv) via an induction argument which controls the transform of the vector field $\mathcal{Z}$ at each step of the blowing-up sequence. Through the proof, we use the following auxiliary definition:

\begin{definition}[$\mathcal{S}$-adapted singular distributions]
Let $\Delta$ be a singular distribution of generic rank two and $\mathcal{S}$ be an analytic subset of a three dimensional manifold $M$. We say that $\Delta$ is $\mathcal{S}$-adapted if at all points $x$ in the smooth part of $\mathcal{S}$, that is, at all points $x$ where there exists a neighborhood $U_x$ of $x$ where $\mathcal{S} \cap U_x$ is a smooth analytic manifold, one of the following conditions hold:
\begin{itemize}
\item If $\mathcal{S} \cap U_x$ is a $2$-dimensional manifold, then there exists an analytic function $f: U_x \to \mathbb{R}$ such that $\mathcal{S}\cap U_x = (f=0)$ and $X \cdot f \in (f)$, for all $\,X \in \underline{\Delta}|_{U_x}$.
\item If $\mathcal{S} \cap U_x$ is a $1$-dimensional manifold, then there exists two analytic functions $f,g: U_x \to \mathbb{R}$ such that $\mathcal{S} \cap U_x = (f=g=0)$ and
\[
 det \left[ \begin{matrix} X \cdot f & X \cdot g\\ Y\cdot f & Y \cdot g \end{matrix}   \right] \in  (f,g), \, \text{ for all } \, X,Y \in \underline{\Delta}|_{U_x}.
\]
\end{itemize}
\end{definition}

Through our proof, we show that the (singular) distribution $\Delta$ and its transforms are adapted in respect to $Sing(\Sigma_{\Delta})$ and its transforms. This will imply that the vector-field $Z$ defined in Proposition \ref{PROPresolution}(iv) is tangent to $E = \sigma^{-1}(Sing(\Sigma_{\Delta}))$, which is a locally finite union of $2$-dimensional analytic manifolds. We start by the following claim:

\begin{claim}\label{cl:step3a}
Let $\mathcal{S}_0:=Sing(\Sigma_{\Delta})$ be the set of singularities of the Martinet analytic space $\Sigma_{\Delta}$. Suppose that $\Delta(x)$ generates the tangent space of the singular set $\mathcal{S}_0$, that is, $\Delta(y) \cap T_y \, \mathcal{S}_0 = T_y \, \mathcal{S}_0$ for all $y\in \mathcal{S}_0 $. Then, the non-holonomic distribution $\Delta$ is $\mathcal{S}_0$-adapted.
\end{claim}
\begin{proof}[Proof of Claim \ref{cl:step3a}]
Recall that $\mathcal{S}_0$ is of codimension $2$, and fix a point $x\in \mathcal{S}_0$ and a neighborhood $\mathcal{V}$ of $x$ where $\mathcal{S}_0 \cap \mathcal{V}$ is locally a smooth curve. Without loss of generality, there exists two vector-fields $X$ and $Y$ which generate $\underline{\Delta}|_{\mathcal{V}}$ and a coordinate system $(x_1,x_2,x_3)$ defined on $\mathcal{V}$ and centered at $x$ such that $\mathcal{S}_0 \cap \mathcal{V}= (x_2 = x_3 =0)$. By hypothesis:
\[
\Delta(y) \supset T_y \, \mathcal{S}_0= \vec{x}_1, \quad \forall \, y \in \mathcal{S}_0 \cap \mathcal{V},
\] 
therefore, without loss of generality we can assume that
\[
X = \partial_{x_1} + A \partial_{x_2} + B\partial_{x_3},
\]
where $A(x_1,0,0) = B(x_1,0,0) =0$, which implies that $X \cdot x_2 = A \in (x_2,x_3)$ and $X \cdot x_3 =B \in (x_2,x_3)$. The Claim follows easily.
\end{proof}

Now, we show that $Z$ defined in Proposition \ref{PROPresolution} has removable singularities on $E$. We recall that, by equation \eqref{EQZ3}, the pulled back vector field $\mathcal{Z}^{\ast}$ satisfies equation \eqref{10aout2}, that is:
\[
\begin{aligned}
\mathcal{Z}^{\ast} = \tilde{\mathcal{Z}} + \tilde{h} \, \tilde{W}, \quad \text{ where }\quad & \tilde{\mathcal{Z}} := \alpha \big[ (X^{\ast} \cdot \tilde{h} )  Y^{\ast} - (Y^{\ast} \cdot \tilde{h}  ) X^{\ast}  \big],\\
& \tilde{W}:=(X^{\ast} \cdot \alpha )  Y^{\ast} - (Y^{\ast} \cdot \alpha  ) X^{\ast},\\
& X^{\ast} := d\sigma^{-1}(X), \text{ and } Y^{\ast} := d\sigma^{-1}(Y).
\end{aligned}
\]
In particular, the vector-field $Z$, given by the expression of the enunciate of Proposition \ref{PROPresolution}, is given by the following expression on $\tilde{\Sigma} \setminus E$:
\begin{equation}
\begin{aligned}
Z: = \frac{\beta}{\alpha} \mathcal{Z}^{\ast} = \frac{\beta}{\alpha} \tilde{\mathcal{Z}} = \beta \,& \big[ (X^{\ast} \cdot \tilde{h} )  Y^{\ast} - (Y^{\ast} \cdot \tilde{h}  ) X^{\ast}  \big] \\
=  \prod_{i=1}^r \rho_i^{\beta_i} &\big[ (X^{\ast} \cdot \tilde{h} )  Y^{\ast} - (Y^{\ast} \cdot \tilde{h}  ) X^{\ast}  \big].
\end{aligned}
\label{eq:step3a}
\end{equation}

Note that the vector-fields $X^{\ast}$ and $Y^{\ast}$ have poles over $E$. In what follows, we show that the exponents $\beta_i$ (which are equal to $codim(\mathcal{C}_i)-1$) are sufficient to compensate the poles. Our proof is done by induction on the height $k$ of the sequence of blowings-up. More precisely, fix a number $k\leq r$ and consider the following objects defined at each $k$-step:
\begin{itemize}
\item Let $\sigma^{k}:= \sigma_1 \circ \ldots \circ \sigma_k: \mathcal{V}_k \to \mathcal{V}_0 = \mathcal{V}$ be the composition of the first $k$ blowings-up, and denote by $E^k$ the exceptional divisor related to the entire sequence $\sigma^{k}$ (in particular, $E^r = E$). 
\item Consider the singular analytic distribution $\Delta_k$ of generic rank two over $\mathcal{V}_k$ which is defined recursively by
\[
\Delta_0 := \Delta|_{\mathcal{V}} \quad \text{ and } \quad \Delta_k := d\sigma^{\ast}(\Delta_{k-1}) \cap Der_{\mathcal{V}_k}(-log\,E_k),
\]
where $Der_{\mathcal{V}_k}(-log\,E_k)$ is the sub-sheaf of analytic derivations in $\mathcal{V}_k$ which are tangent to the exceptional divisor $E_k$ of $\sigma_k$ (we note that $\Delta_k$ is a coherent sub-sheaf of $Der_{\mathcal{V}_k}$ - and therefore a singular distribution - by \cite[Chapter II Corollary 6.8]{Tougeron}).
\item Consider the analytic set $\mathcal{S}_k$ over $\mathcal{V}_k$ which is defined recursively by
\[
\mathcal{S}_0 := Sing(\Sigma_{\Delta}) \cap \mathcal{V} \quad \text{ and } \quad \mathcal{S}_k:= \sigma_{k}^{-1}(\mathcal{S}_{k-1}).
\]
\item Consider the smooth function $h_k : \mathcal{V}_k \to \mathbb{R}$ which is defined recursively by
\[
h_0 := h \quad \text{ and } \quad h_k:=  \rho_{E_k}^{-\alpha_k}  h_{k-1} \circ \sigma_{k},
\]
where $\alpha_k$ is defined in Theorem \ref{THMhironaka} (and is equal to $\mu_{\mathcal{C}_{k}}(h_{k-1})$ - in particular $h_r = \tilde{h}$). Furthermore, consider the analytic set $\Sigma_k = (h_k=0)$ (and note that $\Sigma_r = \tilde{\Sigma}$).
\item We denote by $X^{\ast,k}:= (d\sigma^{k})^{-1}X$ and $Y^{\ast,k}:= (d\sigma^{k})^{-1}Y$. Consider the vector-field $Z_k$ which is defined in $\Sigma_k \setminus E^k$ by:
\[
Z_k := \prod_{i=1}^k \rho_i^{\beta_i} \big[ (X^{\ast,k} \cdot h_k )  Y^{\ast,k} - (Y^{\ast,k} \cdot h_k  ) X^{\ast,k}  \big].
\]
\end{itemize}

\begin{claim}[Induction Claim]\label{cl:step3b}
For each $k \leq r$, the analytic set $\mathcal{S}_k$ is the union of a finite number of two-dimensional smooth manifolds with a codimension two analytic set, the singular distribution $\Delta_k$ is $\mathcal{S}_k$-adapted and at each point $x\in \mathcal{V}_k$, there exists a neighborhood $U_x$ of $x$, two vector fields $X_k$ and $Y_k$ in $\underline{\Delta}_k|_{U_x}$ and a smooth function $\xi_k: U_x \to \mathbb{R}$ which is everywhere non-zero in $U_x$ such that over points in $\left( \Sigma_k \cap U_x\right) \setminus E^k$
\[
Z_k  = \xi_k \, \big[ (X_k \cdot h_k )  Y_k - (Y_k \cdot h_k  ) X_k \big].
\]
\end{claim}

Proposition \ref{PROPresolution} follows from the above claim for $k=r$. Indeed, note that $\mathcal{S}_r = E$ and $h_r = \tilde{h}$. By the above claim, at each point $x \in \tilde{\Sigma}$ there exists a neighborhood $U_x$ of $x$, two analytic vector-fields $\tilde{X}$ and $\tilde{Y}$ over $U_x$ (which are tangent to $E = \mathcal{S}_r$) and a smooth function $\tilde{\xi}: U_x \to \mathbb{R}$ which is everywhere non-zero in $U_x$ such that over points in $(\tilde{\Sigma} \cap U_x) \setminus E$
\[
Z   = \tilde{\xi} \, \big[ (\tilde{X} \cdot \tilde{h} )  \tilde{Y} - (\tilde{Y} \cdot \tilde{h}  ) \tilde{X} \big].
\]
Finally, the right hand side of the above equation is defined over the entire $U_x$ and is tangent to $E$, which proves the Lemma. We now turn to the proof of the Claim.

\begin{proof}[Proof of Claim \ref{cl:step3b}]
The case $k=0$ is a consequence of Claim \ref{cl:step3a}. Now, suppose that the Claim is proved for every $k'<k$ and let us prove that it is also valid for $k$. It is not difficult to see that $\Delta_k$ is $\mathcal{S}_k$ tangent because of the induction hypothesis and the definition of $\Delta_k$ and $\mathcal{S}_k$. Now, fix a point $x\in \mathcal{V}_k$ and consider $y = \sigma(x)$. If the point $y$ does not belong to the center $\mathcal{C}_k$ of the blowing-up $\sigma_k$ or if $\mathcal{C}_k$ has codimension one, then the result is clear from the fact that $\sigma$ is locally a diffeomorphism and $\mathcal{C}_k$ is contained in $\mathcal{S}_{k-1}$. Therefore, we can suppose that $x \in E_k$, $y \in \mathcal{C}_k$ and that $\mathcal{C}_k$ has codimension at least two. There exists a coordinate system $x= (x_1,x_2,x_3)$ (well-defined in an open neighborhood $U_x$ and) centered at $x$ and $y=(y_1,y_2,y_3)$ (well-defined in an open neighborhood $U_y$ and) centered at $y$ such that 
\[
y_1 = x_1, \quad y_2 = x_1 x_2, \quad y_3 = x_1^{\epsilon} x_3,
\]
where $\mathcal{C}_k = (y_1 = y_2 = \epsilon y_3=0)$, $\epsilon = codim(\mathcal{C})-2 \in \{0,1\}$ and $E_k = (x_1=0)$. By induction hypothesis, there exist two analytic vector-fields $X_{k-1}$ and $Y_{k-1}$ in $\Delta_{k-1}|_{U_y}$ and a smooth function $\xi_{k-1}$ such that in $\left( \Sigma_{k-1} \cap U_y\right) \setminus E^{k-1}$
\[
Z_{k-1} =  \xi_{k-1} \, \big[ (X_{k-1} \cdot h_{k-1} )  Y_{k-1} - (Y_{k-1} \cdot h_{k-1}  ) X_{k-1} \big],
\]
and we note that $Z_k$ in $\left( \Sigma_{k} \cap U_x\right) \setminus E^{k}$ is given by:
\begin{equation}\label{eq:step3}
\begin{aligned}
Z_k  &=  \rho_{E_k}^{ \beta_k - \alpha_k}  \, d\sigma_k^{-1}(Z_{k-1}) = \rho_{E_k}^{\beta_k} \,\xi_{k-1}^{\ast}  \, \big[ (X_{k-1}^{\ast} \cdot h_{k} )  Y_{k-1}^{\ast} - (Y_{k-1}^{\ast} \cdot h_{k}  ) X_{k-1}^{\ast} \big],
\end{aligned}
\end{equation}
where $ X_{k-1}^{\ast} = d\sigma_k^{-1}X_{k-1}$ and $ Y_{k-1}^{\ast} = d\sigma_k^{-1}Y_{k-1}$. Before computing the transform $Z_k$ explicitly, we make two remarks:
\begin{itemize}
\item[(i)] We compute the transforms $X_{k-1}^{\ast}$ and $Y_{k-1}^{\ast}$ in terms of $X_{k-1}$ and $Y_{k-1}$:
\begin{equation}
\begin{aligned}
X_{k-1} &= \sum_{i=1}^3 A_i \partial_{y_i} \implies  X_{k-1}^{\ast} = A_1^{\ast} \left(\partial_{x_1} - \frac{x_2}{x_1} \partial_{x_2} - \epsilon\, \frac{x_3}{x_1}\partial_{x_3}   \right)+ \frac{A_2^{\ast}}{x_1}\partial_{x_2}+\frac{A_3^{\ast} }{x_1^{\epsilon}}\partial_{x_3},\\
Y_{k-1} &= \sum_{i=1}^3 B_i \partial_{y_i} \implies  Y_{k-1}^{\ast} = B_1^{\ast} \left(\partial_{x_1} - \frac{x_2}{x_1} \partial_{x_2} - \epsilon\, \frac{x_3}{x_1}\partial_{x_3}   \right)+ \frac{B_2^{\ast}}{x_1}\partial_{x_2}+\frac{B_3^{\ast} }{x_1^{\epsilon}}\partial_{x_3},
\end{aligned}
\label{eq:step3b}
\end{equation}
where $A_i^{\ast}:= A_i \circ \sigma_k$ and $B_i^{\ast}: =B_i \circ \sigma_k$ for $i=1,2 $ and $3$.
\item[(ii)] Given smooth functions $F_1$, $F_2$, $G_1$, $G_2$ and $H$, consider $X = F_1 \tilde{X} +  G_1 \tilde{Y}$ and $Y = F_2 \tilde{X} + G_2 \tilde{Y}$. We note that:
\begin{equation}
\label{eq:step3c}
(X \cdot H )  Y - (Y \cdot H  ) X = (F_1 G_2 - F_2 G_1)\big[ (\tilde{X} \cdot H )  \tilde{Y} - (\tilde{Y} \cdot H  ) \tilde{X} \big].
\end{equation}
\end{itemize}
We divide in three cases depending on the nature of $\mathcal{C}$:

\medskip
\noindent
\emph{Case I:} If $\mathcal{C}_k$ is a point (in which case $\epsilon =1$ and $\beta=2$), then by equation \eqref{eq:step3b} the vector fields $x_k = x_1 X^{\ast}_{k-1}$ and $Y_k = x_1Y^{\ast}_{k-1}$ are well-defined vector-fields which are tangent to $E_k = (x_1=0)$. In particular, $X_k$ and $Y_k$ are local sections of $\Delta_k$. Next, from the definition of $\rho_{E_k}$ (see Lemma \ref{lem:rho}), we know that $\rho_{E_k} = x_1 \cdot \tilde{\xi}_k$ where $\tilde{\xi}_k$ is an everywhere non-zero smooth function. Take $\xi_k:= \tilde{\xi}_k^2 \cdot \xi_{k-1}^{\ast} $ and, by equation \eqref{eq:step3}, we get:
\[
Z_k =  \xi_k \, \big[(X_k \cdot h_k )  Y_k - (Y_k \cdot h_k  )X_k\big],
\]
which proves the Claim in this case.

\medskip
\noindent
\emph{Case II:} If $\mathcal{C}_k$ is a curve contained in a $2$-dimensional locus of $S_{k-1}$ (in which case $\epsilon =0$ and $\beta = 1$). In this case, $(y_i=0) \subset S_{k}$ for $i=1$ or $2$. If $i=1$ (resp. $i=2$), then we can write $B_j = b_j(y_3) + \tilde{B}_j$ where $\tilde{B}_j(0,0,y_3) = 0$ (resp. $A_j = a_j(y_3) + \tilde{A}_j$ where $\tilde{A}_j(0,0,y_3) = 0$). Without loss of generality we can suppose that $b_1$ divides $b_2$ (resp. $a_1$ divides $a_2$). Now consider
\[
\tilde{X} = X_{k-1}, \quad \tilde{Y} = Y_{k-1} - \frac{b_2(y_3)}{b_1(y_3)} X_{k-1}, \quad  (resp. \, \tilde{X} = X_{k-1} , \quad \tilde{Y} = Y_{k-1} - \frac{a_2(y_3)}{a_1(y_3)} X_{k-1}),
\]
which implies that $\tilde{Y}(y_1) \subset (y_1,y_2)$ and $\tilde{Y}(y_2) \subset (y_1,y_2)$. Now by equation \eqref{eq:step3b}, the vector-fields $X_k:= x_1  d\sigma_k^{-1}(\tilde{X})$ and $Y_k := d\sigma_k^{-1}(\tilde{Y})$ belong to $\Delta_k$. Next, from the definition of $\rho_{E_k}$ (see Lemma \ref{lem:rho}), we know that $\rho_{E_k} = x_1 \cdot \tilde{\xi}_k$ where $\tilde{\xi}_k$ is an everywhere non-zero smooth function. Take $\xi_k:= \tilde{\xi}_k \cdot \xi_{k-1}^{\ast} $ and, by equations \eqref{eq:step3} and \eqref{eq:step3c}, we get:
\[
Z_k =  \xi_k \, \big[(X_k \cdot h_k )  Y_k - (Y_k \cdot h_k  )X_k\big],
\]
which proves the Claim in this case. 

\medskip
\noindent
\emph{Case III:} If $\mathcal{C}_k$ is a curve contained in a $1$-dimensional locus of $S_{k-1}$ (in which case $\epsilon =0$, $\beta = 1$ and $(y_1 = y_2=0) \subset S_{k-1}$). In this case, we note that
\[
det \left[\begin{matrix}
X_{k-1}  \cdot y_1 &  X_{k-1} \cdot y_2 \\
Y_{k-1}  \cdot y_1 &  Y_{k-1} \cdot y_2
\end{matrix} \right] \subset (y_1,y_2).
\]
Now, from equation $\eqref{eq:step3b}$ we write $A_j = a_j(y_3) + \tilde{A}_j$ and $B_j = b_j(y_3) + \tilde{B}_j$, where $\tilde{A}_j(0,0,y_3) = \tilde{B}_j(0,0,y_3) = 0$. Without loss of generality we can suppose that $b_1(y_3)$ divides $b_2(y_3)$ and, therefore, that $a_2(y_3) = a_1(y_3) \, \frac{b_2(y_3)}{b_1(y_3)}$. Now consider
\[
\tilde{X} = X_{k-1}, \quad \tilde{Y} = Y_{k-1} - \frac{b_2(y_3)}{b_1(y_3)} X_{k-1},
\]
which implies that $\tilde{Y}(y_1) \subset (y_1,y_2)$ and $\tilde{Y}(y_2) \subset (y_1,y_2)$. Now by equation \eqref{eq:step3b}, the vector-fields $X_k:= x_1  d\sigma_k^{-1}(\tilde{X})$ and $Y_k := d\sigma_k^{-1}(\tilde{Y})$ belong to $\Delta_k$. Next, from the definition of $\rho_{E_k}$ (see Lemma \ref{lem:rho}), we know that $\rho_{E_k} = x_1 \cdot \tilde{\xi}_k$ where $\tilde{\xi}_k$ is an everywhere non-zero smooth function. Take $\xi_k:= \tilde{\xi}_k \cdot \xi_{k-1}^{\ast} $ and, by equations \eqref{eq:step3} and \eqref{eq:step3c}, we get:
\[
Z_k =  \xi_k \, \big[(X_k \cdot h_k )  Y_k - (Y_k \cdot h_k  )X_k\big],
\]
which proves the Claim in this case.

Finally, the above three cases cover all lists of possibilities since $\mathcal{S}_{k-1}$ is the union of a finite number of two-dimensional smooth manifolds with a codimension two analytic set. This finishes the proof.
\end{proof}

\appendix

\section{Characterization of singular curves}

We recall here the characterization of singular horizontal paths which allows to show that in the case of rank-two distritibutions in dimension three the singular paths are those horizontal paths which are contained in the Martinet surface. So, let us consider a totally nonholonomic distribution $\Delta$ of rank $m <n$ on $M$ which  is globally generated by a family of $k$ smooth vector fields $X^1, \ldots, X^k$ on $M$ so that 
\begin{eqnarray*}
\Delta (x) = \mbox{Span} \Bigl\{ X^1(x), \ldots, X^k(x) \Bigr\}  \qquad \forall x \in M. 
\end{eqnarray*}
Thanks to the above parametrization of $\Delta$, for every horizontal path $\gamma : [0,1] \rightarrow M$  there is a control $u\in L^1([0,1],\R^k)$ such that 
\begin{eqnarray}\label{singularL1}
\dot{\gamma}(t) = \sum_{i=1}^k u_i(t) X^i(\gamma(t)) \quad \mbox{for a.e. } t \in [0,1].
\end{eqnarray}
Define $k$ Hamiltonians $h^1, \ldots, h^k :T^*M \rightarrow \R$ by 
$$
h^i := h_{X^i} \qquad \forall i=1, \ldots, k,
$$
that is 
$$
h^i(\psi) = p\cdot X^i(x) \qquad \forall \psi=(x,p)\in T^*M, \, \forall i=1, \ldots, k,
$$
and for every $i=1, \ldots, k$, $\overrightarrow{h}_i$ denote the Hamiltonian vector field on $T^*M$ associated to $h_i$, that is satisfying $\iota_{\overrightarrow{H}}\omega=-dH$, where $\omega$ denotes the canonical symplectic form on $T*M$. In local coordinates on $T^*M$, the Hamiltonian vector field $\overrightarrow{h}_i$ reads
$$
\overrightarrow{h}_i (x,p) = \left( \frac{ \partial h_i}{\partial p} (x,p), -  \frac{ \partial h_i}{\partial x} (x,p) \right).
$$
Singular horizontal controls can be characterized as follows (see \cite[Proposition 1.11]{riffordbook}).  

\begin{proposition}\label{propSING}
An horizontal path $\gamma : [0,1] \rightarrow M$  is singular if and only if there is an absolutely continuous lift $\psi : [0,1] \rightarrow T^*M \setminus \{0\}$ such that
\begin{eqnarray}\label{propsing1}
\dot{\psi}(t) = \sum_{i=1}^k u_i(t) \vec{h}^i \bigl( \psi(t)\bigr)  \quad \mbox{for a.e. } t \in [0,1]
\end{eqnarray}
and
\begin{eqnarray}\label{propsing2}
h^i(\psi(t)) =0, \quad \forall t\in [0,1], \quad \forall i=1,\cdots,k,
\end{eqnarray}
where $u\in L^1([0,1],\R^k)$ is a control satisfying (\ref{singularL1}).
\end{proposition}

Recall that whenever $M$ has dimension three and $\Delta $ rank two, the Martinet surface is defined by
$$
\Sigma := \Bigl\{ x\in M \, \vert \,  \Delta(x) + [\Delta,\Delta](x) \neq T_xM \Bigr\},
$$
In this particular case, Proposition \ref{propSING} implies the following result. For sake of completeness we provide its proof (see \cite[Example 1.17 p. 27]{riffordbook}).

\begin{proposition}\label{PROPsingmartinet}
Let $M$ be a smooth manifold of dimension three and $\Delta$ be a totally nonholonomic distribution of rank two. Then a non-constant horizontal path $\gamma:[0,1] \rightarrow M$ is singular if and only if it is contained in $\Sigma$, that is $\gamma([0,1]) \subset  \Sigma$.
\end{proposition}

\begin{proof}[Proof of Proposition \ref{PROPsingmartinet}]
Let $\gamma:[0,1] \rightarrow M$ be a non-constant singular horizontal path. Argue by contradiction by assuming that there is some $\bar{t}\in [0,1]$ such that $\gamma(\bar{t}) \notin \Sigma$. There is a neighborhood of $\gamma(\bar{t})$ where $\Delta$ is generated by two linearly independent vector fields $X^1, X^2$. Then there are $a<b$ in $[0,1]$ and a control $u\in L^1([a,b],\R^2)$ such that $\dot{\gamma}(t) =  u_1(t) X^1(\gamma(t)) + u_2(t) X^2(\gamma(t))$ for almost every  $t \in [a,b]$.  By Proposition \ref{propSING} there is an absolutely continuous arc $\psi=(\gamma, p) : [a,b] \rightarrow T^*M \setminus \{0\}$ such that (\ref{propsing1}) is satisfied on $[a,b]$ and $h^1(\psi(t)) =h^2(\psi(t))=0$ for all $t\in [a,b]$. Taking the derivative of the latter, we get that for almost every $t$ in $[a,b]$, 
there holds  
\begin{eqnarray}\label{23july99}
0 = \frac{d}{dt}\Bigl\{ h^1(\psi(t))\Bigr\} =  u_2(t) \, h^{1,2}(\psi(t)) \quad \mbox{and} \quad 0 = \frac{d}{dt}\Bigl\{ h^2(\psi(t))\Bigr\} =  u_1(t) \, h^{2,1}(\psi(t)),
\end{eqnarray}
where $h^{i,j}$ is defined by $h^{i,j}(\psi) = p\cdot [X^j,X^i] (x)$ for every $\psi=(x,p)\in T^*M$ and any $i,j =1,2$. Since $\gamma$ is not constant and $\gamma(\bar{t})$ does not belong to $\Sigma$, we may assume in the above equalities that $u_i(t)\neq 0$ for some $i\in \{1,2\}$ and that 
$$
T_{\gamma(t))}M = \mbox{Span} \Bigl\{X^1(\gamma(t)), X^2(\gamma(t)), [X^1,X^2](\gamma(t)) \Bigr\}.
$$
Then, by (\ref{23july99}), we have $h^1(\psi(t))=h^2(\psi(t))=h^{1,2}(\psi(t))=0$ which means that $\psi(t)\neq 0$ three linearly independent tangent vectors, a contradiction.

Let us now prove that any horizontal path which is included in $\Sigma$ is singular. Let $\gamma :[0,1] \rightarrow M$ such a
path be fixed, set $\gamma(0)=x$, and consider a local frame $\{X^1,X^2\}$ for $\Delta$  in a neighborhood
$\mathcal{V}$ of $x$. Let $\delta>0$ be small enough so that $\gamma(t) \in \mathcal{V}$ for any $t\in [0,\delta]$, in such a
way that there is $u\in L^1([0,\delta];\R^2)$ satisfying
$$
\dot{\gamma}(t) =  u_1(t) X^1(\gamma(t)) + u_2(t) X^2(\gamma(t)) \qquad \mbox{a.e. } t\in [0,\delta].
$$
Taking a change of coordinates if necessary, we can assume that we work in $\R^3$. Let $p_0\in (\R^3)^* \setminus \{0\}$ be such that $p_0\cdot X^1(x) =p_0\cdot
X2(x)=0$, and let $p:[0,\delta] \rightarrow (\R^3)^*$ be the solution to the Cauchy problem
$$
\dot{p}(t) = - \sum_{i=1,2} u_i(t) \, p(t)\cdot D_{\gamma(t)} X^i \qquad \mbox{a.e. } t\in [0,\delta], \quad p(0)=p_0.
$$
Define two absolutely continuous function $h_1,h_2 :[0,\delta] \rightarrow \R$ by
$$
h_i(t) = p(t)\cdot X^i(\gamma(t)) \qquad \forall t\in [0,\delta], \quad \forall i=1,2.
$$
As above, for every $t\in [0,\delta]$ we have
$$
\dot{h}_1(t) = \frac{d}{dt} \left[p(t)\cdot X^1(\gamma(t))\right] = -u_2(t) \,p(t)\cdot [X^1,X^2] \bigl(\gamma(t)\bigr)
$$
and
$$
\dot{h}_2 (t) = u_1(t) \, p(t)\cdot \bigl[X^1,X^2\bigr](\gamma(t)).
$$
But since $\gamma(t)\in \Sigma$ for every $t$, there are two continuous functions $\lambda_1,\lambda_2:[0,\delta] \rightarrow \R$ such that
$$
\bigl[X^1,X^2\bigr](\gamma(t)) = \lambda_1(t) X^1(\gamma(t)) + \lambda_2 (t)
X^2(\gamma(t)) \qquad \forall t\in [0,\delta].
$$
This implies that the pair $(h_1,h_2)$ is a solution of the linear differential system
$$
\left\{
\begin{array}{rcl}
\dot{h}_1(t) & = & -u_2(t)\lambda_1(t) h_1(t) -u_2(t)\lambda_2(t) h_2(t) \\
\dot{h}_2(t) & = & u_1(t) \lambda_1(t) h_1(t) +u_1(t) \lambda_2(t) h_2(t).
\end{array}
\right.
$$
Since $h_1(0)=h_2(0)=0$ by construction, we deduce by the
Cauchy-Lipschitz Theorem that $h_1(t)=h_2(t)=0$ for any $t\in [0,\delta]$. In that way, we have constructed an absolutely continuous arc $p:[0,\delta] \rightarrow \bigl(\R^3\bigr)^* \setminus \{0\}$ such that $\psi:=(\gamma,p)$ satisfies (\ref{propsing1})-(\ref{propsing2}) on $[0,\delta]$. We can repeat this construction on a new interval of the form $[\delta,2\delta]$ (with initial condition $p(\delta)$)  and finally obtain an absolutely continuous arc satisfying (\ref{propsing1})-(\ref{propsing2})  on $[0,1]$.  By Proposition \ref{propSING}, we conclude that $\gamma$ is singular.
\end{proof}

\section{Divergence formulas}

Let $M$ be a smooth manifold of dimension $n\geq 2$ and $\omega$ be a volume form on $M$. Given a smooth vector field $Z$ on $M$, we recall that its divergence with respect to $\omega$ is defined by $\left( \mbox{div}^{\omega} Z \right)\, \omega = L_Z \omega$, where $L_Z$ denotes the Lie derivative along $Z$. In other words, if we denote by $\varphi_t$ the flow of $Z$, then for every open set $\Omega$ with finite volume, we have
\begin{eqnarray}\label{formuladiv}
\frac{d}{dt} \Bigl\{ \mbox{vol}^{\omega} \left( \varphi_t(\Omega) \right) \Bigr\}_{\vert t=0}= \int_{\Omega} \mbox{div}^{\omega}_x Z \, \omega(x),
\end{eqnarray}
where $\mbox{vol}^{\omega}$ is the measure associated with $\omega$ (given by $\mbox{vol}^{\omega}=\int \, \omega$). Therefore, if $\alpha: M \rightarrow (0,+\infty)$ and $\beta:M \rightarrow \R$ are smooth functions then we check easily that
\begin{eqnarray}\label{9aout1}
\mbox{div}^{\alpha \omega} ( Z ) = \mbox{div}^{ \omega} \left(  Z \right) + \frac{1}{\alpha} \, \left( Z\cdot \alpha\right) \quad \mbox{and} \quad \mbox{div}^{\omega} ( \beta \, Z ) = \beta \mbox{div}^{ \omega} (  Z ) +  Z\cdot \beta,
\end{eqnarray}
where we use the notation $Z \cdot \alpha= L_Z \alpha$.  One of the main tool of the present paper is the following formula.

\begin{proposition}\label{PROPintdiv}
Let $S$ be a Borel set with finite volume such that $\varphi_t$ is well-defined on $S$ for every $t\geq 0$, then there holds 
\begin{eqnarray}\label{volt}
\mbox{vol}^{\omega}  \left( \varphi_t(S) \right) = \int_S \exp \left(   \int_0^t \mbox{div}_{\varphi_s(x)}^{\omega} Z \, ds    \right) \omega(x).
\end{eqnarray}
\end{proposition}

\begin{proof}[Proof of Proposition \ref{PROPintdiv}]
Let us first assume that $S$ is an open subset of $M$. Then one the one hand thanks to (\ref{formuladiv}) we have for every $t\geq 0$
\begin{multline*}
\frac{d}{dt} \Bigl\{ \mbox{vol}^{\omega} \left( \varphi_t(S) \right) \Bigr\} = \frac{d}{ds} \Bigl\{ \mbox{vol}^{\omega} \left( \varphi_{s}(\varphi_t(S)) \right) \Bigr\}_{\vert s=0} = \int_{\varphi_t(S)} \mbox{div}^{\omega}_y Z \, \omega(y) \\
= \int_{S} \mbox{div}^{\omega}_{\varphi_t(x)} Z \, \left(\varphi_t^*\omega\right)(x) =  \int_{S} \mbox{div}^{\omega}_{\varphi_t(x)} Z \, f_t(x) \, \omega (x),
\end{multline*}
where we did the change of variable $y=\varphi_t(x)$. On the other hand, doing the same change of variable yields
\begin{eqnarray}\label{formulaproofdiv}
 \mbox{vol}^{\omega} \left( \varphi_t(S) \right) =  \int_{\varphi_t(S)} 1 \, \omega(y) =  \int_{S}  f_t(x) \, \omega (x).
\end{eqnarray}
which implies by differentiation under the integral that
$$
\frac{d}{dt} \Bigl\{ \mbox{vol}^{\omega} \left( \varphi_t(S) \right) \Bigr\}  =  \int_{S}  \frac{\partial f_t}{\partial t}(x) \, \omega (x).
$$
Since the above formulas hold for any open sets, we infer that for every $x\in M$, $ \frac{\partial f_t}{\partial t}(x)=  \mbox{div}^{\omega}_{\varphi_t(x)} Z\, f_t(x)$ with $f_0(x)=1$. We conclude easily. 
\end{proof}

If $g$ is  a smooth Riemannian metric on $M$ then we recall that the volume form $\mbox{vol}^g$ associated with $g$ is defined in a set of positively oriented local coordinates by 
$$
\mbox{vol}^g =\sqrt{\det (g_{ij})} \, dx_1 \wedge \cdots \wedge dx_n,
$$
where $g$ reads $g=\sum_{i,j=1}^n g_{ij} dx_i \, dx_j$. In the sequel, the divergence operator $\mbox{div}^{vol^g}$ associated with $\mbox{vol}^g$ will be denoted by $\mbox{div}^g$. Given a  smooth submanifold $\Sigma$ of $M$, we can equip it with the Riemannian metric induced by $g$, in the sequel we denote by $\mbox{vol}^{\Sigma}$ and $\mbox{div}^{\Sigma}$ the corresponding volume and divergence operator on $\Sigma$. The following result plays also a major role in the present paper.

\begin{proposition}\label{PROPdivcodim1}
Let $\Sigma$ be a smooth hypersurface of $M$ and  $U$ be an open subset of $M$ such that 
$$
\Sigma \cap U := \Bigl\{ x \in U \, \vert \, h(x) = 0, \, d_xh \neq 0 \Bigr\},
$$
for some smooth function $h:U \rightarrow \R$. Then there holds $\mbox{vol}^{\Sigma}=i_N (\mbox{vol}^g)$ on $\Sigma\cap U$, where $N:= \nabla h/|\nabla h|$ and $i_N$ denotes the interior product with $N$. Moreover, for every smooth vector field $Z$ on $U$ satisfying $Z\cdot h=0$ on $U$, the restriction $Z$ of $Z$  to $\Sigma \cap U$ is tangent to $\Sigma \cap U$ and it satisfies
\begin{eqnarray}\label{24july1}
\mbox{div}_x^{\Sigma} Z =  \mbox{div}_x^{g} Z  + \frac{\left( Z\cdot |\nabla h|^2\right) (x)}{2|\nabla_x h|^2} \qquad \forall x \in \Sigma \cap U.
\end{eqnarray}
\end{proposition}

\begin{proof}[Proof of Proposition \ref{PROPdivcodim1}]
The first part follows easily from the fact that the restriction of $\nabla h$ to $\Sigma$ is orthogonal to $\Sigma$. To prove (\ref{24july1}), we can notice that the divergences $\mbox{div}^g$ and $\mbox{div}^{\Sigma}$ can be expressed in term of the corresponding Levi-Civita connections $\nabla$ and $\nabla^{\Sigma}$. If $(e_1, \ldots, e_{n-1})$ is an orthonormal basis of $T_{x}\Sigma$ at $x\in \Sigma \cap U$, then the family $(e_1, \ldots, e_{n})$ with $e_n:= N(x) = \nabla_xh/ | \nabla_x h|$ is an orthonormal basis of $T_xM$ and we have 
$$
\mbox{div}^{\Sigma}_x Z = \sum_{i=1}^{n-1} \nabla_{e_i}^{\Sigma} Z(x) \cdot e_i  \quad \mbox{and} \quad \mbox{div}^{g}_x Z = \sum_{i=1}^{n} \nabla_{e_i} Z(x) \cdot e_i. 
$$
Since the connection $\nabla^{\Sigma}$ is given by the orthogonal projection of $\nabla$ on $T_x\Sigma$, we infer that
$$
\mbox{div}^{\Sigma}_x Z =  \mbox{div}^{g}_x Z -  \nabla_{e_n} Z(x) \cdot e_n. 
$$
The fact that $Z\cdot h=0$ on $U$ yields $\nabla_{e_n} Z(x) \cdot \nabla_xh + Z(x) \cdot \nabla_{e_n} (\nabla h)(x)=0$. We conclude easily.
\end{proof}

\section{The Martinet analytic space}\label{secMartinetspace}

In this section we show that the Martinet set $\Sigma$ has the structure of an analytic space, denoted by $\Sigma_{\Delta}$, whose singular set has codimension two. We recall that an analytic space is a pair  $(X,\mathcal{O}_X:=\mathcal{O}_M / \mathcal{I})$ where $X $ is an analytic set in $M$, $\mathcal{O}_M$ denotes the sheaf of analytic functions over $M$ and $\mathcal{I}$ is a coherent ideal sheaf (of $\mathcal{O}_M$) with support $X$ (see, e.g. \cite[Definition 1.6]{h73} or \cite{t71}). When $\mathcal{I}$ is a principal ideal sheaf, the singular set of $(X, \mathcal{O}_X)$ is the set of points where $\mathcal{I}$ has order two \cite[Definition 5.6]{h73} (see subsection \ref{SSECmanifoldcorners}).

Given a vector bundle $E$, we denote by $\underline{E}$ the sheaf of analytic sections of $E$. Following \cite[p. 281]{baumbott}, a distribution $\Delta$ is analytic if $\underline{\Delta}$ is a coherent subsheaf of $Der_M$, where $Der_M$ is the sheaf of all derivations over $M$. In this setting, we define a coherent subsheaf $[\underline{\Delta}, \underline{\Delta}]$ of $Der_M$, where the stalks of $[\underline{\Delta}, \underline{\Delta}]$ at a point $x\in M$ are generated by the set:
$$
\Bigl\{ [X,Y] \, \vert \, X,Y \in  \underline{\Delta}_x \Bigr\}
$$

\begin{lemma}\label{LEMAS1}
Suppose that $\Delta$ is a rank-two analytic distribution. The sheaf $[\underline{\Delta}, \underline{\Delta}]$ is coherent, contains $\underline{\Delta}$ and is everywhere locally generated by three derivations. More precisely, given a point $x\in M$ and two generators $X$ and $Y$ of $\underline{\Delta}_x$, the derivations $X$, $Y$ and $[X,Y]$ are generators of $[\underline{\Delta}, \underline{\Delta}]_x$.  
\label{lem:TechnicalAlgebraicProp1}
\end{lemma}

\begin{proof}[Proof of Lemma \ref{LEMAS1}]
Fix a point $x\in M$ and an open neighborhood $\mathcal{V}$ of $x$. Since $\underline{\Delta}$ is coherent, apart from shrinking $\mathcal{V}$, there two regular analytic vector-fields $X$ and $Y$ defined on $\mathcal{V}$ which locally generate the restriction of $\underline{\Delta}$ to $\mathcal{V}$ (which we denote by $\underline{\Delta} \cdot \mathcal{O}_{\mathcal{V}}$). Since $X$ and $Y$ are regular, there exists two analytic functions $h_X$ and $h_Y$ defined in $\mathcal{V}$ such that $X\cdot h_X$ and $Y \cdot h_Y$ are units, that is $X\cdot h_X(y) \neq 0$ and $Y \cdot h_Y(y)\neq 0$ for all $y\in \mathcal{V}$. Note that:
\[
[X,h_X Y] = \left(X\cdot h_X\right) Y + h_X [X,Y] \quad \text{ and } \quad [ Y, h_YX] = \left(Y\cdot h_Y\right) X - h_Y [X,Y]
\]
belong to $[\underline{\Delta}, \underline{\Delta}] \cdot \mathcal{O}_{\mathcal{V}}$, which implies that $X$, $Y$ and $[X,Y]$ belongs to $[\underline{\Delta}, \underline{\Delta}] \cdot \mathcal{O}_{\mathcal{V}}$. Since all derivations in $\underline{\Delta} \cdot \mathcal{O}_{\mathcal{V}}$ are $\mathcal{O}_{\mathcal{V}}$-linear combinations of $X$ and $Y$, it is clear that these three derivations generate $[\underline{\Delta}, \underline{\Delta}]\cdot \mathcal{O}_{\mathcal{V}}$ (in particular, we can already conclude that $\underline{\Delta} \subset [\underline{\Delta},\underline{\Delta}]$). This implies that $[\underline{\Delta}, \underline{\Delta}] \cdot \mathcal{O}_{\mathcal{V}}$ is of finite type and, by Oka's Theorem \cite[Chapitre 2, Corollaire 6.9]{Tougeron}, it is coherent. Since the choice of $x$ was arbitrary, the proof is complete. 
\end{proof}

We now consider a first (non-reduced) generator of the Martinet surface $\Sigma$. Let us define the (auxiliary) ideal sheaf $\mathcal{J}_{\Delta}$ whose stalk at a point $x\in M$, is generated by the following set
\[
\left\{ det \left[\begin{matrix}

X\cdot f & X \cdot g & X \cdot h\\
Y \cdot f  & Y \cdot g & Y \cdot h\\
Z \cdot f  & Z \cdot g & Z \cdot h\\

\end{matrix} \right]   ; \, X,\,Y,\,Z\in  [\underline{\Delta}, \underline{\Delta}]_x \text{ and } f,\, g, \, h \in \mathcal{O}_{x} \right\}
\]
where $\mathcal{O}_{x}$ is the localization of $\mathcal{O}_M$ at the point $x$. 

\begin{lemma}
The ideal sheaf $\mathcal{J}_{\Delta}$ is principal and coherent. Furthermore its support is equal to the Martinet set $\Sigma$, that is the set of points which vanish over all functions in $\mathcal{J}_{\Delta}$ is equal to $\Sigma$.
\label{LEMAS15}
\end{lemma}
\begin{proof}[Proof of Lemma \ref{LEMAS15}]
Following the same kind of argument given in Lemma \ref{LEMAS1}, it is not difficult to show that $\mathcal{J}_{\Delta}$ is a principal and coherent ideal sheaf. Next, note that a point $x$ is in the support of $\mathcal{J}_{\Delta}$ if and only if for all locally defined derivations $X$, $Y$ and $Z$ in $[\underline{\Delta},\underline{\Delta}]_x$, the linear space generated by $\{X(x),Y(x),Z(x)\}$ is different from $T_xM$. We easily conclude from Lemma \ref{LEMAS1}.
\end{proof}

Finally, we consider the ideal sheaf $\mathcal{I}_{\Delta} := \sqrt{\mathcal{J}_{\Delta}}$. 

\begin{lemma}\label{LEMAS2}
The ideal sheaf $\mathcal{I}_{\Delta}$ is principal, coherent and its support is the Martinet surface $\Sigma$. Furthermore, the singular set of $\mathcal{I}_{\Delta}$ ({\it i.e.} the set of points where $\mathcal{I}$ has order at least two) has codimension two.
\end{lemma}
\begin{proof}[Proof of Lemma \ref{LEMAS2}]
Since $\mathcal{J}_{\Delta}$ is principal and $\mathcal{O}_x$ unique factorization domain, it follows that $\mathcal{I}_{\Delta}$ is principal. Next, $\mathcal{I}_{\Delta}$ is coherent by \cite[Chapitre 2 Corollaire 7.4]{Tougeron} and its support is clearly equal to the support $\mathcal{J}_{\Delta}$ which is the Martinet surface $\Sigma$. Finally, from the fact that $\sqrt{\mathcal{I}_{\Delta}} = \mathcal{I}_{\Delta}$ and that $\mathcal{I}_{\Delta}$ is coherent, it is clear that its singular set has codimension two.
\end{proof}

Finally, Lemma \ref{LEMAS2} implies that $\Sigma_{\Delta}:= (\Sigma, \mathcal{O}_M/ \mathcal{I}_{\Delta})$ is a coherent analytic space whose singular locus has codimension two. 

\begin{remark}\label{rk:MartinetSpace}
The ideal sheaf $\mathcal{I}_{\Delta}$ is not necessarily maximal in respect to the sub-category of coherent ideal sheaves whose support is the Martinet set $\Sigma$. We could replace $\mathcal{I}_{\Delta}$ by such a maximal coherent ideal sheaf, provided that it exists. Because of this remark, if $\Sigma$ is a coherent analytic set (see \cite[p. 93, Definition 2]{n66}) then we can assume that $\mbox{Sing}(\Sigma_{\Delta}) = \mbox{Sing}(\Sigma)$.
\end{remark}

\begin{remark}
The Martinet surface $\Sigma$ is not necessarily a coherent analytic set (see \cite[p. 93, Definition 2]{n66}) but it is a $\mathbb{C}$-analytic set (see \cite[p. 104, Definition 6]{n66}).
\end{remark}

\section{Computations}\label{appendixcomputations}
Here we provide the proof of the estimates (\ref{18july1})-(\ref{18july4}). We have 
\begin{multline*}
 z(+\infty)-z(-\infty) =  \frac{4}{3} \int_{-\infty}^{+\infty}  x(t)y(t)^4 \, dt =  \frac{8}{3} \int_{-\infty}^{0}  x(t)y(t)^4 \, dt \\
=   -  \frac{4}{3} \int_{-\infty}^{0} \dot{x}(t) \, y(t)^3 \, dt=    -  \frac{4}{3} \int_{-\infty}^{0} |\dot{x}(t)| \, |y(t)|^3 \, dt  \\
     =      - \frac{4}{3} \int_{-\infty}^{0} |\dot{x}(t)| \, |x(t)|^3 \left|x(t)+z(t)\right|^{\frac{3}{2}} \, dt \\
     \leq  - \frac{4}{3} \, \int_{-\infty}^{-T} |\dot{x}(t)|  \, |x(t)|^3  \left|x(t)-\bar{x}\right|^{\frac{3}{2}} \, dt \\
    = -\frac{4}{3} \int_{0}^{-\bar{x}} r^3 (-r-\bar{x})^{3/2} \, dr   \leq - \frac{1}{35} |-\bar{x}|^{11/2} \leq -\frac{1}{35} z(+\infty)^{11/2},
\end{multline*}
\begin{multline*}
 z(+\infty)-z(-\infty) =  \frac{4}{3} \int_{-\infty}^{+\infty}  x(t)y(t)^4 \, dt \\
=   -  \frac{2}{3} \int_{-\infty}^{+\infty} \dot{x}(t) \, y(t)^3 \, dt=    -  \frac{2}{3} \int_{-\infty}^{+\infty} |\dot{x}(t)| \, |y(t)|^3 \, dt  \\
     =      - \frac{2}{3} \int_{-\infty}^{+\infty} |\dot{x}(t)| \, |x(t)|^3 \left|x(t)+z(t)\right|^{\frac{3}{2}} \, dt \\
     \geq  - \frac{2}{3}  \left|\bar{x} \right|^3 \, \left| z(-\infty)\right|^{3/2} \int_{-\infty}^{+\infty} |\dot{x}(t)|  \, dt   \geq - \frac{2}{3} |z(-\infty)|^{9/2} \, \ell \left( \hat{P} \right)
\end{multline*}
\begin{multline*}
z(-\infty) = -\bar{x} - \frac{4}{3} \int_{-\infty}^{0}  x(t)y(t)^4 \, dt =  -\bar{x} +  \frac{2}{3} \int_{-\infty}^{0} \dot{x}(t) \, y(t)^3 \, dt \\
=   -\bar{x} +  \frac{2}{3} \int_{-\infty}^{0} |\dot{x}(t)| \, |y(t)|^3 \, dt  \\
     =   |\bar{x} |  + \frac{2}{3} |\bar{x}|^{9/2} \, \int_{-\infty}^{0} |\dot{x}(t)|  \, dt \leq \ell (\hat{P}) \left[ \frac{1}{2} + \frac{|z(-\infty)|^{9/2}}{3} \right],
\end{multline*}
\begin{eqnarray*}
\ell (\hat{P}) \leq \ell (P) :=\int_{-\infty}^{+\infty} \left| \dot{P}(t) \right| \, dt,
\end{eqnarray*}
and
\begin{multline*}
\ell (P)  \leq \sqrt{2} \, \ell (\hat{P}) + \int_{-\infty}^{+\infty} \left| \dot{z}(t) \right| \, dt \leq \ell (\hat{P}) \, \left[ \sqrt{2} +   \frac{2}{3} |z(-\infty)|^{9/2} \right] \\
\leq \left[2 \left| \bar{x} \right| + 4  \left| \bar{x} \right|^{3/2}\right] \, \left[ \sqrt{2} +   \frac{2}{3} |z(-\infty)|^{9/2} \right] \\
\leq 2 \, \left[ \left| z(-\infty) \right| + 2  \left| z(-\infty) \right|^{3/2}\right]  \, \left[ \sqrt{2} +   \frac{2}{3} |z(-\infty)|^{9/2} \right].
\end{multline*}

\end{document}